\documentclass[10pt,journal,compsoc]{IEEEtran}
\usepackage{natbib}
\bibpunct{[}{]}{,}{n}{}{}
\usepackage{times}
\usepackage{xcolor}
\usepackage{soul}
\usepackage[utf8]{inputenc}
\usepackage[small]{caption}

\usepackage{times}
\usepackage{amsmath}
\usepackage{amssymb}
\usepackage{url}
\usepackage{amsthm} % proof environment
\usepackage{algorithm}% algorithm environment
\usepackage{algorithmic}
\usepackage{diagbox} % diagbox
\usepackage{bbm}
\usepackage{bm}
\usepackage{diagbox}
\usepackage{graphicx}
\usepackage{subcaption}
\usepackage{epstopdf}
\usepackage{booktabs}
\usepackage{pdfpages}
\usepackage{enumitem}

\usepackage{courier}
\usepackage{color}
\definecolor{colorpink}{RGB}{251,53,155}
\definecolor{colorblue}{RGB}{0,148,200}
\definecolor{colorgreen}{RGB}{0,150,0}

\def\beq{\begin{eqnarray}}
\def\eeq{\end{eqnarray}}
\def\noi{\noindent}
\def\nn{\nonumber}
\def\la{\langle}
\def\ra{\rangle}
\def\vec{\text{vec}}
\def\mat{\text{mat}}

\newtheorem{lemma}{Lemma}
\newtheorem{theorem}{Theorem}

\newtheorem{proposition}{Proposition}
\newtheorem{assumption}{Assumption}

\newcommand{\bbb}[1]{\boldsymbol{\mathbf{#1}}}
\def\S{\bar{\mathcal{S}}}
\def\Z{\bar{\mathcal{Z}}}
\def\N{\mathcal{N}}

\usepackage{multirow}
\usepackage{rotating}
\def\prox{\mathrm{prox}}

\def\hone{62pt}
\def\htwo{62pt}

\usepackage[pagebackref=true,breaklinks=true,letterpaper=true,colorlinks,bookmarks=false]{hyperref}
\ifCLASSINFOpdf
\else
\fi
\begin{document}
\title{A Coordinate-wise Optimization Algorithm for Sparse Inverse Covariance Selection}

\author{Ganzhao Yuan,~Haoxian Tan,~Wei-Shi Zheng% <-this % stops a space
\IEEEcompsocitemizethanks{
\IEEEcompsocthanksitem Ganzhao Yuan is with School of Data and Computer Science, Sun Yat-sen University (SYSU), China. E-mail: yuanganzhao@gmail.com \protect\\

\IEEEcompsocthanksitem Haoxian Tan is with School of Data and Computer Science, Sun Yat-sen University (SYSU), China. E-mail: tanhx3@mail2.sysu.edu.cn \protect\\

\IEEEcompsocthanksitem Wei-Shi Zheng is with School of Data and Computer Science, Sun Yat-sen University (SYSU), China.E-mail: zhwshi@mail.sysu.edu.cn \protect\\

%\IEEEcompsocthanksitem J. Doe and J. Doe are with Anonymous University.
}% <-this % stops a space
}


\markboth{}%Journal of \LaTeX\ Class Files,~Vol.~14, No.~8, August~2015
{Shell \MakeLowercase{\textit{et al.}}: Bare Advanced Demo of IEEEtran.cls for IEEE Computer Society Journals}



\IEEEtitleabstractindextext{%
\begin{abstract}
%\begin{abstract}
Sparse inverse covariance selection is a fundamental problem for analyzing dependencies in high dimensional data. However, such a problem is difficult to solve since it is NP-hard. Existing solutions are primarily based on convex approximation and iterative hard thresholding, which only lead to sub-optimal solutions. In this work, we propose a coordinate-wise optimization algorithm to solve this problem which is guaranteed to converge to a coordinate-wise minimum point. The algorithm iteratively and greedily selects one variable or swaps two variables to identify the support set, and then solves a reduced convex optimization problem over the support set to achieve the greatest descent. As a side contribution of this paper, we propose a Newton-like algorithm to solve the reduced convex sub-problem, which is proven to always converge to the optimal solution with global linear convergence rate and local quadratic convergence rate. Finally, we demonstrate the efficacy of our method on synthetic data and real-world data sets. As a result, the proposed method consistently outperforms existing solutions in terms of accuracy.

%\end{abstract}

\end{abstract}

}

\maketitle

% is not selected <OR> if conference mode is selected - because compsoc
% conference papers position the abstract like regular (non-compsoc)
\IEEEdisplaynontitleabstractindextext

% \fi
\IEEEpeerreviewmaketitle

\section{Introduction}
In this paper, we mainly focus on the following nonconvex optimization problem:
\beq \label{eq:main}
\begin{split}
 \min_{\bbb{X} \in \mathbb{R}^{n\times n}}~f(\bbb{X})\triangleq  \langle \bbb{\Sigma},\bbb{X} \rangle  -\log \det(\bbb{X}),\\
s.t.~\bbb{X} \succ \bbb{0},~ \| \bbb{X} \|_{0,\text{off}} \leq s,~~~~~~~~~
\end{split}
\eeq
\noi where $\bbb{\Sigma} \in \mathbb{R}^{n\times n}$ is a given symmetric covariance matrix of the input data set, $\|\cdot  \|_{0,\text{off}} $ counts the number of non-diagonal and non-zero elements of a square matrix, and $s$ is a positive integer that specifies the sparsity level of the solution. $\bbb{X} \succ 0$ means $\bbb{X}$ is positive definite. $\la \cdot , \cdot \ra$ stands for the standard inner product.

The optimization problem in (\ref{eq:main}) is known as sparse inverse covariance selection in the literature \cite{friedman2008,jordan1998learning}. It provides a good way of analyzing dependencies in high dimensional data and captures varieties of applications in computer vision and machine learning (e.g. biomedical image analysis \cite{honorio2009sparse}, scene labeling \cite{SoulyS16}, brain
functional network classification \cite{zhou2014discriminative}). The log-determinant function is introduced for maximum likelihood estimation, and the $\ell_0$ norm is used to reduce over-fitting and improve the interpretability of the model. We remark that when the sparsity constraint is absent, one can set the gradient of the objective function $f(\cdot)$ to zero (i.e. $\bbb{\Sigma}-\bbb{X}^{-1}=\bbb{0}$) and output $\bbb{\Sigma}^{-1}$ as the optimal solution.

Problem (\ref{eq:main}) is very challenging due to the introduction of the combinatorial $\ell_0$ norm. Existing solutions can be categorized into two classes: convex $\ell_1$ approximation and iterative hard thresholding. Convex $\ell_1$ approximation simply replaces the $\ell_0$ norm by its tightest convex relaxation $\ell_1$ norm. In the past decades, a plethora of approaches have been proposed to solve the $\ell_1$ norm approximation problem, which include projected sub-gradient method \cite{DuchiGK08}, (linearized) alternating direction method \cite{ScheinbergMG10,yuan2012}, quadratic approximation method \cite{oztoprak2012,OlsenoNR12,HsiehSDR14,HsiehDRB12}, block coordinate descent method \cite{friedman2008,BanerjeeGd08}, Nesterov's first-order optimal method \cite{lu2009smooth,lu2010adaptive,dAspremontBG08}, primal-dual interior point method \cite{Li2010}. Despite the popularity of convex methods, they fail to control the sparsity of the solution and often lead to sub-optimal accuracy for the original non-convex problem. Recent attention has been paid to solving the original non-convex problem directly by the researchers \cite{YuanG15,zhang2010analysis,candes2008enhancing}.
%using non-convex methods directly to

Iterative hard thresholding method considers iteratively setting the small elements (in magnitude) to zero in a gradient descent manner. By using this strategy, it is able to control the sparsity of the solution directly and exactly. Due to its simplicity, it has been widely used and incorporated into the optimization framework of penalty decomposition algorithm \cite{lu2013} and mean doubly alternating direction method \cite{dong2013efficient}. In \cite{lu2013}, it is shown that for the general sparse optimization problem, any accumulation point of the sequence generated by the penalty decomposition algorithm always satisfies the first-order optimality condition of the problem.

Recently, A. Beck and Y. Vaisbourd present and analyze a new optimality criterion which is based on coordinate-wise optimality \cite{beck2016}. They show that coordinate-wise optimality is strictly stronger than the optimality criterion based on hard thresholding. They apply their algorithm to principal component analysis and show that their method consistently outperforms the well-known truncated power method \cite{yuan2013truncated}. Inspired by this work, we extend their method to solve sparse inverse covariance selection problem. We are also aware of the work \cite{marjanovic2015} where a cyclic coordinate descent decent algorithm (combined with a randomized initialization strategy) is considered to solve the sparse inverse covariance selection problem. However, their method only addresses the $\ell_0$ norm regularized optimization problem and it fails to control the sparsity level of the solution.

%Coordinate-wise optimization algorithms perform at each iteration
%an optimization step with respect to one or two variables, while keeping
%the rest fixed. The coordinates that need to be altered are chosen to be the ones
%that produce the maximal decrease among all possible alternatives, or by applying
%an index selection strategy based on a local first order information.

%\vspace{-12pt}
\bbb{Contributions:} The contributions of this work are three-fold. (i) We propose a new coordinate-wise optimization algorithm for sparse inverse covariance selection. The algorithm iteratively and greedily selects one variable or swap two variables to identify the support set, and then solves a reduced convex optimization problem over the support set (See Section \ref{algo}). (ii) An efficient Hessian-free Newton-like algorithm to solve the convex subproblem is proposed (Section \ref{re-optimizer}). (iii) We provide some theoretical analysis for the proposed Coordinate-Wise Optimization Algorithm (CWOA) and the Newton-Like Optimization Algorithm (NLOA). We prove that CWOA is guaranteed to converge to a coordinate-wise minimum point of the original nonconvex problem and the NLOA is guaranteed to converge to the global optimal solution of the convex subproblem with global linear rate and local quadratic convergence rate (Section \ref{eq:theana}). (iv) Extensive experiments have shown that our method \emph{consistently} outperforms existing solutions in terms of accuracy (Section \ref{expr}).

%. In addition, we also prove that our Newton-like algorithm converges to the global optimal solution 

%achieves state-of-the-art performance.

%to estimate the value of new non-zero element and compute the change in objective value. Throughout
%\vspace{-12pt}
\bbb{Notations:} In this paper, boldfaced lowercase letters denote vectors and uppercase letters denote real-valued matrices. We denote $\bbb{\lambda}(\bbb{X}) \in \mathbb{R}^n$ as the eigenvalues of $\bbb{X}$ in increasing order. All vectors are column vectors and superscript $T$ denotes transpose. $\vec(\bbb{X}) \in \mathbb{R}^{n^2 \times 1}$ stacks the columns of the matrix $\bbb{X} \in \mathbb{R}^{n\times n}$ into a column vector and $\mat(\bbb{x}) \in \mathbb{R}^{n\times n}$ converts $\bbb{x}\in\mathbb{R}^{n^2 \times 1}$ into a matrix. Thus, $\vec(\mat(\bbb{x}))=\bbb{x}$ and $\mat(\vec(\bbb{X}))=\bbb{X}$. We use $\la \bbb{X},\bbb{Y}\ra$ and $\bbb{X}\otimes\bbb{Y}$ to denote the Euclidean inner product and Kronecker product of $\bbb{X}$ and $\bbb{Y}$, respectively. For any matrix $\bbb{X}\in \mathbb{R}^{n\times n}$ and any $i,j\in \{1,2,...,n\}$, we denote by $\bbb{X}_{ij}$ the element of $\bbb{X}$ in $i^{\text{th}}$ row and $j^{\text{th}}$ column and use $\bbb{X}_k$ to denote the $k$ position of $\vec(\bbb{X})$. Therefore, we have $\bbb{X}_{ij}=\bbb{X}_{(j-1)\times n+i}$. We denote $\bbb{e}_i$ as a unit vector with a 1 in the $i^{\text{th}}$ entry and 0 in all other entries. We use $j\in\{1,2,...,n^2\}$ to denote any position in a square matrix of size $n\times n$ where $n$ is known from the context, and use $row(j)$ and $col(j)$ to denote the corresponding row and column for $j$. We denote $\bbb{E}_j$ is a square symmetric matrix with the entries $(row(j),col(j))$ and $(col(j),row(j))$ equal 1 and 0 in all other ones. Note that when $row(j)\neq col(j)$, we have $\bbb{E}_j=\bbb{e}_{row(j)}\bbb{e}_{col(j)}^T + \bbb{e}_{col(j)}\bbb{e}_{row(j)}^T$. Finally, for any matrices $\bbb{X}\in \mathbb{R}^{n\times n}$ and $\bbb{H}\in \mathbb{R}^{n^2\times n^2}$, we define $\|\bbb{D}\|_{\bbb{H}}^2 \triangleq \vec(\bbb{D})^T \bbb{H} \vec(\bbb{D})$ and $\bbb{H}\circ \bbb{X} \triangleq \mat( \bbb{H}\vec(\bbb{X})) \in \mathbb{R}^{n\times n}$.

%Organization and
%The rest of the paper is organized as follows. Section \ref{algo} presents our Coordinate-Wise Optimization Algorithm (CWOA) and its theoretical analysis. Section \ref{re-optimizer} discusses a convex optimization algorithm to solve the convex sub-problem. Section \ref{expr} contains a thorough set of experiments. Section \ref{conclude} concludes the paper.

\section{ Coordinate-wise Optimization Algorithm}\label{algo}

This section presents our coordinate-wise optimization algorithm which is guaranteed to converge after a finite amount of iterations to a coordinate-wise minimum point \cite{beck2013sparsity,beck2016}. We denote $\S(\bbb{X})$ and $\Z(\bbb{X})$ as the index of \textit{non-diagonally} non-zero elements and zero elements of $\bbb{X}$, respectively.

First of all, we notice that when the support set $S$ is known, problem (\ref{eq:main}) reduces to the following convex optimization problem:
\beq \label{eq:subproblem}
\min_{\bbb{X}\succ 0}~ f(\bbb{X}),~s.t.~\bbb{X}_{Z} = \bbb{0},~\text{with}~Z\triangleq \{1,2,...,n^2\} \setminus S.
\eeq
Our algorithm iteratively and greedily selects one variable or swaps two variables to identify the support set $S$, and then solves a reduced convex sub-problem in (\ref{eq:subproblem}) to achieve the greatest descent.

%\begin{algorithm}[t]
%\caption{\label{algo:pen} {\bf A simple penalty method for Eq (\ref{pca1}) and Eq (\ref{pca2})}
%\\(with objective $J(\mathbf{L})$ in Eq (\ref{pca1}) or Eq (\ref{pca2}), s.t. $\textbf{L}^T\textbf{L}=\textbf{I}$)}
%\begin{algorithmic}[1]
%  \STATE Initialize the scalar penalty parameters $\beta(\beta>0)$, $\sigma(\sigma>1)$ and stopping parameter $\varepsilon$, set $t=0$.
%  \WHILE{not converge}
%  \STATE solve $J(\textbf{L}_t)+\beta\|\textbf{L}_t^T\textbf{L}_t-\textbf{I}\|_{\text{F}}^2$
%      \IF{$\|\textbf{L}_t^T\textbf{L}_t-\textbf{I}\|_{\text{F}}^2<\varepsilon$}
%      \STATE Return $\mathbf{L}_t$
%      \ELSE
%      \STATE $\beta=\beta \times \sigma$
%      \ENDIF
%    \STATE Increment $t$ by 1
%  \ENDWHILE
%\end{algorithmic}
%\end{algorithm}

\begin{algorithm}[!tb]
\caption{ {\bf CWOA: A Coordinate-wise Optimization Algorithm for Sparse Inverse Covariance Selection.}}
\label{alg:main}
	\begin{algorithmic}
	\STATE Input: Sparsity level $s$.
	\STATE Output: The solution $\bbb{X}^*$.
	\STATE Initialization: Set $\bbb{X}^0=\bbb{O}^{-1}$, where $\bbb{O}$ is a diagonal matrix with $\bbb{O}_{ii}=\bbb{\Sigma}_{ii},~\forall i\in[n]$. Set $k=0$. 
	\WHILE{\TRUE}
\STATE $\setminus \setminus$ Greedy Pursuit Stage
		\WHILE{$   \| \bbb{X}^k \|_{0,\text{off}}   < s$}
			\FOR{every $j \in \Z(\bbb{X}^k)$}
				\STATE
\vspace{-10pt}
\beq \label{eq:stage:1}
\begin{split}
 f_j = \min_{\theta}~f(\bbb{X}^k + \theta \bbb{E}_j),s.t.~\bbb{X}^k + \theta \bbb{E}_j \succ 0
\end{split}
\eeq
\vspace{-10pt}
			\ENDFOR
			\STATE $j_k = \arg\min\{f_j: j \in \Z(\bbb{X}^k)\}$
			\IF{$f_{j_k} < f(\bbb{X}^k)$}
%				\STATE $\bbb{X^{k+1}} = O(\bbb{X}^k,S(\bbb{X}^k)\cup j_{k}) $
\STATE Solve (\ref{eq:subproblem}) to get $\bbb{X}^{k+1}$ with $S=\S(\bbb{X}^k)\cup j_{k}$.
				\STATE $k = k + 1$
			\ENDIF
		\ENDWHILE
\STATE
\noindent \rule{6cm}{0.4pt}
\STATE $\setminus \setminus$ Swap Coordinates Stage
		\FOR{every $i \in \S(\bbb{X}^k) $}
                  \FOR {every $j \in \Z(\bbb{X}^k)$}
				\STATE
\vspace{-10pt}
\beq \label{eq:stage:2}
\begin{split}
 f_{i,j} = \min_{\theta}~ f(\bbb{X}^k - \bbb{X}^k_{i}\bbb{E}_i +
				\theta \bbb{E}_j),\\
 s.t.~\bbb{X}^k - \bbb{X}^k_{i}\bbb{E}_i +
				\theta \bbb{E}_j \succ 0
\end{split}
\eeq
\vspace{-10pt}
		\ENDFOR
\ENDFOR
		\STATE $(i_k,j_k) = \arg\min\{f_{i,j}: i \in \S(\bbb{X}^k), j \in \Z(\bbb{X}^k)\}$
		\IF{$f_{i_k,j_k} < f(\bbb{X}^k) $}
%			\STATE $\bbb{X^{k+1}} = O(\bbb{X}^k,S(\bbb{X}^k)\setminus i_k\cup j_k) $
\STATE Solve (\ref{eq:subproblem}) to get $\bbb{X}^{k+1}$ with $S=(\S(\bbb{X}^k)\setminus i_k)\cup j_k$.
			\STATE $ k = k + 1$
		\ELSE
			\STATE set $\bbb{X}^* \leftarrow \bbb{X}^{k+1}$
			\STATE break
		\ENDIF
	\ENDWHILE

	\end{algorithmic}

\end{algorithm}

We summarize our proposed method in Algorithm \ref{alg:main} and have a few remarks on it below.

$\bullet$ Two-stage algorithm. At each iteration of the algorithm, one or two variables of the solution are updated. At the first greedy pursuit stage, the algorithm greedily picks one coordinate $i\in\Z(\bbb{X}^k)$ that leads to the greatest descent from $\Z(\bbb{X}^k)$. This strategy is also known as forward greedy selection in the literature \cite{tropp2007signal,zhang2011adaptive}. At the second swap coordinates stage, the algorithm enumerates all the possible pairs $(i,j)$ with $i\in\S(\bbb{X}^k)$ and $j\in\Z(\bbb{X}^k)$ that leads to the greatest descent and changes the two coordinates from zero/non-zero to non-zero/zero. At both stages, once the support set has been updated, Algorithm \ref{alg:main} runs a convex subproblem procedure to solve (\ref{eq:subproblem}) over the support set to compute a more `compact' solution.

$\bullet$ One-dimensional sub-problem. The problems in (\ref{eq:stage:1}) and (\ref{eq:stage:2}) reduce to the following optimization problem:
\beq \label{eq:one:dim}
\begin{split}
&&  \min_{\theta}~f(\theta) \triangleq \la  \bbb{\Sigma},\bbb{V}+\theta \bbb{E}_j \ra-\log \det( \bbb{V}+\theta \bbb{E}_j)\\
&& s.t.~\bbb{V}+\theta \bbb{E}_i \succ \bbb{0}~~~~~~~~~~~~~~~~~~~~~
\end{split}
\eeq
\noi with $\bbb{V}=\bbb{X}^k$ for (\ref{eq:stage:1}) and $\bbb{V}=\bbb{X}^k-\bbb{X}_i^k \cdot \bbb{E}_i$ for (\ref{eq:stage:2}). We now discuss how to simplify problem (\ref{eq:one:dim}). We define $\bbb{Y} \triangleq \bbb{V}^{-1}$ and obtain the following equations: $ \det(\bbb{V} + \theta\bbb{E}_j) = \det(\bbb{V} (\bbb{I} + \theta \bbb{Y} \bbb{E}_j))
=\det(\bbb{V})  \det(\bbb{I} + \theta \bbb{U}^T_{cr}\bbb{Y}\bbb{U}_{rc})= \det(\bbb{V})  (1 + 2\bbb{Y}_{rc}\theta -   \bbb{O}_{rc}  \theta^2 ) $, where $\bbb{O}_{rc} \triangleq \bbb{Y}_{rr} \bbb{Y}_{cc}- \bbb{Y}_{rc}^2 >0 $,~$r\triangleq row(j)$,~$c \triangleq col(j)$,~$\bbb{U}_{rc}  \triangleq [\bbb{e}_{r}~\bbb{e}_{c}] \in \mathbb{R}^{n\times 2}$,~$\bbb{U}_{cr} \triangleq [\bbb{e}_{c}~\bbb{e}_{r}] \in \mathbb{R}^{n\times 2}$. Noticing that $-\log\det(\bbb{V} + \theta\bbb{E}_j) = - \log\det(\bbb{V}) - \log(1 + 2\bbb{Y}_{rc}\theta -   \bbb{O}_{rc}  \theta^2 )$ and $ \la \bbb{\Sigma},\bbb{E}_j \ra = 2\bbb{\Sigma}_j = 2\bbb{\Sigma}_{rc}$, we can simplify problem (\ref{eq:one:dim}) to the following one-dimensional convex problem:
\beq
\textstyle \min_{\theta}~ f(\theta) \triangleq 2\theta \bbb{\Sigma}_{rc} - \log \left(1 + 2\bbb{Y}_{rc}\theta -  \bbb{O}_{rc} \theta^2 \right) + C \nn\\
\textstyle s.t.~1 + 2\bbb{Y}_{rc}\theta -   \bbb{O}_{rc} \theta^2 >0,~~~~~~~~~~~~~~~~~\nn
\eeq
\noi where $C = \la \bbb{\Sigma}, \bbb{V}\ra - \log \det(\bbb{V})$ is a constant. Noting that $f(\theta)$ is differentiable, we set the gradient of $f(\theta)$ to zero, we obtain $0=2\bbb{\Sigma}_{rc}+ \frac{2 \bbb{O}_{rc} \theta - 2  \bbb{Y}_{rc}}{1 + 2\bbb{Y}_{rc}\theta -   \bbb{O}_{rc} \theta^2}$. There are two solution to this equation. However, only one of these satisfies the bound constraint $1 + 2\bbb{Y}_{rc}\theta -   \bbb{O}_{rc} \theta^2  >0$. Thus, the optimal solution $\theta^*$ can be computed as:
\begin{equation}%\footnotesize
\textstyle
\theta^* = \left \{
\begin{aligned}
{\bbb{Y}_{rc}}/{\bbb{O}_{rc}},~~~~~~~~~~~~~~~~~~~~~~~~~~~~~~~~ \text{if}~ \bbb{\Sigma}_{rc} = 0;\\
 {\bbb{Y}_{rc}}/{\bbb{O}_{rc}} + {1}/{(2\bbb{\Sigma}_{rc})} - ~~~~~~~~~~~~~~~~~~~~~~~~~~~~~~~~~~~~~~~\nn\\
  \sqrt{\bbb{O}^2_{rc}+4\bbb{\Sigma}^2_{rc}\bbb{Y}_{rr}\bbb{Y}_{cc}}/{(2\bbb{O}_{rc}\bbb{\Sigma}_{rc})},~ \text{if}~\bbb{\Sigma}_{rc} \neq 0. \nn
\end{aligned}\nn
\right.
\label{non-diag-minimizer}
\end{equation}

$\bullet$ Fast matrix computation. In our algorithm, we assume that $\bbb{Y} \triangleq \bbb{V}^{-1}$ is available. This can be achieved by using the follow strategy. We keep a record of $\bbb{X}^{-1}$ in every iteration. Once the solution $\bbb{X}$ is changed to $\bbb{T} = \bbb{X} + \varpi \bbb{E}_j$, we quickly estimate $\bbb{T}^{-1}$ using the well-known Sherman-Morrison-Woodbury formula \footnote{  $(\bbb{A}+\bbb{PCQ})^{-1} = \bbb{A}^{-1} - \bbb{A}^{-1}\bbb{P} (\bbb{C}^{-1}+\bbb{QA}^{-1}\bbb{P})^{-1}\bbb{QA}^{-1}$  }. Specifically, we rewrite $\bbb{T}$ as $\bbb{T}=\bbb{X}+ \bbb{U}_{rc} \text{diag}(\varpi \bbb{I}_2) \bbb{U}_{cr}^T$ and apply the Sherman-Morrison-Woodbury formula with $\bbb{A}=\bbb{X},~\bbb{P}=\bbb{U}^{rc},~\bbb{C}=\text{diag}(\varpi \bbb{I}_2)$ and $\bbb{Q}=\bbb{U}_{cr}^T$, leading to the following equation: $\bbb{T}^{-1}=\bbb{X}^{-1}-\bbb{X}^{-1}\bbb{U}^{rc}( (\text{diag}(\varpi \bbb{I}_2))^{-1} + \bbb{U}_{cr}^T \bbb{X}^{-1} \bbb{U}^{rc} )^{-1}\bbb{U}_{cr}^T\bbb{X}^{-1}$, where \text{diag}$(\bbb{z})$ is a diagonal matrix with $\bbb{z}$ as the main diagonal entries and $\bbb{I}_d \in \mathbb{R}^{d\times d}$ is an identity matrix. Finally, we obtain the following results: $\bbb{T}^{-1} = \bbb{X}^{-1} - \frac{\varpi}{1-\delta \varpi^2+2 \bbb{Y}_{rc} \varpi}\cdot \bbb{E}_{rc}\bbb{W}\bbb{E}_{cr}^T$, where $\delta \triangleq \bbb{Y}_{rr} \bbb{Y}_{cc}-\bbb{Y}_{rc}^2$, $\bbb{E}_{rc} \triangleq [\bbb{Y}_{:r} \bbb{Y}_{:c}]\in \mathbb{R}^{n\times 2}$, $\bbb{E}_{cr} \triangleq [\bbb{Y}_{:c} \bbb{Y}_{:r}]\in \mathbb{R}^{n\times 2}$, $\bbb{W} \triangleq {\tiny \begin{pmatrix}
                                              1+\varpi \bbb{Y}_{rc}& - \varpi \bbb{Y}_{cc} \\
                                               -\varpi \bbb{Y}_{rr}& 1+\varpi \bbb{Y}_{rc}\\
                                             \end{pmatrix}}$, and $\bbb{Y}_{:r}\in \mathbb{R}^n$ denotes the $r$-th column of $\bbb{Y}$.

\textbf{Remarks:} (i) Algorithm \ref{alg:main} can be viewed as an improved version of classical greedy pursuit method for solving the sparsity-constrained inverse covariance selection problem. Given the fact that greedy pursuit methods achieve state-of-the-art performance in varieties of non-convex optimization problems (e.g. compressed sensing \cite{tropp2007signal}, kernel learning \cite{KeerthiCD06}, and sensor selection \cite{KrauseG07}), our proposed method is expected to achieve state-of-the-art performance as well. (ii) Algorithm \ref{alg:main} is also closely related to forward-backward greedy method in the literature \cite{zhang2011adaptive}. To obtain the greatest descent, while  forward-backward strategy considers the removal step and adding step sequentially, the swapping strategy (refer to the swap coordinates stage in Algorithm \ref{alg:main}) considers these two steps simultaneously. Thus, the swapping strategy is generally stronger than the forward-backward strategy.

%Note that for one iteration while the computationalcomplexity for forward-backward strategy is $2 \times \mathcal{O}(p)$, the computational complexity for swapping strategy is $\mathcal{O}(p-s) \times \mathcal{O}(s)$ (refer to the two FOR loops in the swap coordinates stage in Algorithm 1). Here $p$ denotes the number of unknown variables. When $s$ is small, the computational complexities of the two strategies are comparable.

\section{Convex Optimization Over Support Set}\label{re-optimizer}

After the support set has been determined, one need to solve the reduced convex sub-problem as in (\ref{eq:subproblem}), which is equivalent to the following convex composite minimization problem:
\beq \label{eq:general:op}
\begin{split}
\min_{\bbb{X}\succ 0}~F(\bbb{X}) \triangleq f(\bbb{X}) + p(\bbb{X}), ~~~~~~~~\\
\text{with}~p(\bbb{X}) \triangleq I_{\Omega}\left(\bbb{X}\right),~\Omega\triangleq \{\bbb{X} ~|~ \bbb{X}_{Z}=\textbf{0}\},
\end{split}
\eeq
where $Z\triangleq \{1,2,...,n^2\} \setminus S$ and $I_{\Omega}$ is an indicator function of the convex set $\Omega$ with $I_{\Omega}(\bbb{V}) \tiny = \begin{cases}
0, & \bbb{V} \in \Omega\\
\infty, &{\text{otherwise.}}\\
\end{cases}$. In what follows, we present an efficient Newton-Like Optimization Algorithm (NLOA) to tackle this problem. This method has the good merits of greedy descent and fast convergence.

%monotonicity

%Following \cite{tseng2009coordinate,yun2011block,YuanYZH16}, we keep the non-smooth function $g(\bbb{X})$ and build a quadratic Newton approximation around any solution $\bbb{X}$ for the objective function using second-order Taylor expansion by:
%\beq \label{eq:TTT}
%\mathcal{T}(\bbb{X}^t) \triangleq \arg \min_{\bbb{Y}}~q(\bbb{Y},~{\bbb{X}^t}) + p(\bbb{\bbb{Y}})
%\eeq
%\noi where
%\beq
%q(\bbb{Y},\bbb{X}) \triangleq f(\bbb{X}) + \langle \bbb{Y} - \bbb{X},g(\bbb{X})\rangle + \frac{1}{2}\vec(\bbb{Y}-\bbb{X})^Th(\bbb{X}) \vec(\bbb{Y}-\bbb{X}).\nn
%\eeq
%\noi Note that the first-order and second-order derivatives of the objective function $f(\bbb{X})$ can be expressed as \cite{HsiehSDR14}:
%\beq
%g(\bbb{X}) = \bbb{\Sigma} - \bbb{X}^{-1},~h(\bbb{X}) =  \bbb{X}^{-1} \otimes \bbb{X}^{-1}.
%\eeq
%
%\noi

Following \cite{tseng2009coordinate,yun2011block,LeeSS14,YuanYZH16}, we develop a quadratic approximation around any solution $\bbb{X}$ for the objective function using second-order Taylor expansion:
\beq
q(\bbb{\Theta,X}) \triangleq f(\bbb{X}) + \langle \bbb{\Theta},g(\bbb{X})\rangle + \frac{1}{2}\vec(\bbb{\Theta})^Th(\bbb{X}) \vec(\bbb{\Theta}),\nn
\eeq
\noi where the first-order and second-order derivatives of the objective function $f(\bbb{X})$ can be expressed as \cite{HsiehSDR14}:
\beq
g(\bbb{X}) = \bbb{\Sigma} - \bbb{X}^{-1},~h(\bbb{X}) =  \bbb{X}^{-1} \otimes \bbb{X}^{-1}.\nn
\eeq

\noi Then, we keep the non-smooth function $g(\bbb{X})$ and build a quadratic approximation for the smooth function $f(\bbb{X})$ by:
  \beq \label{eq:newton:direction}
  d(\bbb{X}^t) \triangleq \arg \min_{\bbb{\Delta}}~q(\bbb{\Delta};{\bbb{X}^t}) + p(\bbb{X}^t+\bbb{\Delta}).
  \eeq

  \noi Once the Newton direction $d(\bbb{X}^t)$ has been computed, one can employ an Arimijo-rule based step size selection to ensure positive definiteness and sufficient descent of the next iterate. We summarize our Newton-like algorithm in Algorithm \ref{alg:newtoncg}. Note that the initial point $\bbb{X}^0$ has to be a feasible solution and the positive definiteness of all the following iterates $\bbb{X}^t$ will be guaranteed by the step size selection procedure (see step \ref{alg:step:linesearch:end} in Algorithm \ref{alg:linearcg}). For notational convenience, we use the shorthand
  \beq
  f^t = f(\bbb{X}^t), ~\bbb{G}^t = g(\bbb{X}^t),~\bbb{H}^t=h(\bbb{X}^t),~\bbb{D}^t=d(\bbb{X}^t)\nn
  \eeq
  \noi to denote the objective value, first-order gradient, hessian matrix and the search direction at the point $\bbb{X}^t$, respectively.

\begin{algorithm}[!h]
\caption{ {\bf NLOA: Newton-Like Optimization Algorithm to Solve (\ref{eq:subproblem}) for Optimization Over Support Set.}}
\begin{algorithmic}[1]
\label{alg:newtoncg}
  \STATE Input: $\bbb{X}^0$ such that $\bbb{X}^0\succ0$ and $\bbb{X}_{Z} =\textbf{0}$.
  \STATE Output: $\bbb{X}^t$
  \STATE Initialize $t=0$
  \FOR {$t=1$ \TO $T_{\text{out}}$}
  \STATE Solve Problem (\ref{eq:newton:direction}) by Algorithm \ref{alg:linearcg} to obtain $d(\bbb{X}^t)$.
  \STATE \label{alg:step:linesearch:begin} Perform step-size search to get $\alpha^t$ such that: % Use an Armijo-rule based step-size selection to get $\alpha^t$ such that
  \STATE \label{alg:step:linesearch:end}~~~(1) $\bbb{X}^{t+1} = \bbb{X}^t + \alpha^t d(\bbb{X}^t)$ is positive definite and
  \STATE \label{alg:step:linesearch:end2} ~~~(2) there is sufficient decrease in the objective.
  \STATE Increment $t$ by 1
  \ENDFOR
\end{algorithmic}
\label{algo:newton}
\end{algorithm}

%With the choice of $\bbb{X}^{t-1} \succ 0$ and $\bbb{X}^{t-1}_Z=\textbf{0}$

\subsection{Computing the Search Direction} \label{subsect:dir}

This subsection focuses on finding the search direction in (\ref{eq:newton:direction}). With the choice of $\bbb{X}^{0} \succ 0$ and $\bbb{X}^{0}_Z=\textbf{0}$, (\ref{eq:newton:direction}) boils down to the following optimization problem:
\beq\label{eq:newton:dir}
\textstyle \min_{\bbb{\Delta}}~\la \bbb{\Delta},  \bbb{G}^t \ra + \frac{1}{2} \vec(\bbb{\Delta})^T \bbb{H}^t \vec(\bbb{\Delta}),~s.t.~\bbb{\Delta}_Z=\textbf{0}.
\eeq
\noi It appears that (\ref{eq:newton:dir}) is very difficult to solve. First, it involves computing and storing an $n^2 \times n^2$ Hessian matrix $\bbb{H}^t$. Second, it is a constrained optimization program with $n\times n$ variables and $|Z|$ equality constraints.

We carefully analyze (\ref{eq:newton:dir}) and consider the following solutions. For the first issue, one can exploit the Kronecker product structure of the Hessian matrix to avoid storing it. Recall that $\left(\bbb{A}\otimes \bbb{B}\right) \vec(\bbb{C}) = \vec(\bbb{BCA}),\forall \bbb{A},\bbb{B},\bbb{C} \in \mathbb{R}^{n\times n}$. Given any vector $\vec(\bbb{D}) \in \mathbb{R}^{n^2 \times 1}$, using the fact that the Hessian matrix can be computed as $\bbb{H}= \bbb{X}^{-1} \otimes \bbb{X}^{-1}$, the Hessian-vector product can be computed efficiently as: $\bbb{H}\vec(\bbb{D}) =(\bbb{X}^{-1} \otimes  \bbb{X}^{-1}) \vec{(\bbb{D})} = \vec( \bbb{X}^{-1} \bbb{D} \bbb{X}^{-1} )$, which only involves matrix-matrix computation. For the second issue, (\ref{eq:newton:dir}) is, in fact, a unconstrained quadratic program with $n^2 - |Z|$ variables. In order to deal with the variables indexed by $Z$, one can explicitly enforce the entries of $Z$ for current solution and its corresponding gradient to \textbf{0}. Therefore, the constraint $\bbb{\Delta}_Z=\textbf{0}$ can always be satisfied. Finally, linear conjugate gradient method can be used to solve (\ref{eq:newton:dir}).

%The specific value of $T_{\text{in}}$ does not affect the correctness of algorithm \ref{alg:linearcg}, but only its efficiency.

We summarize our modified linear conjugate gradient method for computing the search direction in Algorithm \ref{alg:linearcg}. The algorithm involves a parameter $T_{\text{in}}$ controlling the maximum number of iterations. Empirically, we found that a value of $T_{\text{in}}=5$ usually leads to good overall efficiency.

\begin{algorithm}[!h]
\caption{ {\bf A Modified Linear Conjugate Gradient to Find the Newton Direction $\bbb{D}$ as in (\ref{eq:newton:dir}).}}
	\label{alg:linearcg}
	\begin{algorithmic}
\STATE Input: $\bbb{Y} =(\bbb{X}^t)^{-1}$, and current gradient $\bbb{G}=g(\bbb{X}^t)$, Specify the maximum iteration $T\in \mathbb{N}$
\STATE Output: Newton direction $\bbb{D}\in \mathbb{R}^{n\times n}$
	\STATE $\bbb{D} = \bbb{0}$,~$\bbb{R} = -\bbb{G} - \bbb{YDY}$, $\bbb{R}_Z = \bbb{0}$
	\STATE $\bbb{P = R}$,~$r_{old} = \langle \bbb{R},\bbb{R}\rangle$
	\FOR {$p=1$ \TO $T_{\text{in}}$}
	\STATE $\bbb{B} = \bbb{YPY}$,~$\alpha = \frac{r_{old}}{\bbb{\langle P,B\rangle}}$
	\STATE $\bbb{D} = \bbb{D} + \alpha\bbb{P}$, $\bbb{R = R} - \alpha\bbb{B}$
	\STATE $\bbb{D_Z} = \bbb{0}$, $\bbb{R_Z} = \bbb{0}$,~$r_{new} = \langle \bbb{R},~\bbb{R}\rangle$
	\STATE $\bbb{P = R} + \tfrac{r_{\text{new}}}{r_{\text{old}}}\bbb{P}$,~$r_{old} = r_{new}$
	\ENDFOR
	\end{algorithmic}
\end{algorithm}

% $$tr(ABCD)=\vec(D^T)^T \left(C^T \otimes A\right) \vec(B)=\vec(D)^T \left(A \otimes C^T\right) \vec(B^T)$$

%\begin{lemma}
%The objective $f(X)$ is convex on $\mathbb{S}_{++}^{n}$. Moreover, when $X\succeq \theta I$, the objective function $f(X)$ is strongly convex with parameter $\omega=\theta tr(V)$.
%\begin{proof}
%For each and every $Y\in \mathbb{S}^n$, we have:
%\beq
%\frac{d^2}{dt^2} f(X+tY) = \la 2 (X+tY)^{-1}Y(X+tY)^{-1}Y(X+tY)^{-1}, V\ra \geq 0
%\eeq
%The inequality holds due to the fact that $\la A,B \ra \geq 0$.
%\end{proof}
%\end{lemma}

%For all $X\in \mathbb{R}^n$,$D\in \mathbb{R}^{n\times n}$ and some open interval of $t\in \mathbb{R}$, we have:
%\beq
%f(X+tD) = f(X) + t \partial f(X;D) + \frac{1}{2} t^2 \partial f^2(X;D) + O(t^3)
%\eeq
%where
%\beq
%\partial f(X;D) =  \la - X^{-1} D X^{-1}, V \ra \nn \\
%\partial f^2(X;D) =  \la 2 X^{-1} D X^{-1}D X^{-1}, V \ra \nn
%\eeq

% \footnote{Here we abuse that notations and use $\beta^t$ to denote the $t$th power of $\beta$.}
%\begin{table}[!ht]
%\centering
%
%\begin{tabular}{|c|c|}
%\hline
%data set & $n$ \\ \hline
%Gaussian-Random-500  & 500 \\
%Gaussian-Random-1000 & 1000 \\
%Gaussian-Random-2000 & 2000 \\
%Gaussian-Random-4000 & 4000 \\
%Sparse-Structured-500  & 500 \\
%Sparse-Structured-1000 & 1000 \\
%Sparse-Structured-2000 & 2000 \\
%Sparse-Structured-4000 & 4000 \\
%Real-World-Lymph & 587 \\
%Real-World-Lymph & 587 \\
%Real-World-isolet & 617 \\
%Real-World-usps & 256 \\ \hline
%\end{tabular}
%\caption{Datasets}
%\label{tabl:datasets}
%\end{table}

\subsection{Computing the Step Size}\label{subsect:step}

Once the Newton direction $\bbb{D}$ is computed, we need to find a step size $\alpha\in (0,1]$ in order to ensure the positive definiteness of the next iterated result, i.e. $\bbb{X}+\alpha \bbb{D}$, so that a sufficient decrease of the objective function will be resulted. We use Armijo's rule and try step size $\alpha \in \{\eta^0,\eta^1,...\}$ with a constant decrease rate $0<\eta<1$ until we find the smallest $t\in \mathbb{N}$ with $\alpha=\eta^t$ such that $\bbb{X}+\alpha \bbb{D}$ is (i) positive definite, and (ii) satisfies the following sufficient decrease condition \cite{tseng2009coordinate}:
\beq \label{eq:suff:dec}
\textstyle f(\bbb{X}^t+\alpha^t \bbb{D}^t) \leq f(\bbb{X}^t) + \alpha^t \omega \la \bbb{G}^t ,\bbb{D}^t \ra,\nn
\eeq
\noi where $0<\omega<0.5$. In our experiments, we set $\eta = 0.1$ and $\omega=0.25$.

We verify positive definiteness of the solution when we compute its Cholesky factorization (taking $\frac{1}{3}n^3$ flops). We note that the Cholesky factorization dominates the computational cost in the step-size computations. To reduce the computation cost, we can reuse the Cholesky factor in the previous iteration when evaluating the objective function (that requires the computation of $\log\det(\bbb{X})$) and the gradient (that requires the computation of $\bbb{X}^{-1}$).% The decrease condition in (\ref{eq:suff:dec}) has been considered in \cite{tseng2009coordinate} to ensure that the objective value not only decreases but also decreases by a certain amount $\alpha^t\omega \la \bbb{G}^t ,\bbb{D}^t \ra$, where $\la \bbb{G}^t ,\bbb{D}^t \ra$ measures the optimality of the current solution.

%For convex composite function minimization, $\bbb{D}^t=\bbb{0}$ implies the global convergence of Algorithm \ref{alg:newtoncg} .

\section{Theoretical Analysis} \label{eq:theana}

\subsection{Convergence Analysis of Algorithm 1}

 We present the convergence results for Algorithm \ref{alg:main}, which are analogous to the results in \cite{beck2016}.
\begin{proposition}
Let ${\bbb{X}^k}$ be the sequence generated by algorithm \ref{alg:main}. Algorithm \ref{alg:main} outputs a coordinate-wise minimum point $\bbb{X}^*$ with $f(\bbb{X}^*)\leq f(\bbb{P})$ for every $\bbb{P} \in \N$, where $\N \triangleq \{\bbb{X}~|~\bbb{X}  \succ \bbb{0},~\| \bbb{X}^* - \bbb{X}\|_0 \leq 2\}$.

\end{proposition}

\begin{proof}
 Note that it takes finite iterations for any convex optimization algorithm to produce an optimal solution with a given support set. Combining with the monotonicity of Algorithm \ref{alg:main}, we conclude that the sequence of function values $f(\bbb{X}^k)$ are monotonically decreasing and Algorithm \ref{alg:main} stops after a finite number of iterations. We define:
\beq
&& \textstyle \N^0  = \{ \bbb{P} ~|~\bbb{P}\in\N,~\S(\bbb{P}) \subseteq \S(\bbb{X}^*)\}, \nn\\
&& \textstyle \N^1 = \{\bbb{P} ~|~\bbb{P}\in\N,~\S(\bbb{P}) = \S(\bbb{X}^*) \cup \{ j\}\},\nn\\
&& \textstyle \N^2 = \{\bbb{P} ~|~\bbb{P}\in\N,~ \S(\bbb{P}) = \S(\bbb{X}^*) / \{i\}\cup \{j\} \},\nn
\eeq
\noi for all $i\in\S(\bbb{X}^*),~j \in \Z(\bbb{X}^*)$. Clearly, we have $\N = \N^0 \cup \N^1 \cup \N^2$. Now we assume that point $\bbb{X}^*$ is generated by Algorithm \ref{alg:main}.

For the case $\N^0$, $\bbb{X^*}$ is a global optimal point generated by the convex optimization subproblem on the given support set. Therefore, $f(\bbb{X}^*) \leq f(\bbb{X})$ for any $\bbb{X} \in \N^0$.

For the case $\N^2$, we notice that Algorithm \ref{alg:main} terminates only if after the swap coordinates stage. For any $i \in \S(\bbb{X^*})$ and $j \in \Z(\bbb{X^*})$, we have the following inequality:
\beq
\textstyle f_{i,j} = \min_{\theta}\{ f(\bbb{X}^k - \bbb{X}^k_{i}\mathbf{E}_i+ \theta \mathbf{E}_j)\} \geq f(\bbb{X^*}).
\notag
\eeq
\noi Therefore, we have that $\N^2 = \varnothing$, which implicates that we cannot find any swap from support set and non-support set to achieve descent on the objective value. Thus, $f(\bbb{X}^*) \leq f(\bbb{X})$ for any $\bbb{X} \in \N^2$.

For the case $\N^1$, Algorithm \ref{alg:main} must perform greedy pursuit stage before entering the swap coordinates stage. The greedy stage terminates only if for any
$j \in \Z(\bbb{X^*})$,
\beq
\textstyle f_j = \min_{\theta}\{ f(\bbb{X}^k + \theta \mathbf{E}_j)\} \geq f(\bbb{X}^*).\nn
\eeq
\noi It implies that we have selected the element that leads to greatest descent as a new member of non-zero elements when $\|\bbb{X}\|_0 \leq s$. We conclude that $f(\bbb{X}^*) \leq f(\bbb{X})$, for any $\bbb{X} \in \N^1$.

%Therefore, we conclude that $\bbb{X}^*$ is a coordinate-wise minimum point of problem (\ref{alg:main}).
Therefore, we finish the proof of this lemma.
\end{proof}

\subsection{Convergence Analysis of Algorithm 2}

 %It is worthwhile to point out that our convergence proofs relay on this sufficient decrease condition.
%\subsection{Convergence Analysis}\label{subsect:conv}

This subsection provides some convergence analysis for the proposed Newton-like optimization algorithm in Algorithm \ref{algo:newton}. We denote $\{\bbb{X}^t\}_{t=0}^{\infty}$ as the sequence generated by the algorithm and $\bbb{X}^*$ as the global optimal solution set for the convex problem in (\ref{eq:subproblem}). Throughout this subsection, we make the following assumption.

\begin{assumption}
The objective function $f(\bbb{X})$ is strongly convex with the modulus $\sigma$ and gradient Lipschitz continues with constant $L$ for all $\bbb{X}^t$ with $t=0,1,2,...,\infty$.
\end{assumption}

%There exists constants $C_1$ and $C_2$ such that $C_1 \leq \bbb{\lambda}(\bbb{X}^t) \leq C_2,~\forall t$.

% strong convexity and gradient Lipschitz continuity for $f(\cdot)$ in (\ref{eq:main}).

\noi \textbf{Remarks:} This assumption is mild and equivalent to assuming the solution is bounded since it holds that
\beq
 && \sigma \leq \bbb{\lambda}(\bbb{H}^t) \leq L \Leftrightarrow \sigma \leq \bbb{\lambda}({(\bbb{X}^t)}^{-1} \otimes {(\bbb{X}^t)}^{-1}) \leq L \nn\\
&\Leftrightarrow& \sqrt{\sigma} \leq \bbb{\lambda}({(\bbb{X}^t)}^{-1}) \leq \sqrt{L}
\Leftrightarrow 1/\sqrt{L} \leq \bbb{\lambda}(\bbb{X}^t) \leq 1/\sqrt{\sigma},\nn
\eeq
\noi where $\bbb{\lambda}(\bbb{X}) \in \mathbb{R}^n$ as the eigenvalues of $\bbb{X}$ in increasing order with $\bbb{\lambda}_1(\bbb{X})\leq \bbb{\lambda}_2(\bbb{X})\leq,...,\leq\bbb{\lambda}_n(\bbb{X})$.

The following lemma characterizes the optimality of $d(\bbb{X}^t)$. It is nearly identical to Lemma 1 in \cite{TsengY09}. For completeness, we present the proof here.

\begin{lemma} \label{lemma:optimal:d}
It holds that
\beq \label{eq:newton:suff:dec}
\|\bbb{D}^t\|_{\bbb{H}^t}^2  + \la \bbb{G}^t, \bbb{D}^t \ra  \leq 0,~\forall \bbb{D}^t~\text{with}~\bbb{X}^t +\bbb{D}^t \in \Omega.
\eeq

\begin{proof}
Noticing $\bbb{D}^t \triangleq d(\bbb{X}^t)$ is the minimizer of (\ref{eq:newton:direction}), we have:
\beq
q(\bbb{D}^t;{\bbb{X}^t}) + p(\bbb{X}^t+\bbb{D}^t) \leq q(\bbb{Z};{\bbb{X}^t}) + p(\bbb{X}^t+\bbb{Z}),~\forall \bbb{Z}. \nn
\eeq
\noi Letting $\bbb{Z} = \alpha \bbb{D}^t$ where $\alpha$ is any constant with $\alpha\in[0,1]$, we obtain:
\begin{align*}
&~~\la \bbb{G}^t, \bbb{D}^t \ra + \tfrac{1}{2}\|\bbb{D}^t\|^2_{\bbb{H}^t} + p(\bbb{X}^t + \bbb{D}^t)\nn\\
\leq&~~ \la \bbb{G}^t, \alpha \bbb{D}^t \ra + \tfrac{1}{2}\|\alpha \bbb{D}^t\|^2_{\bbb{H}^t} + p(\bbb{X}^t + \alpha \bbb{D}^t ) \nn\\
\leq&~~ \la \bbb{G}^t, \alpha \bbb{D}^t \ra + \tfrac{1}{2}\|\alpha \bbb{D}^t\|^2_{\bbb{H}^t} + \alpha p(\bbb{X}^t + \bbb{D}^t ) + (1-\alpha) p(\bbb{X}^t),\nn
\end{align*}
\noi where the last inequality uses the convexity of $p(\cdot)$. Rearranging terms yields:
\beq
(1-\alpha) (\la \bbb{G}^t, \bbb{D}^t \ra +  p(\bbb{X}^t + \bbb{D}^t ) - p(\bbb{X}^t))     \leq \tfrac{\alpha^2-1}{2}\|\bbb{D}^t\|^2_{\bbb{H}^t}\nn\\
\la \bbb{G}^t, \bbb{D}^t \ra +  p(\bbb{X}^t + \bbb{D}^t ) - p(\bbb{X}^t)    \leq -\tfrac{1+\alpha}{2}\|\bbb{D}^t\|^2_{\bbb{H}^t}.\nn~~~~~
\eeq
\noi Since $\bbb{X}^t\in\Omega,~\bbb{X}^t + \bbb{D}^t \in \Omega$, we have $p(\bbb{X}^t + \bbb{D}^t ) = p(\bbb{X}^t)=0$. Letting $\alpha=1$, we obtain (\ref{eq:newton:suff:dec}).

\end{proof}

\end{lemma}

\begin{theorem}
\label{theorem:conv}
(Global Convergence). We have the following results: (i) There exists a strictly positive constant $\forall t,~\alpha^t \leq \min(1,1/(\sqrt{L} C_1)-\epsilon,C_2)$ such that the positive definiteness and sufficient descent conditions (refer to step \ref{alg:step:linesearch:end}-\ref{alg:step:linesearch:end2} of Algorithm \ref{algo:newton}) are satisfied. Here $\epsilon$ denotes a sufficient small positive constant. $C_1 \triangleq \bbb{\lambda}_n(\bbb{\Sigma})/\sigma + \sigma^{-3/2}$ and $C_2 \triangleq {2\sigma (1-\omega)}/{L}$ are some constants which are independent of the current solution $\bbb{X}^t$. (ii) The sequence $f(\bbb{X}^t)$ is non-increasing and any cluster point of the sequence $\bbb{X}^t$ is the global optimal solution of (\ref{eq:general:op}).

%The following lemma provides some theoretical insights of the line search program. It states that a strictly positive step size can always be achieved in Algorithm \ref{algo:newton}. We remark that this property is very crucial in our global convergence analysis of the algorithm.

\begin{proof}

\noi (i) First, we focus on the positive definiteness condition. We now bound $\lambda_{n}(\bbb{D}^t)$. By Lemma \ref{lemma:optimal:d}, we have $\forall \bbb{D}^t~\text{with}~\bbb{X}^t +\bbb{D}^t \in \Omega$:
\beq \label{eq:quad:D}
0 &\geq&  \la \bbb{D}^t, \bbb{G}^t \ra + \|\bbb{D}^t\|_{\bbb{H}^t}^2 \nn\\
&\geq& \textstyle -\bbb{\lambda}_n(\bbb{D}^t) \bbb{\lambda}_n(\bbb{G}^t) + \sigma\|\bbb{D}^t\|_{\text{F}}^2 \nn \\
&=& \textstyle -\bbb{\lambda}_n(\bbb{D}^t) \bbb{\lambda}_n(\bbb{\Sigma}  -  (\bbb{X}^t)^{-1}) + \sigma\|\bbb{D}^t\|_{\text{F}}^2 \nn \\
&\geq& \textstyle -\bbb{\lambda}_n(\bbb{D}^t) \cdot (  \bbb{\lambda}_n(\bbb{\Sigma}) + 1/\sqrt{\sigma})  + \sigma (\bbb{\lambda}_n(\bbb{D}^t))^2,~~~
\eeq
\noi where the second step uses the fact that $\bbb{\lambda}(\bbb{H}^t) \geq \sigma$ and the inequality that $\la \bbb{A},\bbb{B} \ra \geq -\bbb{\lambda}_n(\bbb{A})\bbb{\lambda}_n(\bbb{B}),~\forall \bbb{A},~\bbb{B}$; the third step uses the definition of $\bbb{G}^t=\bbb{\Sigma}-(\bbb{X}^t)^{-1}$; the last step uses the inequalities that $\bbb{\lambda}_n(\bbb{\Sigma} - (\bbb{X}^t)^{-1}) \leq \bbb{\lambda}_n(\bbb{\Sigma}) + \bbb{\lambda}_n((\bbb{X}^t)^{-1}) \leq \bbb{\lambda}_n(\bbb{\Sigma}) + 1/\sqrt{\sigma}$. Solving the quadratic inequality in (\ref{eq:quad:D}) gives $\bbb{\lambda}_n(\bbb{D}) \leq (  \bbb{\lambda}_n(\bbb{\Sigma}) + 1/\sqrt{\sigma})/\sigma =   \bbb{\lambda}_n(\bbb{\Sigma})/\sigma + \sigma^{-3/2} \triangleq C_1$. Therefore, we have:
\beq
0 \prec  (1/\sqrt{L}-{C_1 \alpha^t})  \bbb{I} \preceq \bbb{X}^t - \alpha^t \bbb{\lambda}_n(\bbb{D}^t) \bbb{I} \preceq \bbb{X}^t + \alpha^t \bbb{D}^t. \nn
\eeq
\noi where the first steps uses the fact that $\alpha^t \leq  1/(\sqrt{L} C_1)-\epsilon$; the second step uses $\bbb{X}^t \succeq (1/\sqrt{L}) \cdot \bbb{I}$ and $\bbb{\lambda}_n(\bbb{D}) \leq C_1$; the last step uses $\lambda_{n}(\bbb{D}^t) \bbb{I} \succeq -\bbb{D}^t$.

 Second, we focus on the sufficient decrease condition. For any $\alpha \in (0,1]$, we have:
\beq \label{eq:suff:dec:theorem}
&& \textstyle  f(\bbb{X}^t+\alpha^t \bbb{D}^t) - f(\bbb{X}^t)   \nn\\
& \leq& \textstyle \textstyle \alpha \la \bbb{D}^t, \bbb{G}^t \ra +  \frac{(\alpha^t)^2 L}{2}\|\bbb{D}^t\|_{\text{F}}^2 \nn\\
\textstyle  &\leq& \textstyle \alpha^t \left(\la \bbb{D}^t, \bbb{G}^t \ra +  \frac{\alpha^t   L}{2\sigma } \|\bbb{D}\|_{\bbb{H}^t}^2\right)\nn\\
 \textstyle &\leq& \textstyle \alpha^t \left(\la \bbb{D}^t, \bbb{G}^t \ra -  \frac{\alpha^t  L}{2\sigma } \la \bbb{D}^t, \bbb{G}^t \ra\right)   \nn\\
 \textstyle &=& \textstyle  \alpha^t \la \bbb{D}^t, \bbb{G}^t \ra ( 1 -  \frac{\alpha^t   L}{2\sigma})  \leq \textstyle \alpha^t \la \bbb{D}^t, \bbb{G}^t \ra \cdot \omega,
\eeq
\noi where the first step uses the $L$-Lipschitz continuity of the gradient of $f(\bbb{X})$ that: $\forall \bbb{X},\bbb{Y}\in \Omega,~f(\bbb{Y}) \leq f(\bbb{X})+ \la g(\bbb{X}), \bbb{Y}-\bbb{X}\ra + \frac{L}{2} \|\bbb{X}-\bbb{Y}\|_{\text{F}}^2$; the second step uses the lower bound of the Hessian matrix that $\sigma\|\bbb{D}^t\|_{\text{F}}^2 \leq \|\bbb{D}^t\|_{\bbb{H}^t}^2$; the third step uses (\ref{eq:newton:suff:dec}) that $\|\bbb{D}^t\|_{\bbb{H}^t}^2\leq - \la \bbb{D}^t,\bbb{G}^t\ra$; the last step uses the choice that $\alpha^t \leq {2\sigma (1-\omega)}/{L} \triangleq  C_2$.

Combining the positive definiteness condition, sufficient decrease condition and the fact that $\alpha \in (0,1]$, we finish the proof of the first part of this lemma.

(ii) From (\ref{eq:suff:dec:theorem}) and (\ref{eq:newton:suff:dec}), we have:
\begin{align} \label{eq:conv}
\forall t,~&~f(\bbb{X}^{t+1})  - f(\bbb{X}^{t})\nn\\
\leq& ~\alpha^t  \omega  \la \bbb{D}^t, \bbb{G}^t \ra \leq  - \alpha^t \omega \|\bbb{D}^t\|_{\bbb{H}^t}^2 \nn \\
\leq& ~- \sigma \alpha^t \omega  \|\bbb{D}^t\|_{\text{F}}^2 = - \nu \|\bbb{D}^t\|_{\text{F}}^2,~\text{with} ~\nu \triangleq \sigma \alpha^t \omega >0.
\end{align}
\noi Therefore, the sequence $f(\bbb{X}^t)$ is non-increasing. Summing the inequality in (\ref{eq:conv}) over $i=0,1,...,t-1$ and using the fact that $f(\bbb{X}^*)\leq f(\bbb{X}^{t})$, we have:
\beq
&&f(\bbb{X}^{t}) - f(\bbb{X}^0) \leq - \nu \textstyle\sum_{i=0}^{t-1} \|\bbb{D}^{i}\|_{\text{F}}^2 \nn\\ &\Rightarrow& f(\bbb{X}^{*}) - f(\bbb{X}^0) \leq - \nu \textstyle\sum_{i=0}^{t-1} \|\bbb{D}^{i}\|_{\text{F}}^2\nn\\
 &\Rightarrow& (f(\bbb{X}^0) -f(\bbb{X}^{*})) /  (t\nu)  \geq \min_{i=0,1,...,t-1} \|\bbb{D}^{i}\|_{\text{F}}^2.\nn
\eeq
\noi As $t\rightarrow \infty$, we have $\bbb{D}^t \rightarrow 0$. We further derive the following results: $\bbb{D}^t=\bbb{0}\Rightarrow (\nabla q(\bbb{\bbb{D}}^t))_S = \bbb{0} \Rightarrow  (\bbb{H}^t \circ\bbb{\bbb{D}}^t + \bbb{G}^t )_S = \bbb{0} \Rightarrow  \bbb{G}^t_S = - (\bbb{H}^t \circ \bbb{\Delta})_S = \bbb{0}$. Based on the fact that $\bbb{X}^t \succ 0$, $\bbb{X}^t_Z=\bbb{0}$, and $\bbb{G}^t_S =0$, we conclude that $\bbb{X}^t$ is the global optimal solution for the convex optimization problem. Therefore, any cluster point of the sequence $\bbb{X}^t$ is the global optimal solution.

\end{proof}
\end{theorem}

%Thus, the solution is guaranteed to be in the interior of the constraint set $\bbb{X}\succ 0$.

\textbf{Remarks:} Due to the strongly convexity and gradient Lipschitz continuity of the objective function, there always exists a strictly positive step size $\alpha^t$ such that both sufficient decrease condition and positive definite condition can be satisfied for all $t$. This is very crucial for the global convergence of Algorithm \ref{algo:newton}.

We now prove the global linear convergence rate of Algorithm \ref{algo:newton}. The following lemma characterizes the relation between $d(\bbb{X}^t)$ and $\bbb{X}^*$ for any $\bbb{X}^t$.

\begin{lemma} \label{eq:error:bound}
If $\bbb{X}^t$ is not the global optimal solution of (\ref{eq:general:op}), there exists a constant $\eta \in (0,\infty)$ such that $\|\bbb{X}^t - \bbb{X}^*\|_{\text{F}} \leq \eta \|d(\bbb{X}^t)\|_{\text{F}}$.
\begin{proof}

First, we prove that $d(\bbb{X}^t)\neq \bbb{0}$. This can be achieved by contradiction. Assuming that $d(\bbb{X}^t)=\bbb{0}$, we obtain: $\bbb{X}^{t+1} = \bbb{X}^t + \alpha^t d(\bbb{X}^t)$, which implies that $\bbb{X}^{t+1}=\bbb{X}^t$ is the stationary point. Since (\ref{eq:general:op}) is a strongly convex optimization problem, we have $\bbb{X}^t=\bbb{X}^*$, which contradicts with the condition that $\bbb{X}^t$ is not the optimal solution. Combining with the boundedness of $\bbb{X}^t$ and $\bbb{X}^*$, we conclude that there exists a sufficiently large constant $\eta \in (0,\infty)$ such that $\|\bbb{X}^t - \bbb{X}^*\|_{\text{F}} \leq \eta \|d(\bbb{X}^t)\|_{\text{F}}$.

\end{proof}
\end{lemma}

\begin{theorem}

(Global Linear Convergence Rate). The sequence $\bbb{X}^t$ converges to the global optimal solution linearly.

\begin{proof}

\noi We assume that there exists a point $\bar{\bbb{X}}$ lying on the segment joining $\bbb{X}^{t+1}$ with $\bbb{X}^*$. Using the fact that $\bbb{X}^{t+1}\in\Omega,~\bbb{X}^{*}\in\Omega$ and the Mean Value Theorem, we derive the following results:
\beq\label{eq:linear:conv}
&&F(\bbb{X}^{t+1}) - F(\bbb{X}^*)=f(\bbb{X}^{t+1}) - f(\bbb{X}^*)\nn\\
&=& \la g(\bar{\bbb{X}}) , \bbb{X}^{t+1}-\bbb{X}^*\ra  \nn\\
&=&  \la \bbb{G}^t +  \bbb{H}^t \circ \bbb{D}^t    ,\bbb{X}^{t+1}-\bbb{X}^* \ra \nn\\
&&+ \la g(\bar{\bbb{X}}) - \bbb{G}^t -   \bbb{H}^t\circ \bbb{D}^t , \bbb{X}^{t+1}-\bbb{X}^*\ra.
\eeq

(i) We now consider to bound the first term in (\ref{eq:linear:conv}).
Since $\bbb{D}^t$ is the optimal solution of (\ref{eq:newton:direction}), we have
\beq \label{eq:linear:conv:1}
\bbb{0} \in \bbb{G}^t +  \bbb{H}^t \circ \bbb{D}^t  + \partial p(\bbb{X}^t+\bbb{D}^t).
\eeq
\noi Due to the convexity of $p(\cdot)$, we obtain: $\la \bbb{X}-\bbb{Z},~ \partial p(\bbb{Z})\ra \leq p(\bbb{X})-p(\bbb{Z})$. Letting $\bbb{X}=\bbb{X}^*$ and $\bbb{Z}=\bbb{X}^t+\bbb{D}^t$, we have:
\beq \label{eq:linear:conv:2}
\la \bbb{X}^*-\bbb{X}^t-\bbb{D}^t,~ \partial p(\bbb{X}^t+\bbb{D}^t)\ra \leq p(\bbb{X}^*)-p(\bbb{X}^t+\bbb{D}^t).
\eeq

\noi Combing (\ref{eq:linear:conv:1}) and (\ref{eq:linear:conv:2}), we have: $\la \bbb{X}^*-\bbb{X}^t-\bbb{D}^t,~ -\bbb{G}^t - \circ\bbb{H}^t \circ \bbb{D}^t \ra \leq p(\bbb{X}^*)-p(\bbb{X}^t+\bbb{D}^t)$. Noticing $\bbb{X}^*\in \Omega$ and $\bbb{X}^t+\bbb{D}^t\in \Omega$, we have $p(\bbb{X}^*)=p(\bbb{X}^t+\bbb{D}^t)=0$. We reach the following inequality:
\beq
\la \bbb{X}^t+\bbb{D}^t-\bbb{X}^* ,~ \bbb{G}^t +  \bbb{H}^t \circ \bbb{D}^t \ra \leq 0.\nn
\eeq

\noi Therefore, the first term in (\ref{eq:linear:conv}) is bounded by:
\beq \label{eq:linear:conv:0}
&&\la \bbb{G}^t + \bbb{H}^t \circ \bbb{D}^t  ,\bbb{X}^{t+1}-\bbb{X}^* \ra \nn\\
&=&   \la \bbb{G}^t + \bbb{H}^t \circ \bbb{D}^t ,(\alpha^t-1) \bbb{D}^t +\bbb{X}^t+\bbb{D}^t-\bbb{X}^* \ra\nn\\
&\leq &   \la \bbb{G}^t +\bbb{H}^t \circ \bbb{D}^t , (\alpha^t-1)\bbb{D}^t \ra \nn \\
&\leq &(\alpha^t-1)   \la \bbb{G}^t,\bbb{D}^t \ra +0\nn\\
&\leq & \tfrac{1-\alpha^t}{\alpha^t \omega} (f(\bbb{X}^{t}) - f(\bbb{X}^{t+1})).
\eeq

(ii) We now consider to bound the second term in (\ref{eq:linear:conv}). We derive the following results:
\begin{align} \label{eq:linear:conv:00}
&~\la g(\bar{\bbb{X}}) - \bbb{G}^t -  \bbb{H}^t \circ \bbb{D}^t , \bbb{X}^{t+1}-\bbb{X}^*\ra \nn\\
 \leq&~ \|\bbb{X}^{t+1}-\bbb{X}^*\|_{\text{F}} \| (\|g(\bar{\bbb{X}}) - \bbb{G}^t\|_{\text{F}} + \|\bbb{H}^t \circ \bbb{D}^t\|_{\text{F}})    \nn\\
 \leq&~ \|\bbb{X}^{t+1}-\bbb{X}^*\|_{\text{F}} (L\|\bar{\bbb{X}} - \bbb{X}^t\|_{\text{F}}+\|\bbb{H}^t \circ \bbb{D}^t\|_{\text{F}})    \nn\\
 \leq&~ (\eta+\alpha^t) \|\bbb{D}^t\|_{\text{F}} (L\|\bar{\bbb{X}} - \bbb{X}^t\|_{\text{F}}+\|\bbb{H}^t \circ \bbb{D}^t\|_{\text{F}})     \nn\\
\leq & ~L (\eta+\alpha^t) \|\bbb{D}^t\|_{\text{F}}(\|\bar{\bbb{X}} - \bbb{X}^t\|_{\text{F}}+\|\bbb{D}^t\|_{\text{F}})     \nn\\
\leq & ~L  (\eta+\alpha^t)(\| \bar{\bbb{X}} - \bbb{X}^{*}\|_{\text{F}}+\| \bbb{X}^{*} - \bbb{X}^t\|_{\text{F}}+\|\bbb{D}^t\|_{\text{F}})   \|\bbb{D}^t\|_{\text{F}}  \nn\\
\leq & ~L (\eta+\alpha^t)(\|\bbb{X}^{t+1}-\bbb{X}^* \|_{\text{F}}+ \eta \|\bbb{D}^t\|_{\text{F}}+\|\bbb{D}^t\|_{\text{F}})   \|\bbb{D}^t\|_{\text{F}}   \nn\\
\leq & ~L (\eta+\alpha^t)((\eta+\alpha^t)\|\bbb{D}^t\|_{\text{F}} + \eta \|\bbb{D}^t\|_{\text{F}} +\|\bbb{D}^t\|_{\text{F}})  \|\bbb{D}^t\|_{\text{F}}   \nn\\
= &~ L(\eta+\alpha^t)(2\eta+\alpha^t+1) \|\bbb{D}^t\|^2_{\text{F}}   \nn\\
\leq &~ ({L(\eta+\alpha^t)(2\eta+\alpha^t+1)}/{\nu})  (f(\bbb{X}^t) - f(\bbb{X}^{t+1})) \nn\\
\leq &~ ({L(\eta+\alpha^t)(2\eta+\alpha^t+1)}/{\sigma \alpha^t \omega})  (f(\bbb{X}^t) - f(\bbb{X}^{t+1})),
\end{align}
\noi where the first step uses Cauchy-Schwarz inequality that $\forall \bbb{A},~\bbb{B},~\la \bbb{A},\bbb{B}\ra\leq \|\bbb{A}\|_{\text{F}}\|\bbb{B}\|_{\text{F}}$ and the triangle inequality that $\forall \bbb{A},~\bbb{B},~\|\bbb{A}+\bbb{B}\|_{\text{F}}\leq \|\bbb{A}\|_{\text{F}}+\|\bbb{B}\|_{\text{F}}$; the second step uses the gradient Lipschitz continouity of $f(\cdot)$ that $\|g(\bar{\bbb{X}}) - g(\bbb{X}^t) \|_{\text{F}}\leq L\|\bar{\bbb{X}} - \bbb{X}^t\|_{\text{F}}$; the third step uses Lemma \ref{eq:error:bound} that $\|\bbb{X}^*-\bbb{X}^t\|_{\text{F}}\leq \eta \|\bbb{D}^t\|_{\text{F}}$ and $\|\bbb{X}^{t+1}-\bbb{X}^*\|_{\text{F}} = \|\bbb{X}^{t}+\alpha^t\bbb{D}^t-\bbb{X}^*\|_{\text{F}} \leq \|\bbb{X}^{t}-\bbb{X}^*\|_{\text{F}} + \alpha^t\|\bbb{D}^t\|_{\text{F}} \leq (\eta+\alpha^t) \|\bbb{D}^t\|_{\text{F}}$; the fourth step uses the inequality that $\|\bbb{H}^t \circ \bbb{D}^t\|_{\text{F}} \leq L \|\bbb{D}^t\|_{\text{F}}$; the fifth step uses the triangle inequality that $\|\bar{\bbb{X}} - \bbb{X}^t\|_{\text{F}} = \|\bar{\bbb{X}} - \bbb{X}^* + \bbb{X}^* - \bbb{X}^t\|_{\text{F}}\leq \| \bar{\bbb{X}} - \bbb{X}^{*}\|_{\text{F}}+\| \bbb{X}^{*} - \bbb{X}^t\|_{\text{F}}$; the sixth step uses the fact that $\|\bar{\bbb{X}} - \bbb{X}^t\|_{\text{F}}  \leq \|\bbb{X}^{t+1}-\bbb{X}^* \|_{\text{F}}$ since $\bar{\bbb{X}}$ is a point lying on the segment joining $\bbb{X}^{t+1}$ and $\bbb{X}^{*}$; the ninth step uses (\ref{eq:conv}); the last step uses the definition of $\mu$.

We define $C^t\triangleq \tfrac{1-\alpha^t}{\alpha^t \omega} +\frac{L(\eta+\alpha^t)(2\eta+\alpha^t+1)}{\sigma \alpha^t \omega}$. Clearly, $C^t \leq  \tfrac{1}{\alpha^t \omega} +\frac{2L(\eta+1)^2/\sigma}{ \alpha^t \omega}  \leq \frac{1+2L(\eta+1)^2/\sigma}{ \min(\alpha^1,\alpha^2,...,\alpha^{\infty}) \omega} \triangleq C$. Combining (\ref{eq:linear:conv:0}), (\ref{eq:linear:conv:00}) and (\ref{eq:linear:conv}), we have the following results: $f(\bbb{X}^{t+1}) - f(\bbb{X}^*) \leq \textstyle C ( f(\bbb{X}^t) - f(\bbb{X}^{t+1}) )=  C( f(\bbb{X}^t)  - f(\bbb{X}^*)) - C ( f(\bbb{X}^{t+1}) - f(\bbb{X}^*))$. Finally, we obtain:
\beq
\textstyle \frac{f(\bbb{X}^{t+1}) - f(\bbb{X}^*) }{ f(\bbb{X}^t) -f(\bbb{X}^*)} \leq \frac{C}{C+1}.\nn
\eeq
Therefore, $f(\bbb{X}^t)$ converges to $f(\bbb{X}^*)$ at least Q-linearly. In addition, from (\ref{eq:conv}), we have: $\|\bbb{X}^{t+1} - \bbb{X}^t\|_{\text{F}}^2  \leq (f(\bbb{X}^t) - f(\bbb{X}^{t+1}))/\nu$. Since $f(\bbb{X}^{t+1}) - f(\bbb{X}^*)$ converges to 0 at least R-linearly, this implies that $\|\bbb{X}^{t+1} - \bbb{X}^t\|_{\text{F}}^2$ converges at least R-linearly as well. We thus complete the proof of this lemma.%By Eq(\ref{eq:dec:theorem}), $\{\|\bbb{X}^t-\bbb{X}^{k+1}\|\}$ converges to 0 at least R-linearly.

\end{proof}
\end{theorem}

\textbf{Remarks:} Global linear convergence rate for convex composite optimization has been extensively studied in Ref. \cite{tseng2009coordinate}. This work extends their analysis to deal with our specific matrix optimization problem which involves an additional positive definite constraint $\bbb{X}\succ \bbb{0}$.

%is defined in a more restricted domain 

%The following lemma are useful in our proof of convergence
%
%\begin{lemma}
%
%If $f(\bbb{X})$ is a standard self-concordant function, we have the following inequalities:
%\beq
%\label{eq:con:pro1} \|{G}(\bbb{Y})-G(\bbb{X})-{H}(\bbb{X})(\bbb{Y}-\bbb{X})\|_{{H}(\bbb{X})} \leq \tfrac{r^2}{1- r}~~~~~~\\
%\label{eq:con:pro2} f(\bbb{Y}) - f(\bbb{X}) - \la {G}(\bbb{X}) ,\bbb{Y}-\bbb{X} \ra \leq \varphi(\|\bbb{Y}-\bbb{X}\|_{{H}(\bbb{X})}).
%\eeq
%\noi $\forall \bbb{X}, \bbb{Y} \in \mathcal{X}$, $r  \triangleq \|\bbb{X}-\bbb{Y}\|_{{H}(\bbb{X})} < 1$,~$\varphi(t) \triangleq -t-\ln(1-t)$.
%
%\begin{proof}
%Refer to Lemma 1 in \cite{Nesterov12} and Theorem 4.1.8 in \cite{Nesterov03}.
%\end{proof}
%
%\end{lemma}

\begin{theorem}

(Local Quadratic Convergence Rate). A full Newton step size with $\alpha^t=1$ will be selected when $\bbb{X}^t$ is close enough to global optimal solution that $\|\bbb{D}^t\|_{\text{F}}\leq \min(1/\sqrt{L}-\epsilon,0.81/\sigma)$. Here $\epsilon$ denotes a sufficient small positive constant. In addition, when $\|\bbb{X}^t-\bbb{X}^*\|_{\text{F}}\leq 1/\sigma$, the sequence $\{\bbb{X}^t\}$ converges to the global optimal solution quadratically.

\begin{proof}
(i) First, we consider the positive definiteness condition (refer to step \ref{alg:step:linesearch:end} in Algorithm \ref{algo:newton}). We derive the following results:
\beq
\bbb{0}& \prec& (\tfrac{1}{\sqrt{L}} -  (\tfrac{1}{\sqrt{L}}-\epsilon))\bbb{I}\nn\\
&\preceq& \bbb{X}^t - \|\bbb{D}^t\|_{\text{F}}\cdot\bbb{I} \preceq \bbb{X}^t - \lambda_n(\bbb{D}^t)\bbb{I} \preceq \bbb{X}^t + \bbb{D}^t,\nn
\eeq
\noi where the first step uses $\tfrac{1}{\sqrt{L}} -  (\tfrac{1}{\sqrt{L}}-\epsilon)>0$; the second step uses $\bbb{X}^t \succeq \frac{1}{\sqrt{L}}\bbb{I}$ and $\|\bbb{D}^t\|_{\text{F}}\leq 1/\sqrt{L}-\epsilon$; the third step uses the inequality $\bbb{\lambda}_n(\bbb{D}^t)\leq\|\bbb{D}^t\|_{\text{F}}$; the last step uses $-\lambda_n(\bbb{D}^t)\bbb{I} \preceq \bbb{D}^t$ and $\alpha^t=1$. Therefore, the positive definite condition is satisfied for $\alpha^t=1$ when
\beq \label{eq:newton:step:1}
\|\bbb{D}^t\|_{\text{F}}\leq 1/\sqrt{L}-\epsilon.
\eeq
Second, we consider the sufficient decrease condition (refer to step \ref{alg:step:linesearch:end2} in Algorithm \ref{algo:newton}). Since $f(\bbb{X})$ is a standard self-concordant function, it holds that (refer to Theorem 4.1.8 in \cite{Nesterov03}):
\beq %\label{eq:self:con}
\label{eq:con:pro2} f(\bbb{Y}) - f(\bbb{X}) - \la g(\bbb{X}) ,\bbb{Y}-\bbb{X} \ra \leq \varphi(\|\bbb{Y}-\bbb{X}\|_{h(\bbb{X})})\nn
\eeq
\noi with $\varphi(t) \triangleq -t-\ln(1-t)$. Applying this inequality with $\bbb{Y}=\bbb{X}^{t+1},~\bbb{X}=\bbb{X}^t$ and using the update rule that $\bbb{X}^{t+1}=\bbb{X}^t+\alpha^t \bbb{D}^t$, we have the following inequalities:
\begin{align*}% \label{eq:unit:step}
&~f(\bbb{X}^{t+1}) \nn\\
 \leq & ~f(\bbb{X}^t)  + \alpha^t \la {\bbb{G}}^t , \bbb{D}^t \ra + \varphi( \alpha^t \|\bbb{D}^t\|_{{H}^t})\nn\\
 \leq & ~f(\bbb{X}^t)  + \alpha^t \la {\bbb{G}}^t , \bbb{D}^t \ra +  \tfrac{(\alpha^t)^2}{2}  \|\bbb{D}^t\|_{{\bbb{H}}^t}^2 + (\alpha^t)^3 \|\bbb{D}^t\|_{{\bbb{H}}^t}^3 \nn \\
\leq&~f(\bbb{X}^t) + \la \bbb{G}^t , \bbb{D}^t \ra + \tfrac{1}{2}\|\bbb{D}^t\|_{\bbb{H}^t}^2+ \|\bbb{D}^t\|_{\bbb{H}^t}^3  \nn\\
\leq & ~f (\bbb{X}^t) + \la \bbb{G}^t , \bbb{D}^t \ra -  \tfrac{1}{2} \la \bbb{G}^t , \bbb{D}^t \ra +  (-\la \bbb{G}^t , \bbb{D}^t \ra) ^{3/2}  \nn\\
=& ~f (\bbb{X}^t)  + \la \bbb{G}^t , \bbb{D}^t \ra  ( \tfrac{1}{2} - \sqrt{-\la  \bbb{G}^t , \bbb{D}^t \ra} )  \nn\\
\leq &~ f (\bbb{X}^t) +  \omega \la \bbb{G}^t , \bbb{D}^t \ra \tfrac{1}{\omega} ,
\end{align*}
\noi where the second step uses the fact that $-z-\ln(1-z)\leq \frac{1}{2}z^2+z^3$ for $0\leq z\leq 0.81$ (see Section 9.6 in \cite{Boyd04}) and the inequality that $z\triangleq \alpha^t \|\bbb{D}^t\|_{{\bbb{H}}^t} \leq 0.81$, where the latter is true since
\beq \label{eq:newton:step:2}
\| \bbb{D}^t \|_{\text{F}} \leq {0.81}/{\sigma} \Rightarrow \|\bbb{D}^t\|_{{\bbb{H}}^t} \leq 0.81
\eeq
\noi and $\alpha^t\leq 1$;  the third step uses the choice $\alpha^t=1$; the fourth step uses (\ref{eq:newton:suff:dec}); the last step uses the inequality that $\tfrac{1}{2} - \sqrt{-\la  \bbb{G}^t , ~\bbb{D}^t \ra} \leq  \tfrac{1}{\omega}$, which is clearly holds since $\omega<1/2$.

Combining (\ref{eq:newton:step:1}) and (\ref{eq:newton:step:2}), we conclude that the full Newton step size will be achieved when $\|\bbb{D}^t\|_{\text{F}}\leq \min(1/\sqrt{L}-\epsilon,0.81/\sigma)$.

(ii) We now prove the second part of this theorem. Recall that when $f(\cdot)$ is a standard self-concordant function, it holds that (refer to Lemma 1 in \cite{Nesterov12}):
\beq
\label{eq:con:pro1} \|g(\bbb{Y})-g(\bbb{X})-h(\bbb{X})(\bbb{Y}-\bbb{X})\|_{h(\bbb{X})} \leq \tfrac{r^2}{1- r}
\eeq

\noi for all $r  \triangleq \|\bbb{X}-\bbb{Y}\|_{h(\bbb{X})} < 1$. We define the generalized proximal operator as:
\beq
\prox_p^{\bbb{N}}(\bbb{X}) \triangleq  \arg \min_{\bbb{Y}}~\tfrac{1}{2}\|\bbb{Y}-\bbb{X}\|_{\bbb{N}}^2 + p(\bbb{Y}). \nn
\eeq
\noi Thus, $\bbb{D}^t$ can be represented as:
\beq
\bbb{D}^t &=& \arg \min_{\bbb{\Delta}}~ \la \bbb{\Delta}, \bbb{G}^t \ra + \tfrac{1}{2}\|\bbb{\Delta}\|_{\bbb{H}^t}^2 + p(\bbb{X}^t +\bbb{\Delta} ) \nn\\
&=&\arg \min_{\bbb{\Delta}}~ \tfrac{1}{2}\|\bbb{\Delta} + (\bbb{H}^t)^{-1}\circ\bbb{G}^t \|_{\bbb{H}^t}^2 + p(\bbb{X}^t +\bbb{\Delta}) \nn\\
& = & \prox_p^{\bbb{H}^t}(\bbb{X}^t - (\bbb{H}^t)^{-1} \circ \bbb{G}^t ) - \bbb{X}^t.\nn
\eeq
\noi We derive the following inequalities:
\begin{align} \label{eq:updaterulaX}
&~\|\bbb{X}^{t+1}-\bbb{X}^*\|_{{\bbb{H}}^t} \nn\\
=&~ \|\bbb{X}^t + \bbb{D}^t - \bbb{X}^*\|_{{\bbb{H}}^t}\nn\\
=&~ \|\prox_p^{\bbb{H}^t}(\bbb{X}^t-(\bbb{H}^t)^{-1}\circ \bbb{G^t}) - \bbb{X}^*\|_{{\bbb{H}}^t}\nn\\
=&~ \| \prox_p^{{\bbb{H}}^t}(\bbb{X}^t- (\bbb{H}^t)^{-1} \circ \bbb{G}^t) - \nn\\
&~~~\prox_p^{{\bbb{H}}^t}(\bbb{X}^*- (\bbb{H}^t)^{-1} \circ g(\bbb{X}^*))\|_{{\bbb{H}}^t} \nn \\
\leq&~ \|\bbb{X}^t - \bbb{X}^* + ({\bbb{H}}^t)^{-1} \circ (\bbb{G}^*-\bbb{G}^t) \|_{{\bbb{H}}^t}\nn\\
=&~ \| ({\bbb{H}}^t)^{-1} \circ ({\bbb{H}}^t \circ \left( \bbb{X}^t - \bbb{X}^* +  ({\bbb{H}}^t)^{-1} (\bbb{G}^*-\bbb{G}^t) \right)) \|_{{\bbb{H}}^t}\nn\\
=&~ \| ({\bbb{H}}^t)^{-1} \circ ( {\bbb{H}}^t \circ \left( \bbb{X}^t - \bbb{X}^*\right) +  (\bbb{G}^*-\bbb{G}^t)) \|_{{\bbb{H}}^t}\nn\\
\leq&~ \| ({\bbb{H}}^t)^{-1} \circ \bbb{I} \|_{\bbb{H}^t} \cdot \|( {\bbb{H}}^t \circ \left( \bbb{X}^t - \bbb{X}^*\right) +  (\bbb{G}^*-\bbb{G}^t)) \|_{{\bbb{H}}^t}\nn\\
\leq&~ \tfrac{1}{\sigma} \cdot \| {\bbb{H}}^t  \left( \bbb{X}^t - \bbb{X}^* + ({\bbb{H}}^t)^{-1} (\bbb{G}^*-\bbb{G}^t) \right) \|_{{\bbb{H}}^t}\nn\\
=&~  \tfrac{1}{\sigma} \| {\bbb{H}}^t (\bbb{X}^t - \bbb{X}^*) -\bbb{G}^t + \bbb{G}^* \|_{{\bbb{H}}^t}\nn\\
\leq&~   \frac{ \tfrac{1}{\sigma} \|\bbb{X}^t - \bbb{X}^*\|_{{\bbb{H}}^t}^2 }{     1 - \|\bbb{X}^t - \bbb{X}^*\|_{{\bbb{H}}^t}  },\nn
\end{align}
\noi where the third step uses the fact that $\prox_p^{{\bbb{H}}^t}(\bbb{X}^*- (\bbb{H}^t)^{-1} \circ g(\bbb{X}^*))\|_{{\bbb{H}}^t}=\bbb{X}^*$; the fourth step uses the fact that the generalized proximal mappings are firmly non-expansive in the generalized vector norm; the seventh step uses the Cauchy-Schwarz inequality; the eighth step uses the fact that $\|({\bbb{H}}^t)^{-1} \circ \bbb{I} \|_{{\bbb{H}}^t} \leq \frac{1}{\sigma}$ with $\bbb{I}$ being an identity matrix of dimension $n$; the last step uses (\ref{eq:con:pro1}).

\noi In particular, when $\|\bbb{X}-\bbb{Y}\|_{h(\bbb{X})} < 1$~($\Leftrightarrow \|\bbb{X}-\bbb{Y}\|_{\text{F}} < 1/\sigma$), we have:
\beq
&&\|\bbb{X}^{t+1} - \bbb{X}^*\|_{{\bbb{H}}^t} \leq \tfrac{1}{\sigma}\|\bbb{X}^t - \bbb{X}^*\|_{{\bbb{H}}^t}^2\nn \\
\Rightarrow &&\|\bbb{X}^{t+1} - \bbb{X}^*\|_{\text{F}} \leq \tfrac{L}{\sigma^2}\|\bbb{X}^t - \bbb{X}^*\|_{\text{F}}^2.\nn
\eeq

\noi In other words, NLOA converges to the global optimal solution $\bbb{X}^*$ with asymptotic quadratic convergence rate.
\end{proof}
\end{theorem}

\textbf{Remarks:} We are also aware of the work of \cite{HsiehSDR14}, which shows the quadratic convergence rate for their Newton-like method. However, their work focus on a different convex $\ell_1$ norm regularized sparse inverse covariance selection problem and their results are not applicable to our problem. In addition, their work is based on the assumption that the objective is Hessian Lipschitz continuous, while ours is based on the self-concordant analysis \cite{Nesterov03,Boyd04} of the objective function.% $f(\bbb{X})$ is self-concordant.% property of the objective function. % and our bound is affine-invariant and does not depend on any unknown constants.

\begin{figure*} [!t] 
%\captionsetup{singlelinecheck = true, justification=justified}
\captionsetup{singlelinecheck = on, format= hang, justification=justified, font=footnotesize, labelsep=space}
\centering
      \begin{subfigure}{0.24\textwidth}\includegraphics[width=1\textwidth,height=\hone]{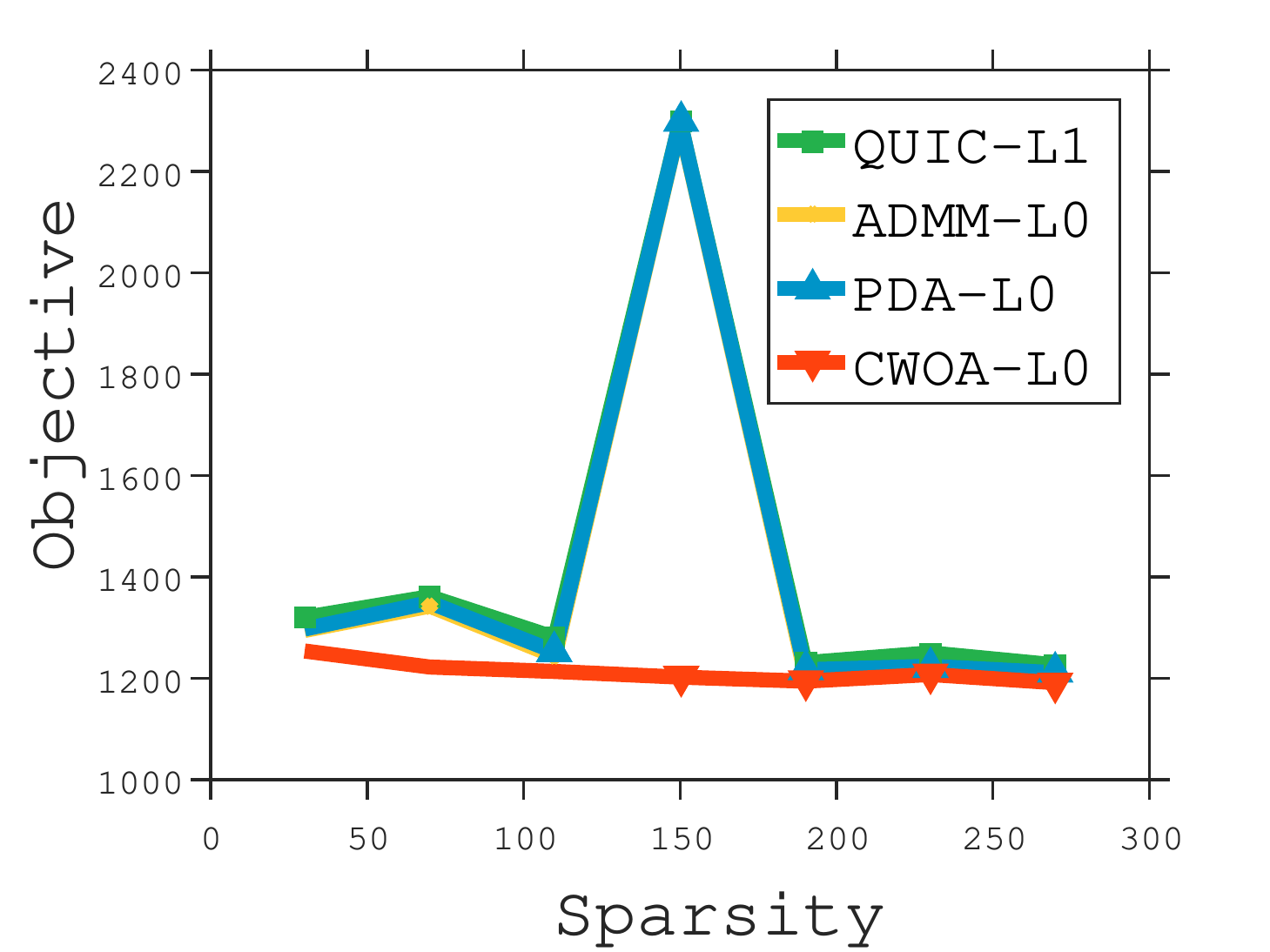}\vspace{-6pt}\caption{\scriptsize Gaussian-Random-500}
      \end{subfigure}
      \begin{subfigure}{0.24\textwidth}\includegraphics[width=1\textwidth,height=\hone]{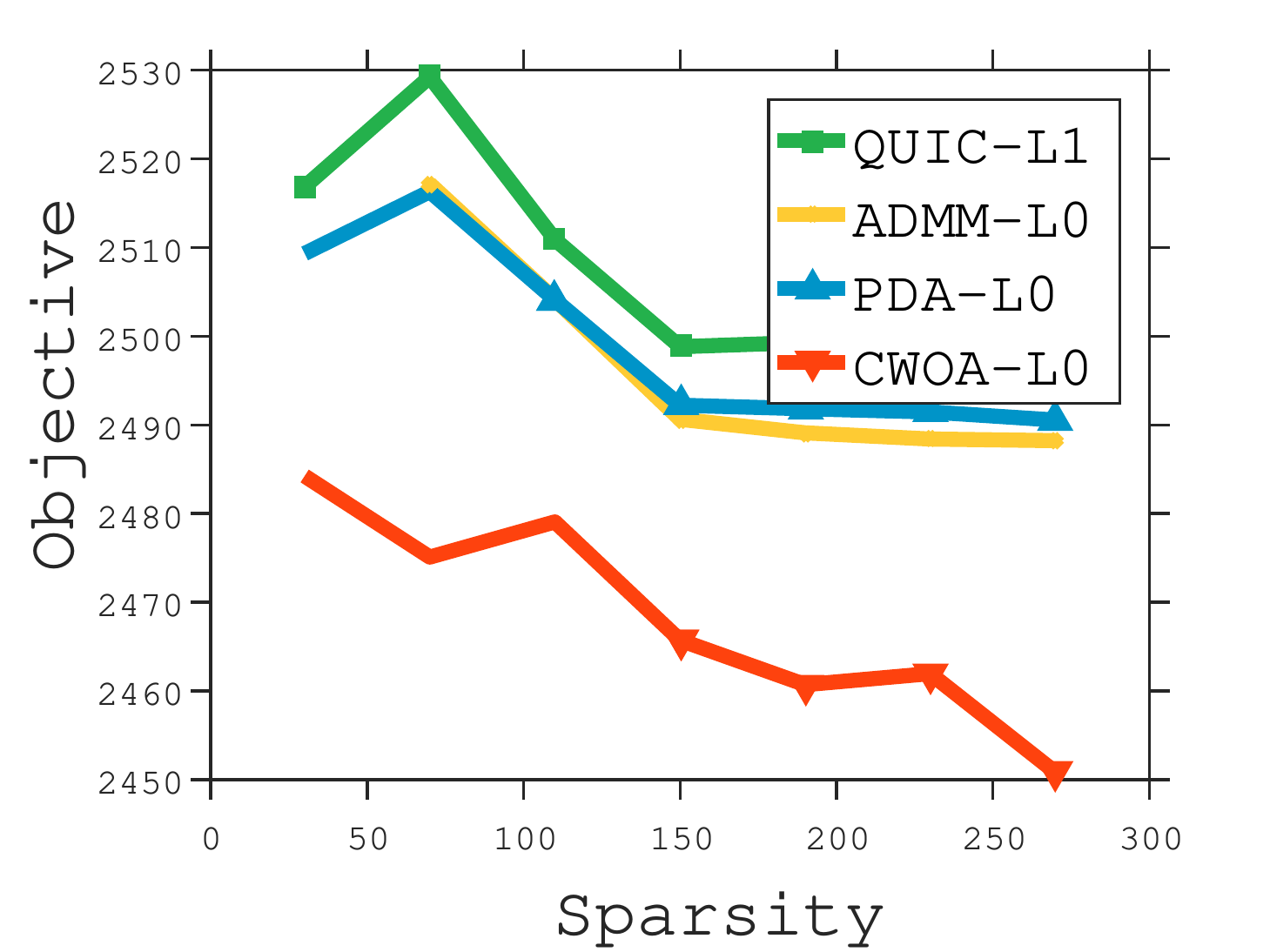}\vspace{-6pt}\caption{\scriptsize Gaussian-Random-1000}
      \end{subfigure}
      \begin{subfigure}{0.24\textwidth}\includegraphics[width=1\textwidth,height=\hone]{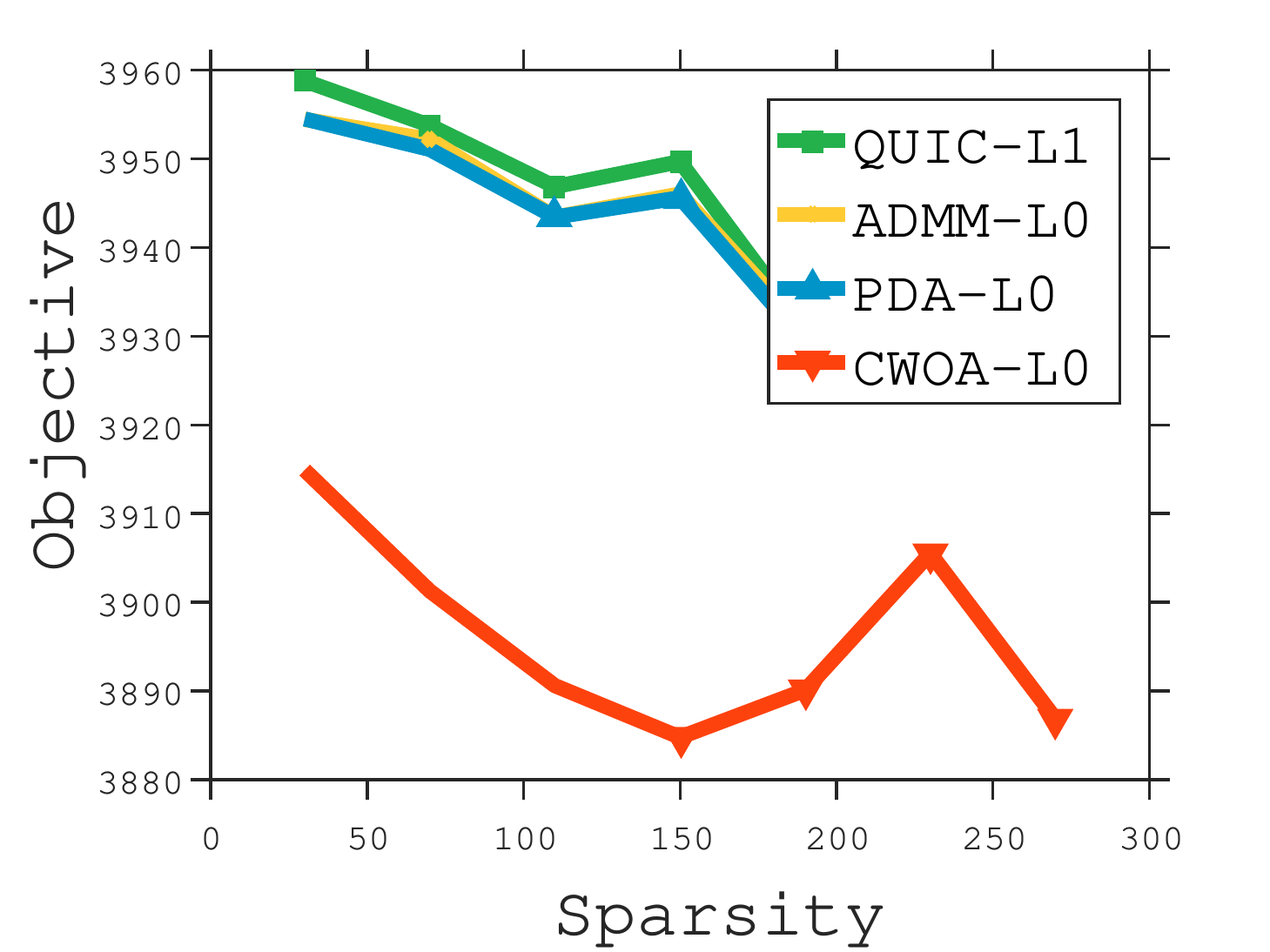}\vspace{-6pt}\caption{\scriptsize Gaussian-Random-1500}
      \end{subfigure}
      \begin{subfigure}{0.24\textwidth}\includegraphics[width=1\textwidth,height=\hone]{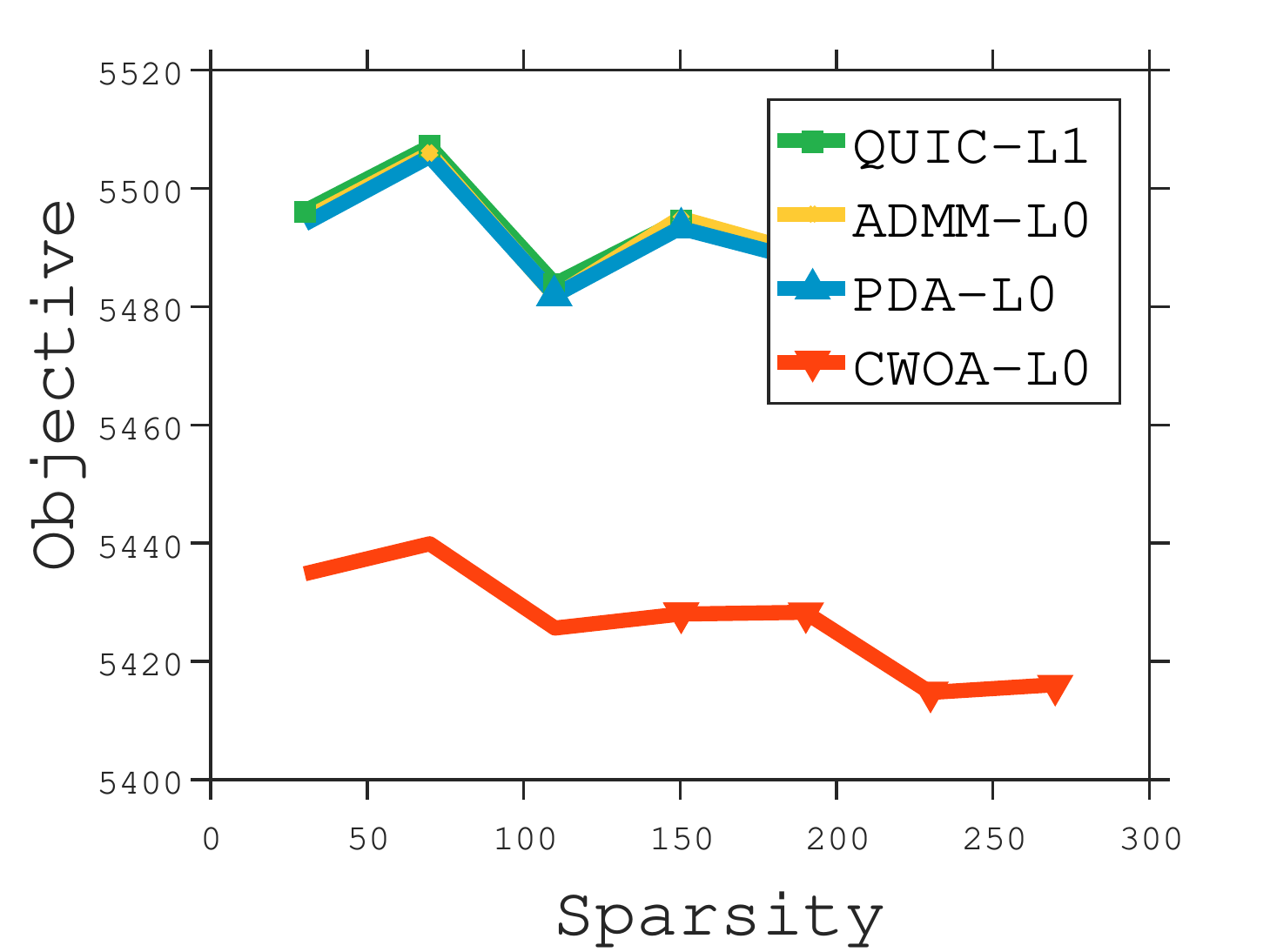}\vspace{-6pt}\caption{\scriptsize Gaussian-Random-2000}
      \end{subfigure}

   %   \vspace{-6pt}

      \begin{subfigure}{0.24\textwidth}\includegraphics[width=1\textwidth,height=\hone]{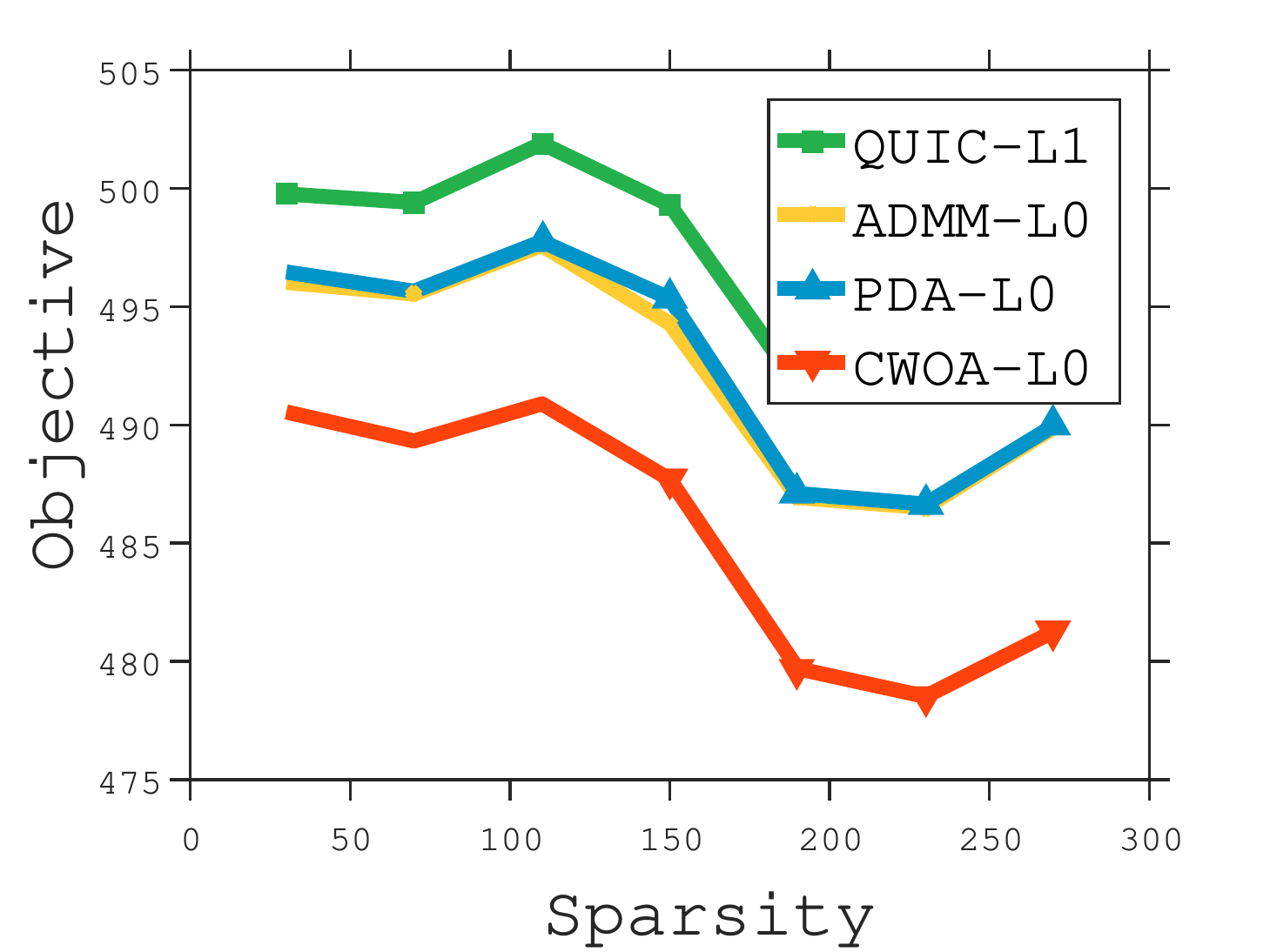}\vspace{-6pt}\caption{\scriptsize Sparse-Structured-500}
      \end{subfigure}
      \begin{subfigure}{0.24\textwidth}\includegraphics[width=1\textwidth,height=\hone]{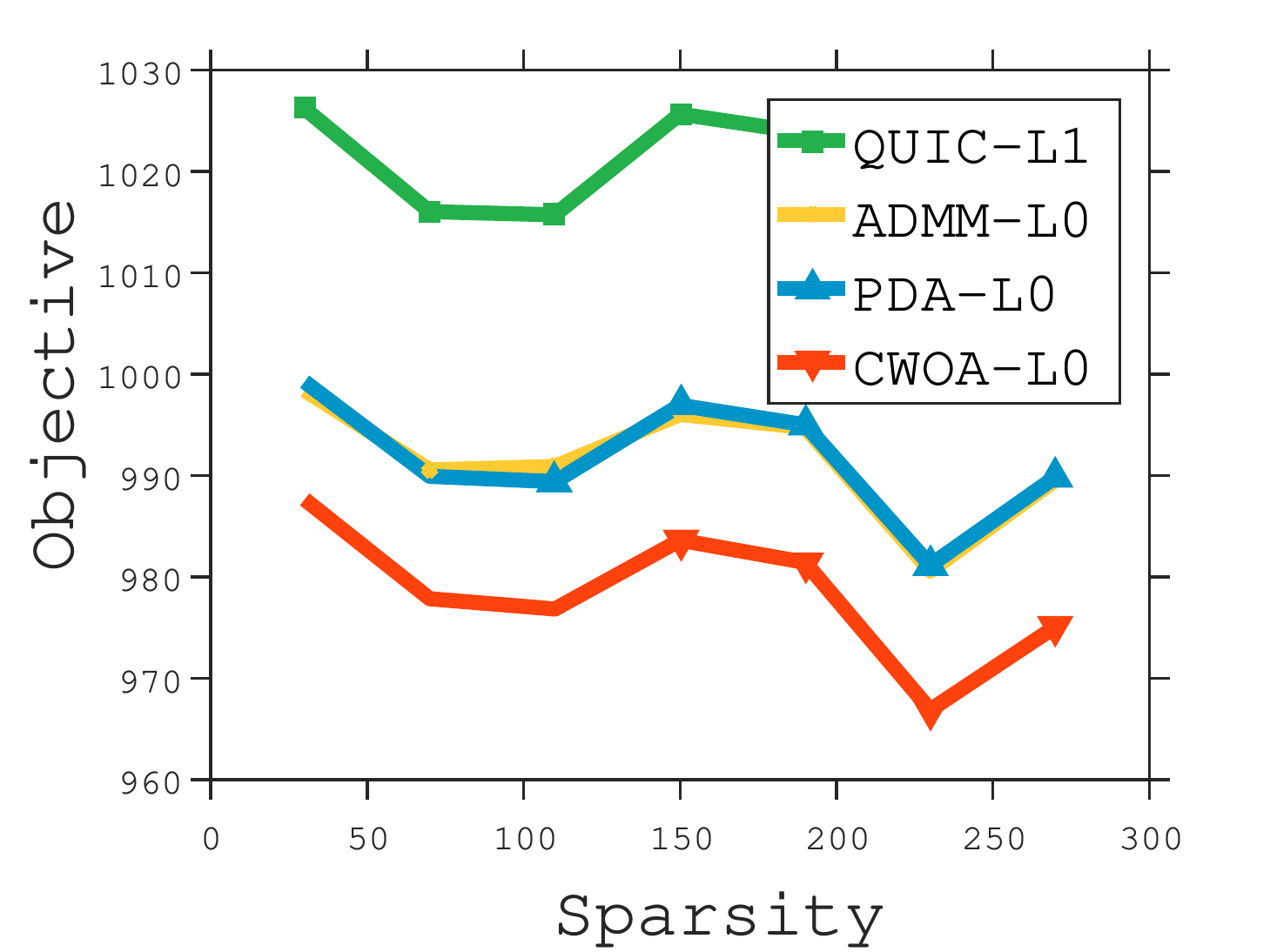}\vspace{-6pt}\caption{\scriptsize Sparse-Structured-1000}
      \end{subfigure}
      \begin{subfigure}{0.24\textwidth}\includegraphics[width=1\textwidth,height=\hone]{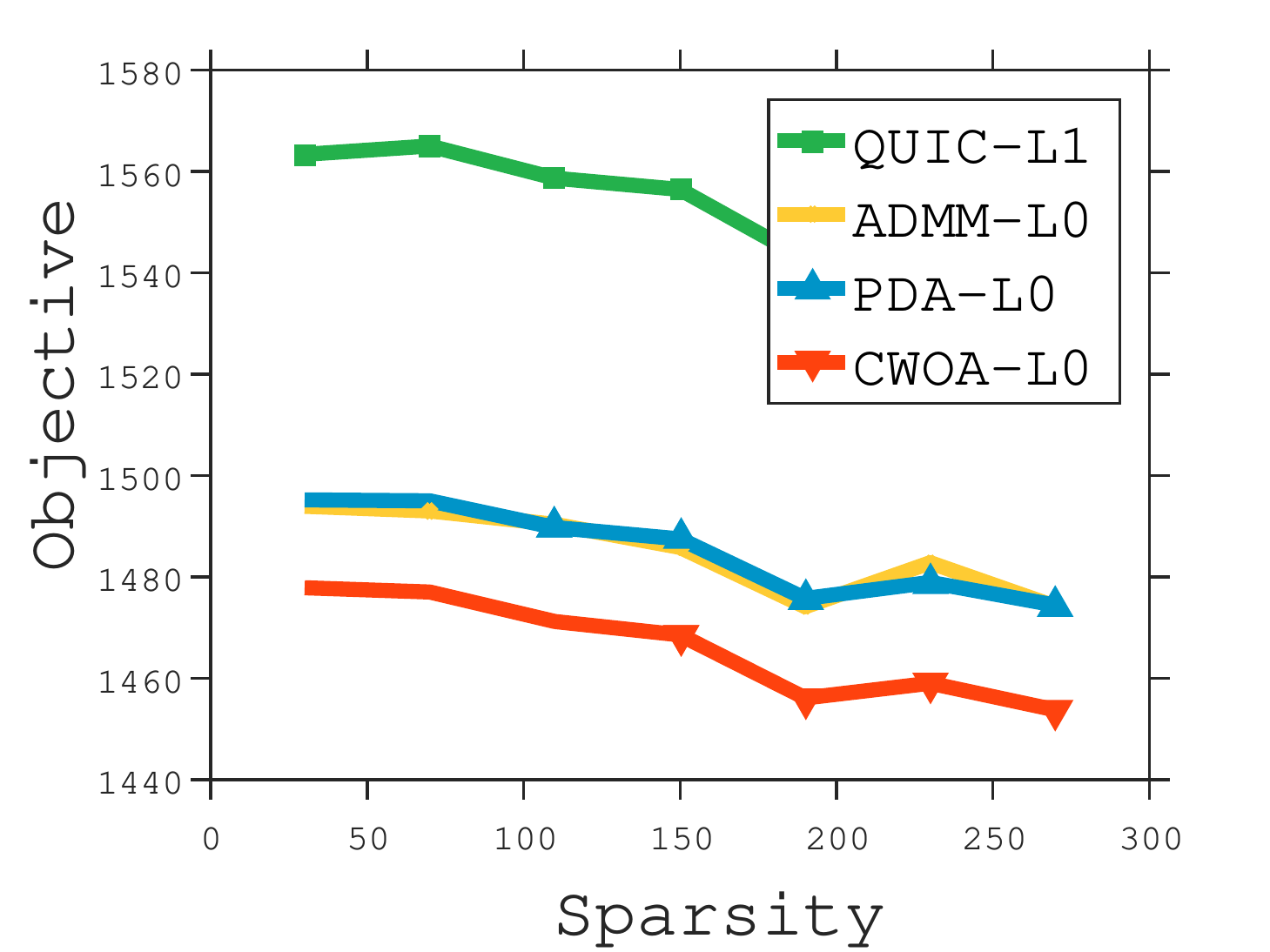}\vspace{-6pt}\caption{\scriptsize Sparse-Structured-1500}
      \end{subfigure}
      \begin{subfigure}{0.24\textwidth}\includegraphics[width=1\textwidth,height=\hone]{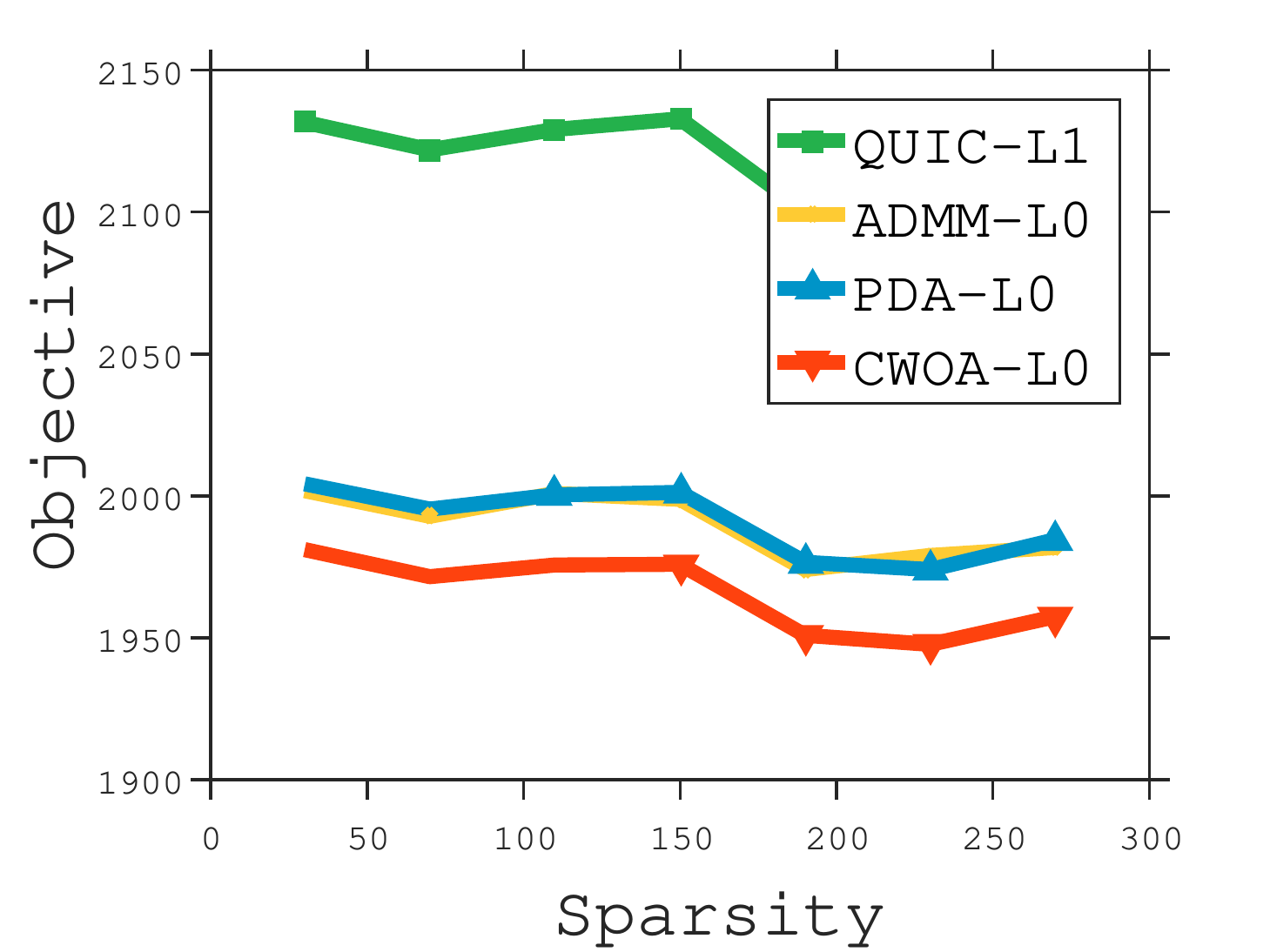}\vspace{-6pt}\caption{\scriptsize Sparse-Structured-2000}
      \end{subfigure}

     % \vspace{-6pt}', `', `', `'

      \begin{subfigure}{0.24\textwidth}\includegraphics[width=1\textwidth,height=\hone]{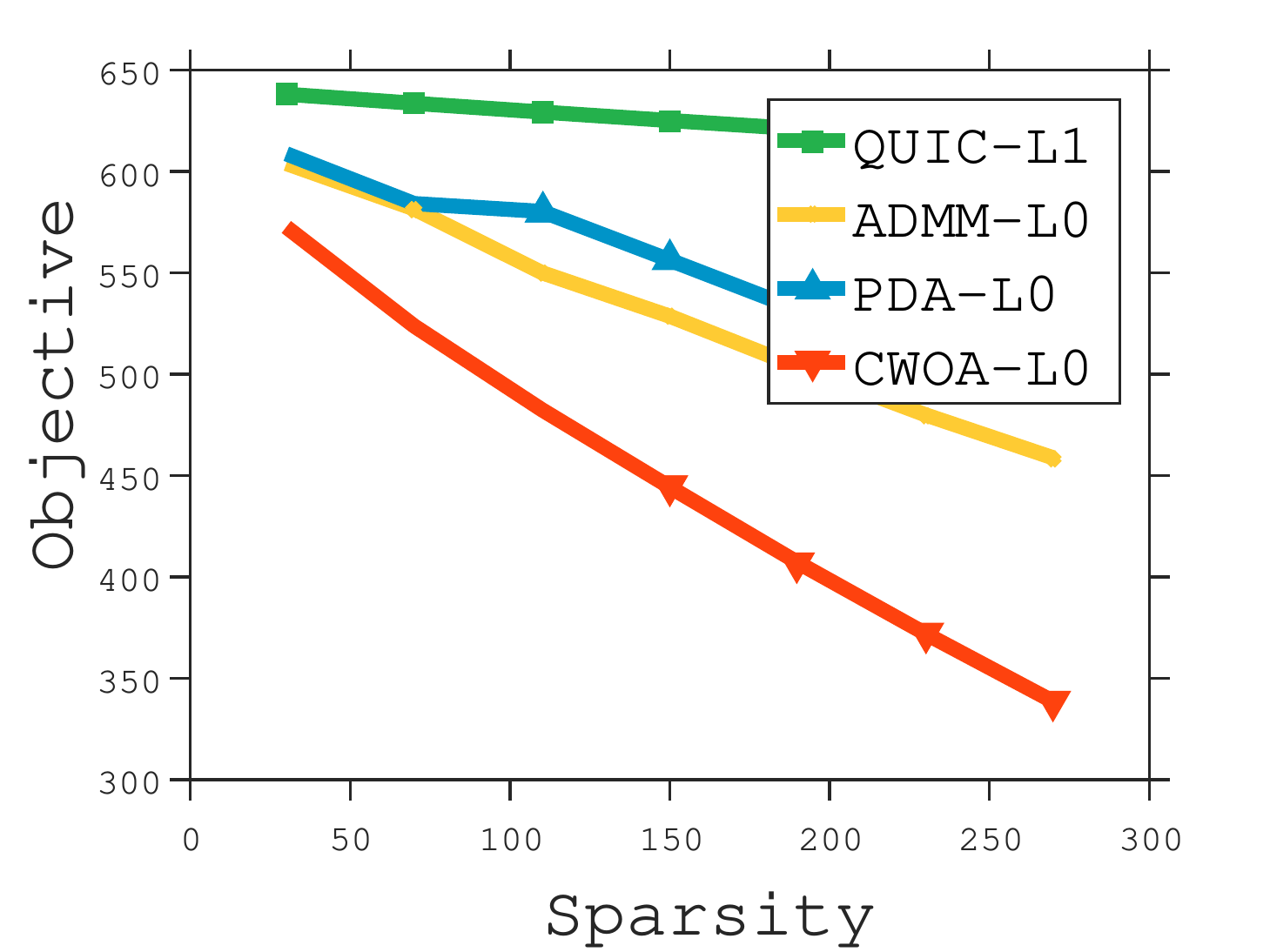}\vspace{-6pt}\caption{\scriptsize Real-World-isolet}
      \end{subfigure}
      \begin{subfigure}{0.24\textwidth}\includegraphics[width=1\textwidth,height=\hone]{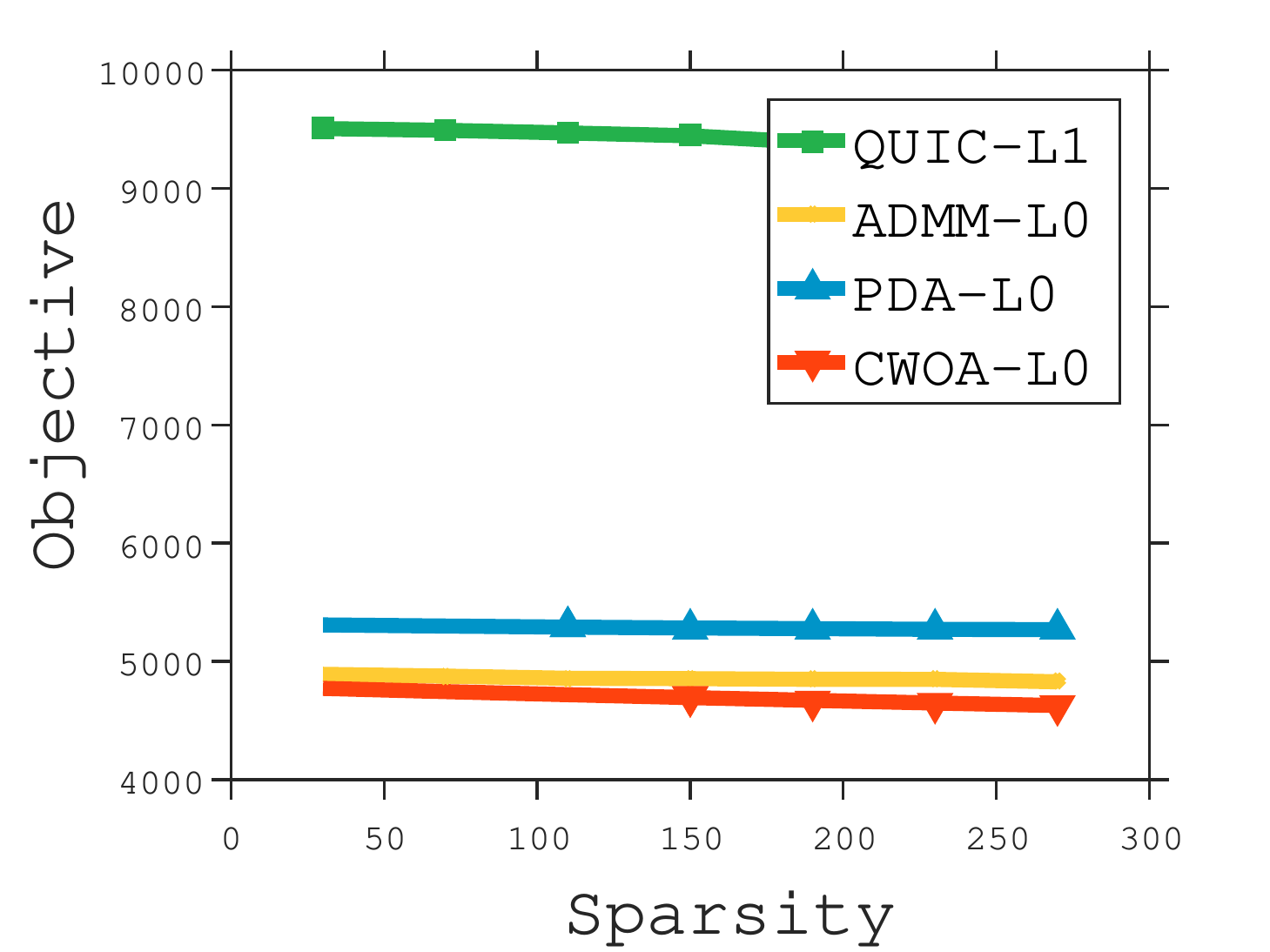}\vspace{-6pt}\caption{\scriptsize Real-World-mnist}
      \end{subfigure}
      \begin{subfigure}{0.24\textwidth}\includegraphics[width=1\textwidth,height=\hone]{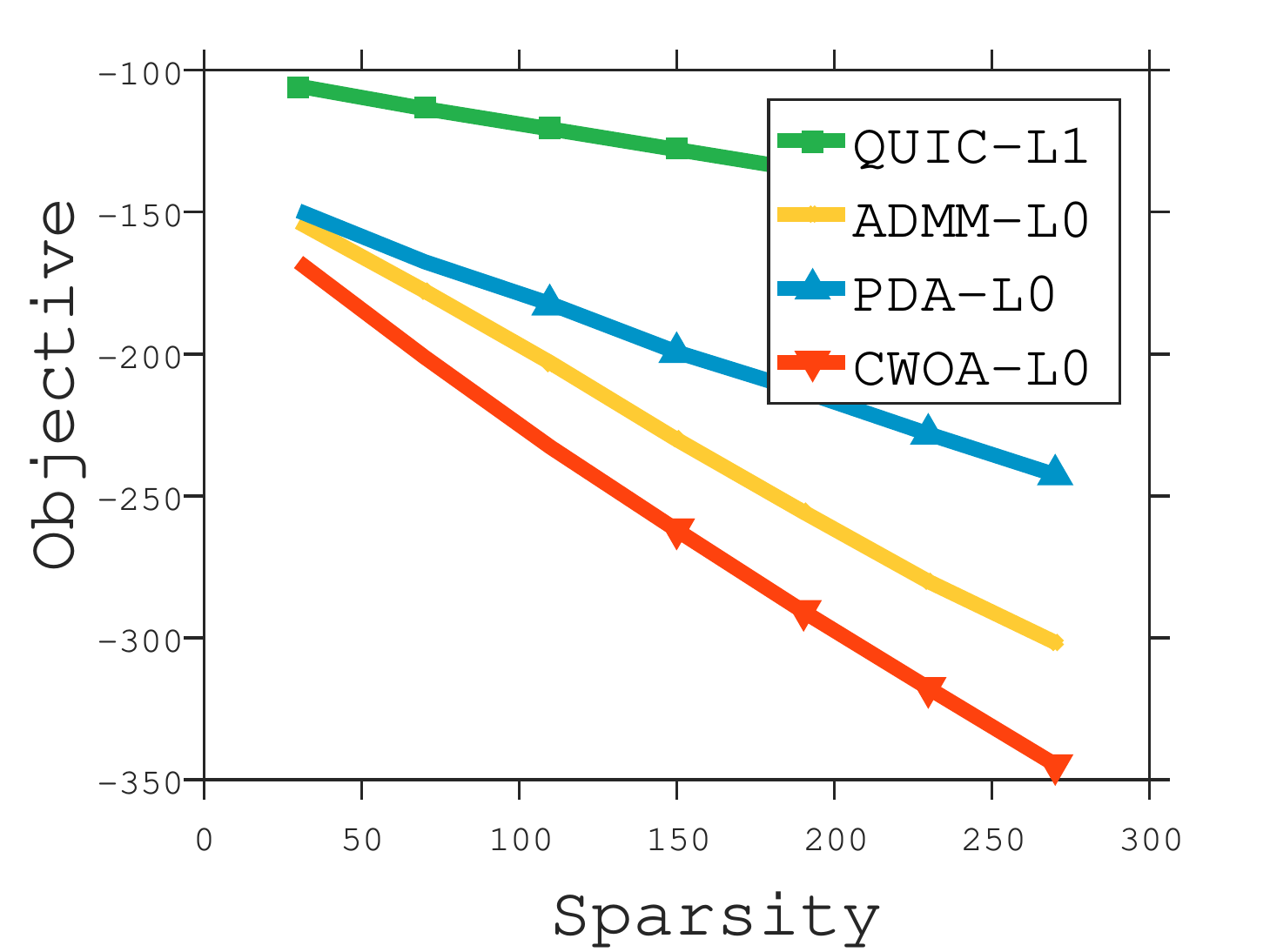}\vspace{-6pt}\caption{\scriptsize Real-World-usps}
      \end{subfigure}
      \begin{subfigure}{0.24\textwidth}\includegraphics[width=1\textwidth,height=\hone]{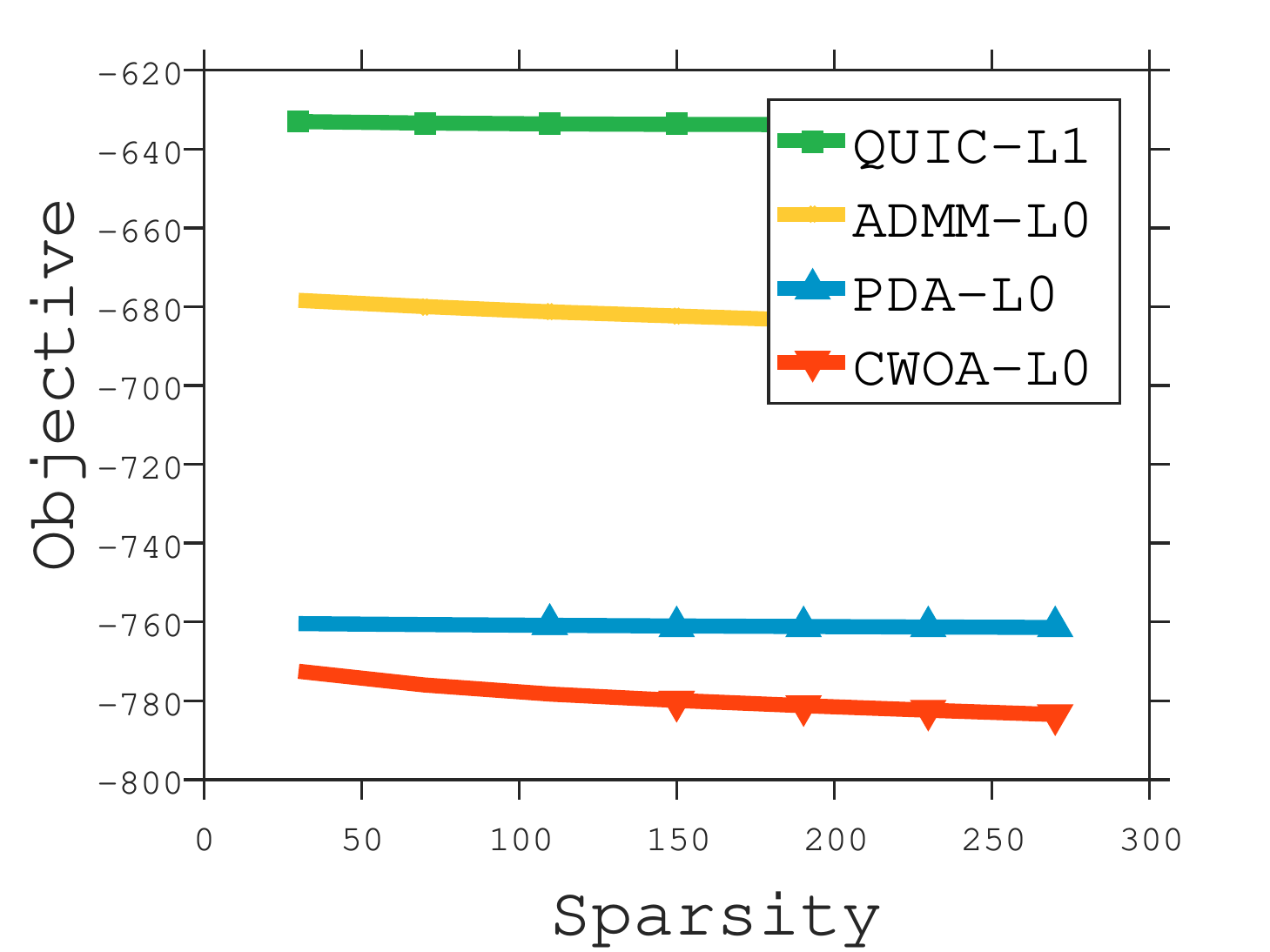}\vspace{-6pt}\caption{\scriptsize Real-World-w1a}
      \end{subfigure}
\vspace{-6pt}
\caption{Comparison of objective values for different methods.}
\label{fig:exp:fobj}

%\captionsetup{singlelinecheck = true, justification=justified}
\captionsetup{singlelinecheck = on, format= hang, justification=justified, font=footnotesize, labelsep=space}
\centering
      \begin{subfigure}{0.24\textwidth}\includegraphics[width=1\textwidth,height=\htwo]{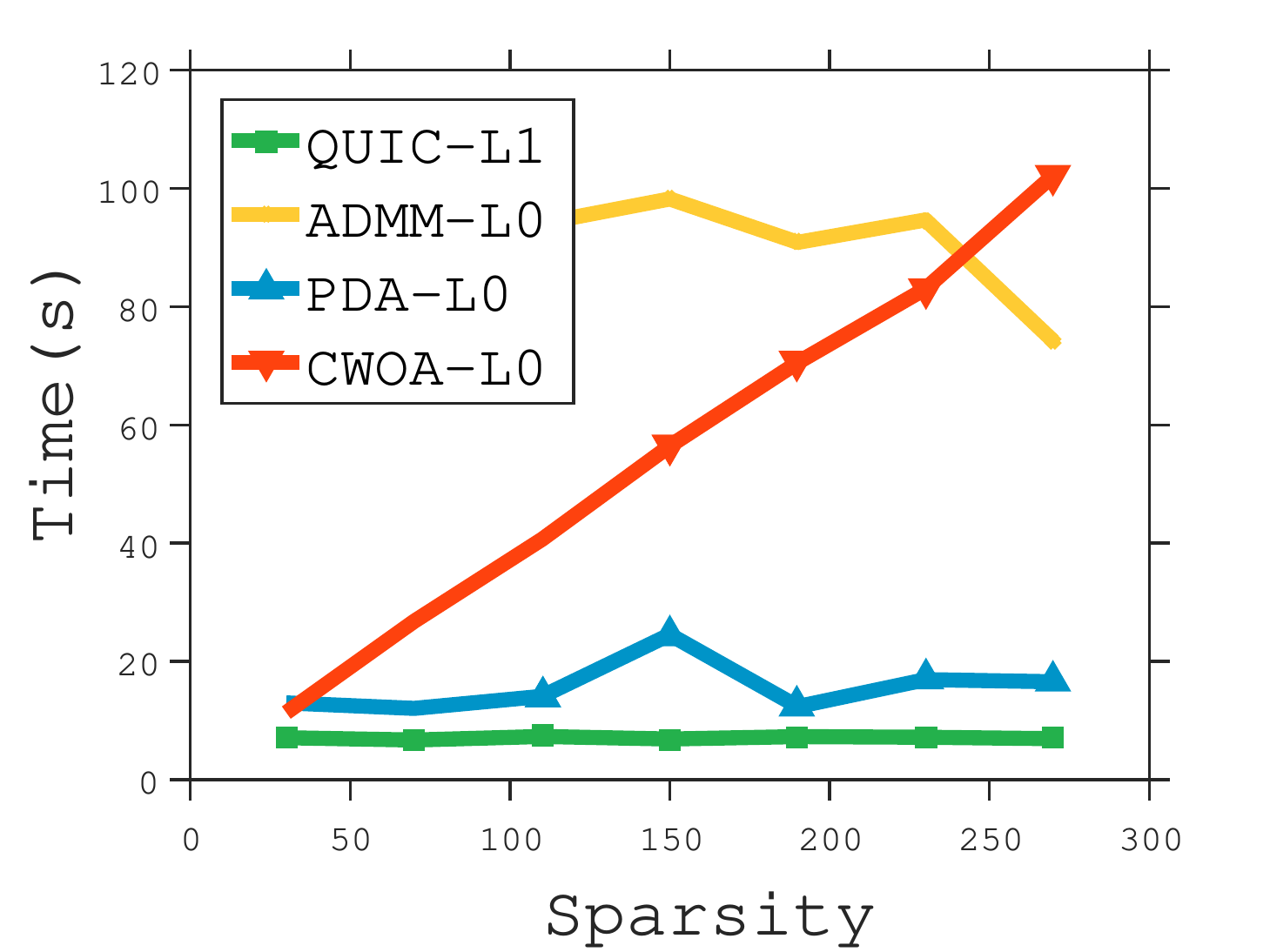}\vspace{-6pt}\caption{\scriptsize Gaussian-Random-500}
      \end{subfigure}
      \begin{subfigure}{0.24\textwidth}\includegraphics[width=1\textwidth,height=\htwo]{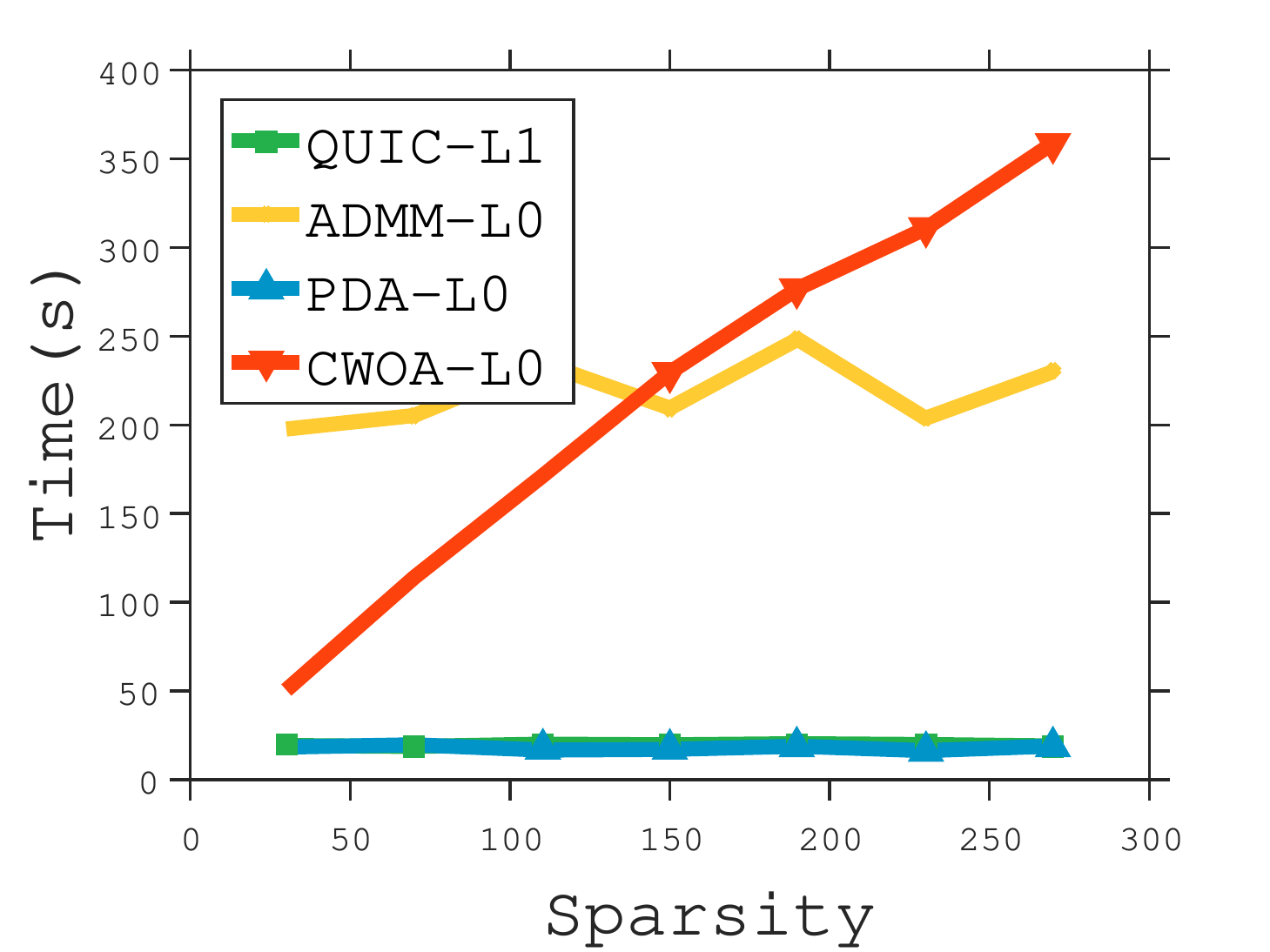}\vspace{-6pt}\caption{\scriptsize Gaussian-Random-1000}
      \end{subfigure}
      \begin{subfigure}{0.24\textwidth}\includegraphics[width=1\textwidth,height=\htwo]{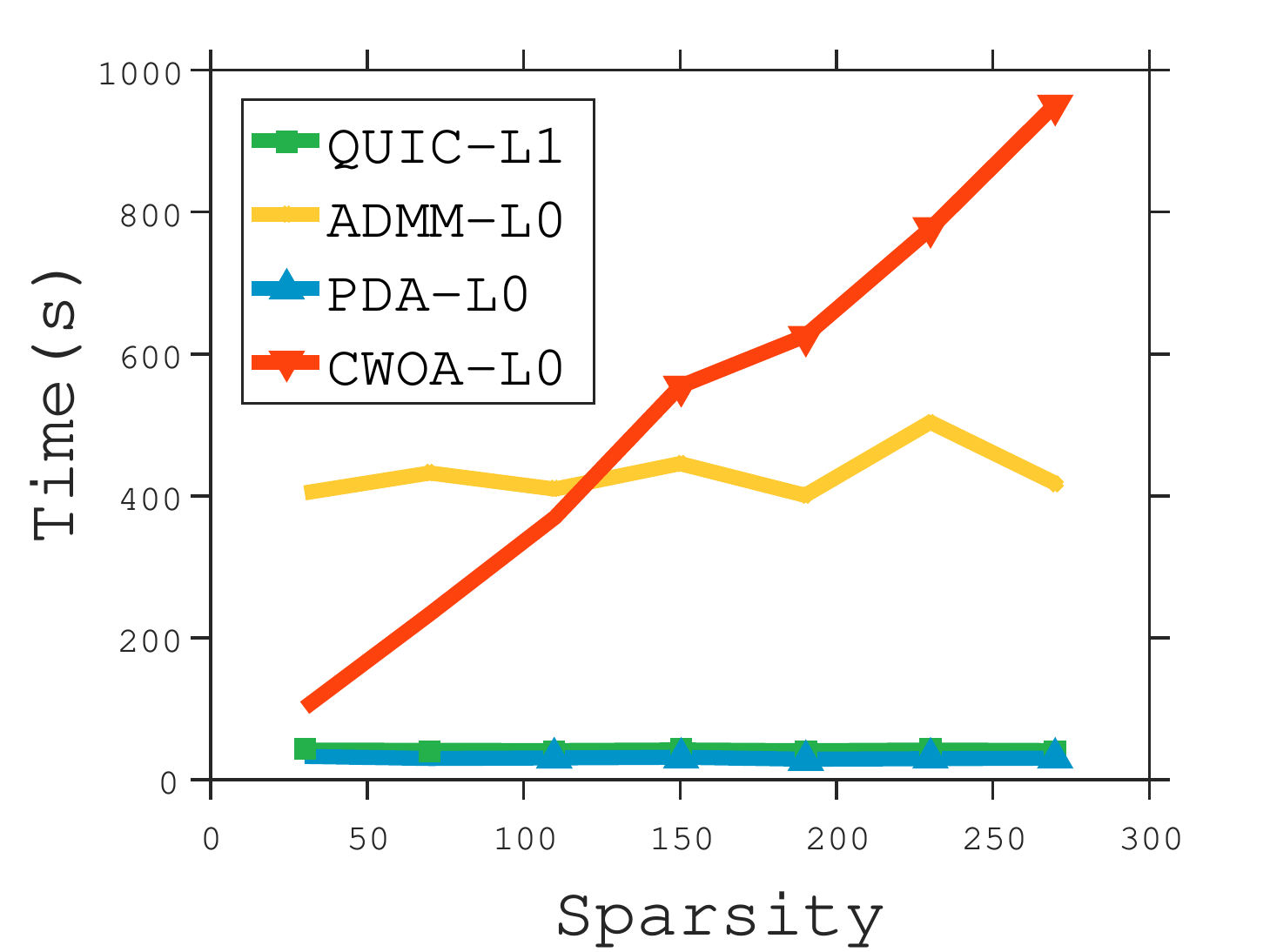}\vspace{-6pt}\caption{\scriptsize Gaussian-Random-1500}
      \end{subfigure}
      \begin{subfigure}{0.24\textwidth}\includegraphics[width=1\textwidth,height=\htwo]{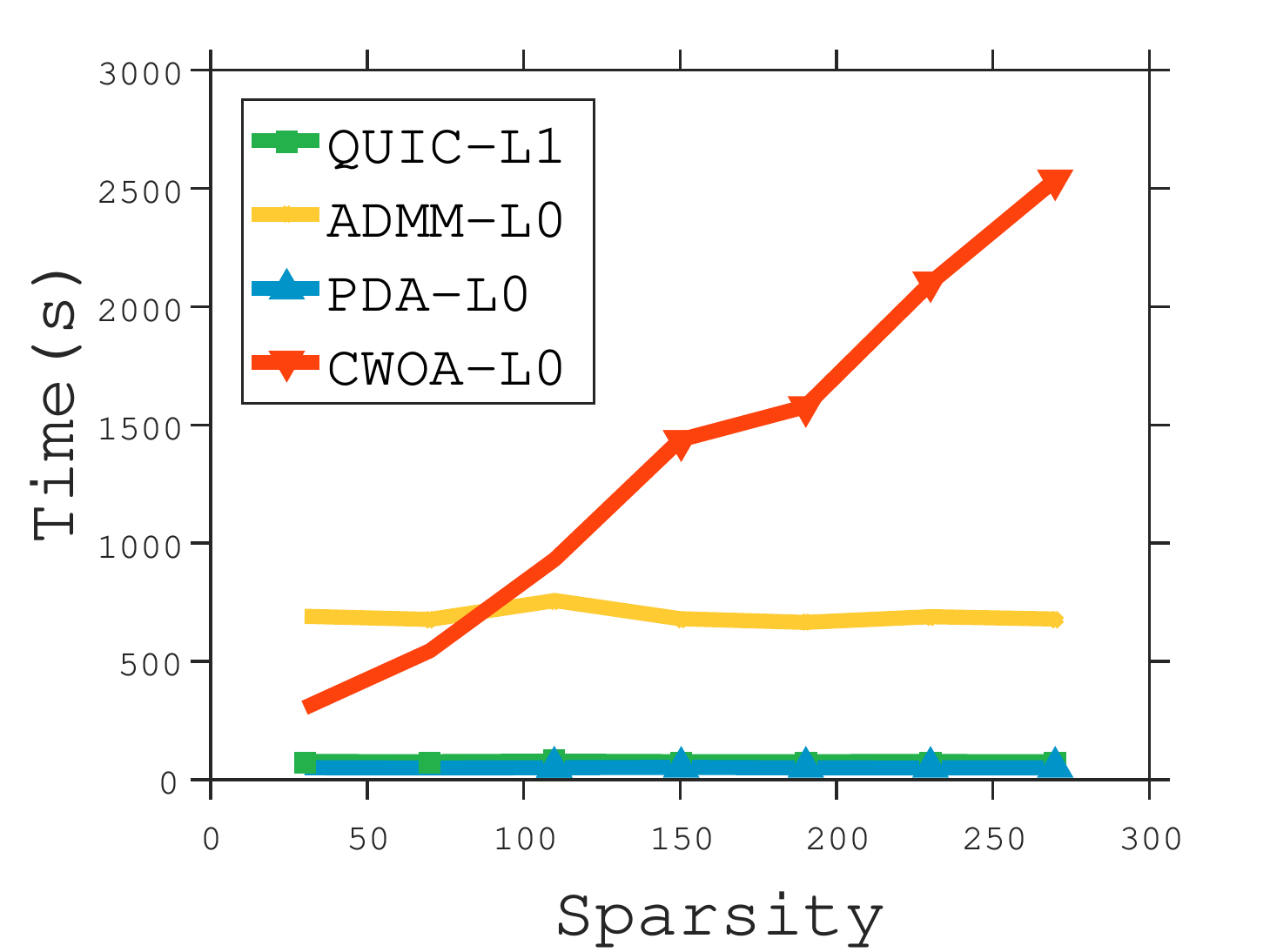}\vspace{-6pt}\caption{\scriptsize Gaussian-Random-2000}
      \end{subfigure}

    %  \vspace{-6pt}

      \begin{subfigure}{0.24\textwidth}\includegraphics[width=1\textwidth,height=\htwo]{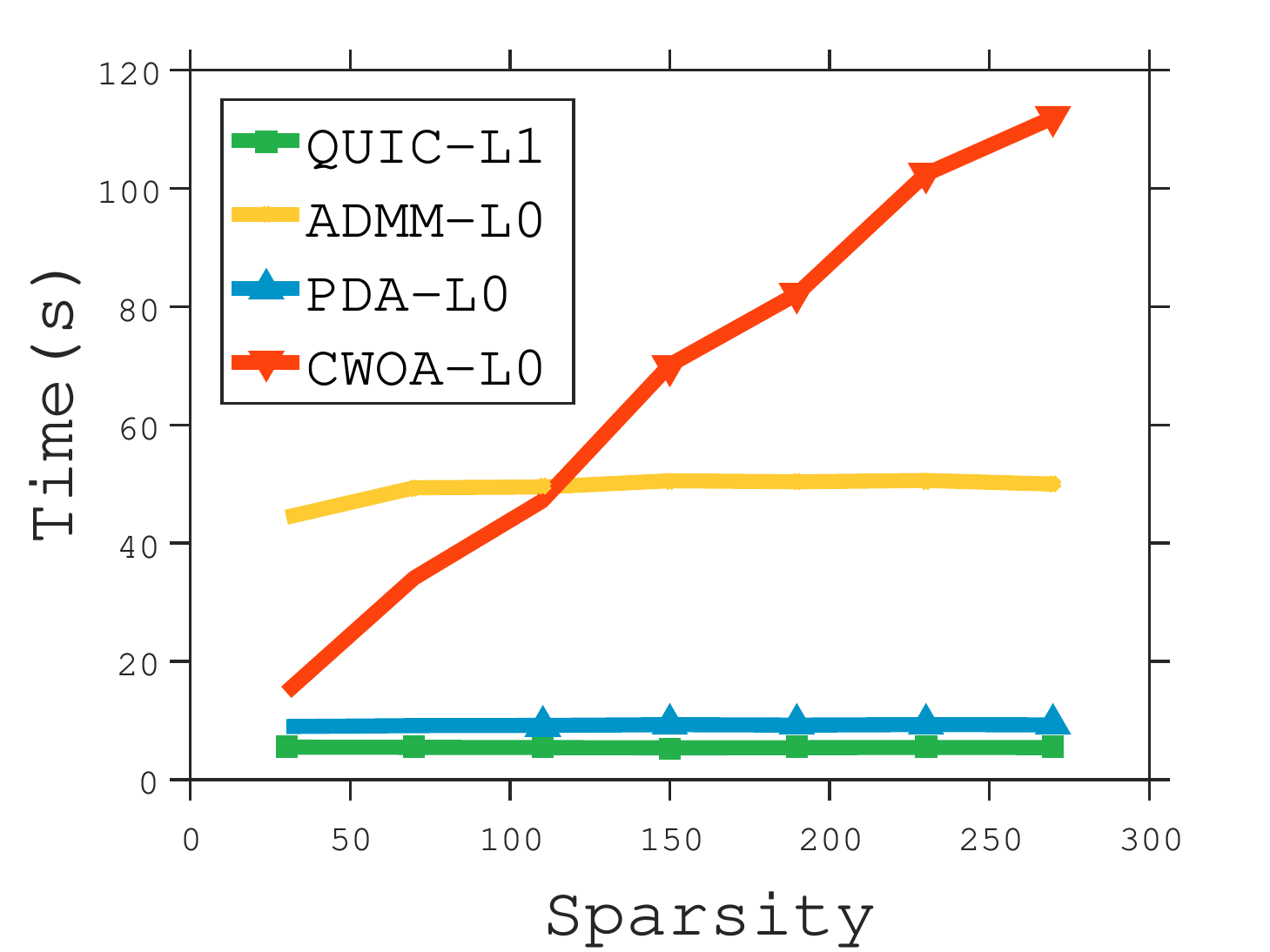}\vspace{-6pt}\caption{\scriptsize Sparse-Structured-500}
      \end{subfigure}
      \begin{subfigure}{0.24\textwidth}\includegraphics[width=1\textwidth,height=\htwo]{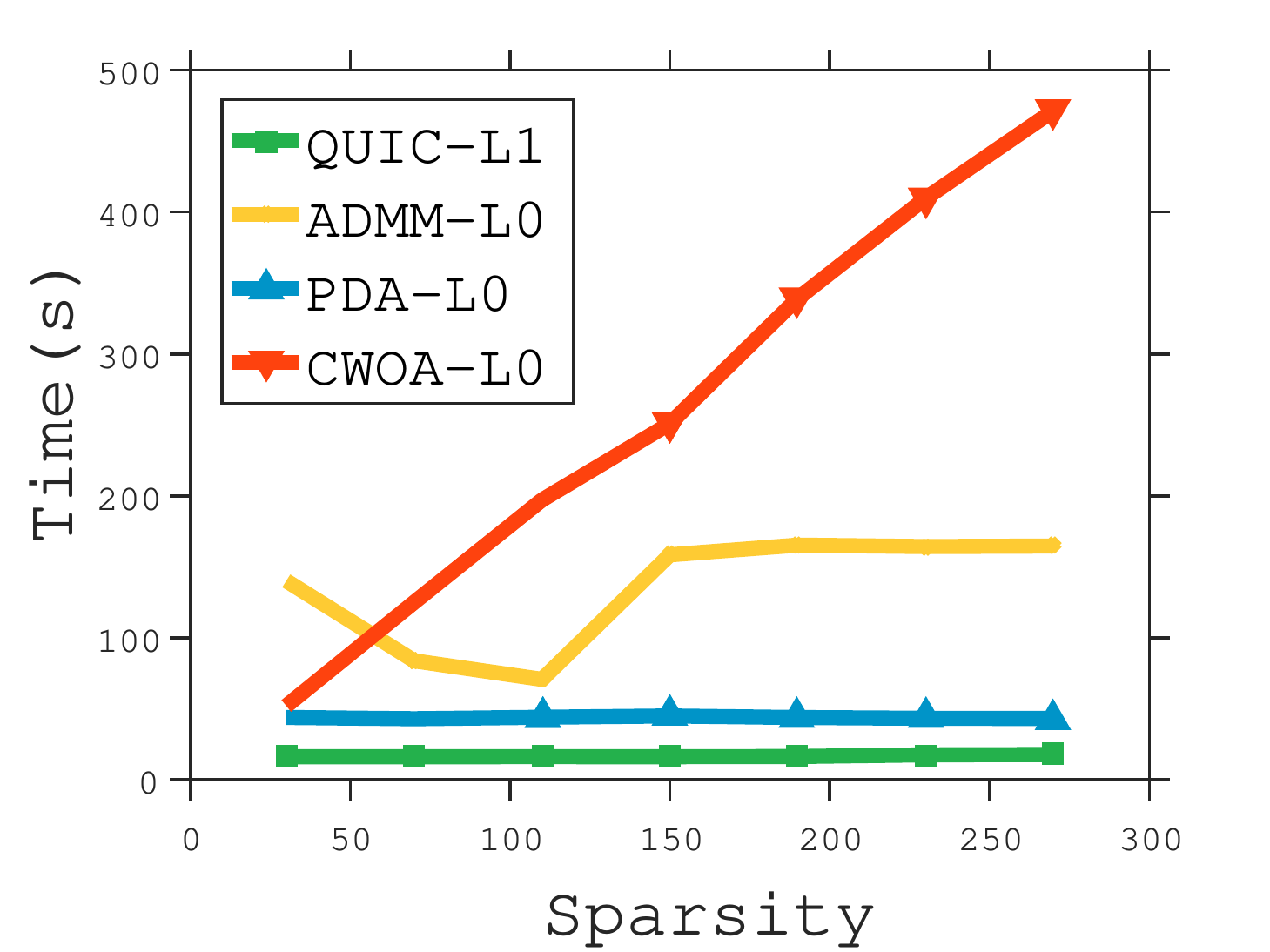}\vspace{-6pt}\caption{\scriptsize Sparse-Structured-1000}
      \end{subfigure}
      \begin{subfigure}{0.24\textwidth}\includegraphics[width=1\textwidth,height=\htwo]{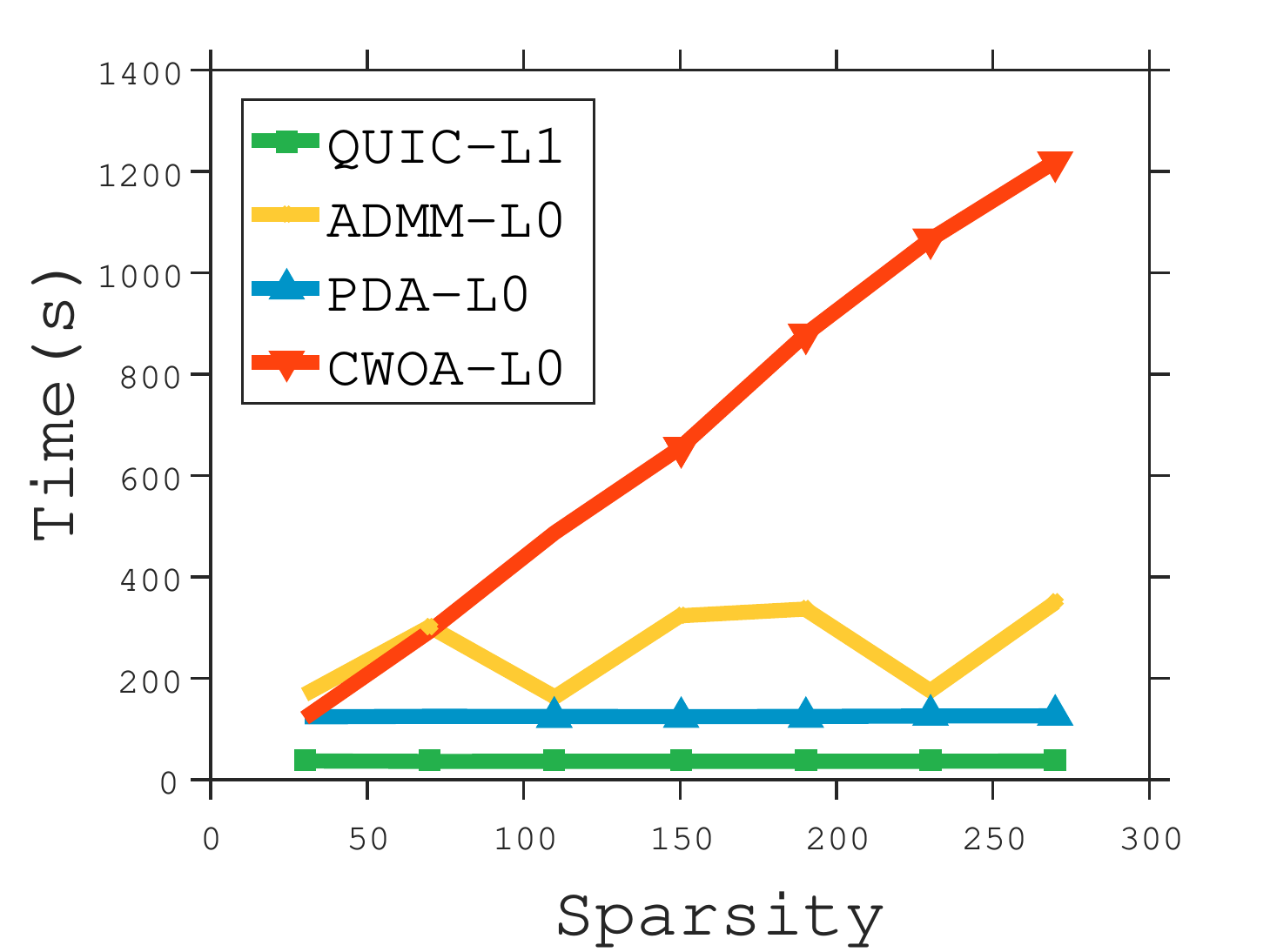}\vspace{-6pt}\caption{\scriptsize Sparse-Structured-1500}
      \end{subfigure}
     \begin{subfigure}{0.24\textwidth}\includegraphics[width=1\textwidth,height=\htwo]{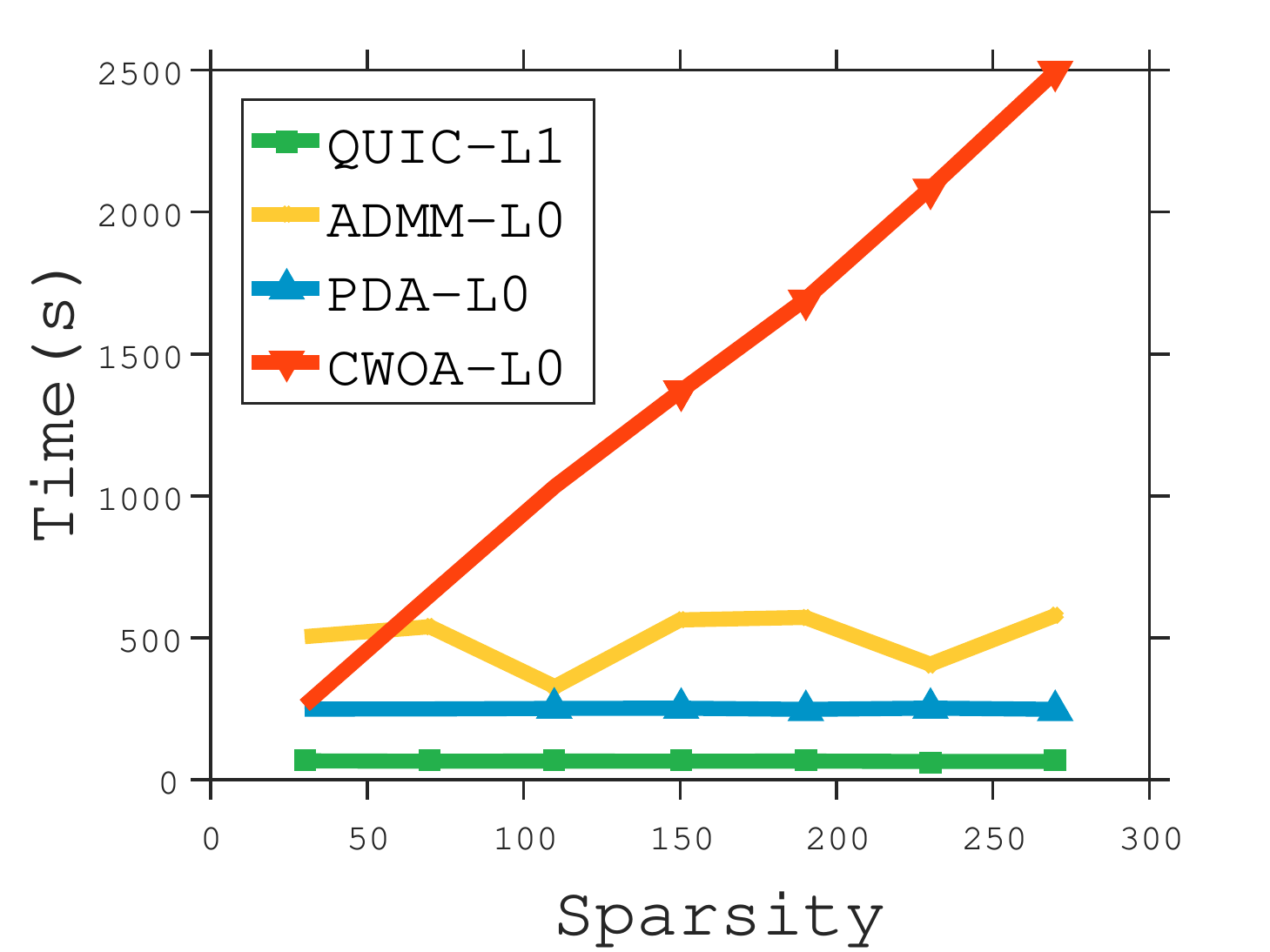}\vspace{-6pt}\caption{\scriptsize Sparse-Structured-2000}
      \end{subfigure}

     % \vspace{-6pt}

      \begin{subfigure}{0.24\textwidth}\includegraphics[width=1\textwidth,height=\htwo]{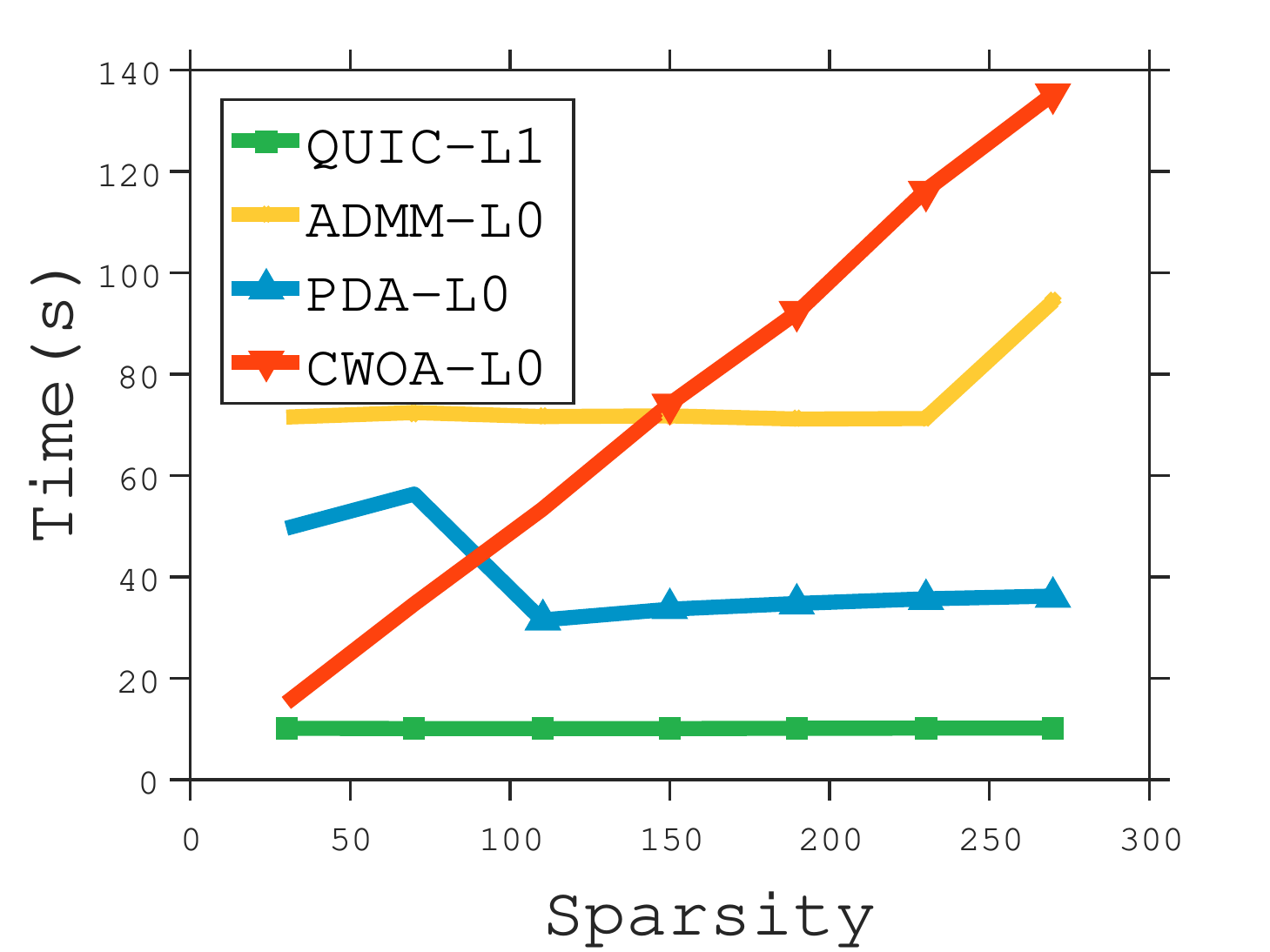}\vspace{-6pt}\caption{\scriptsize Real-World-isolet}
      \end{subfigure}
      \begin{subfigure}{0.24\textwidth}\includegraphics[width=1\textwidth,height=\htwo]{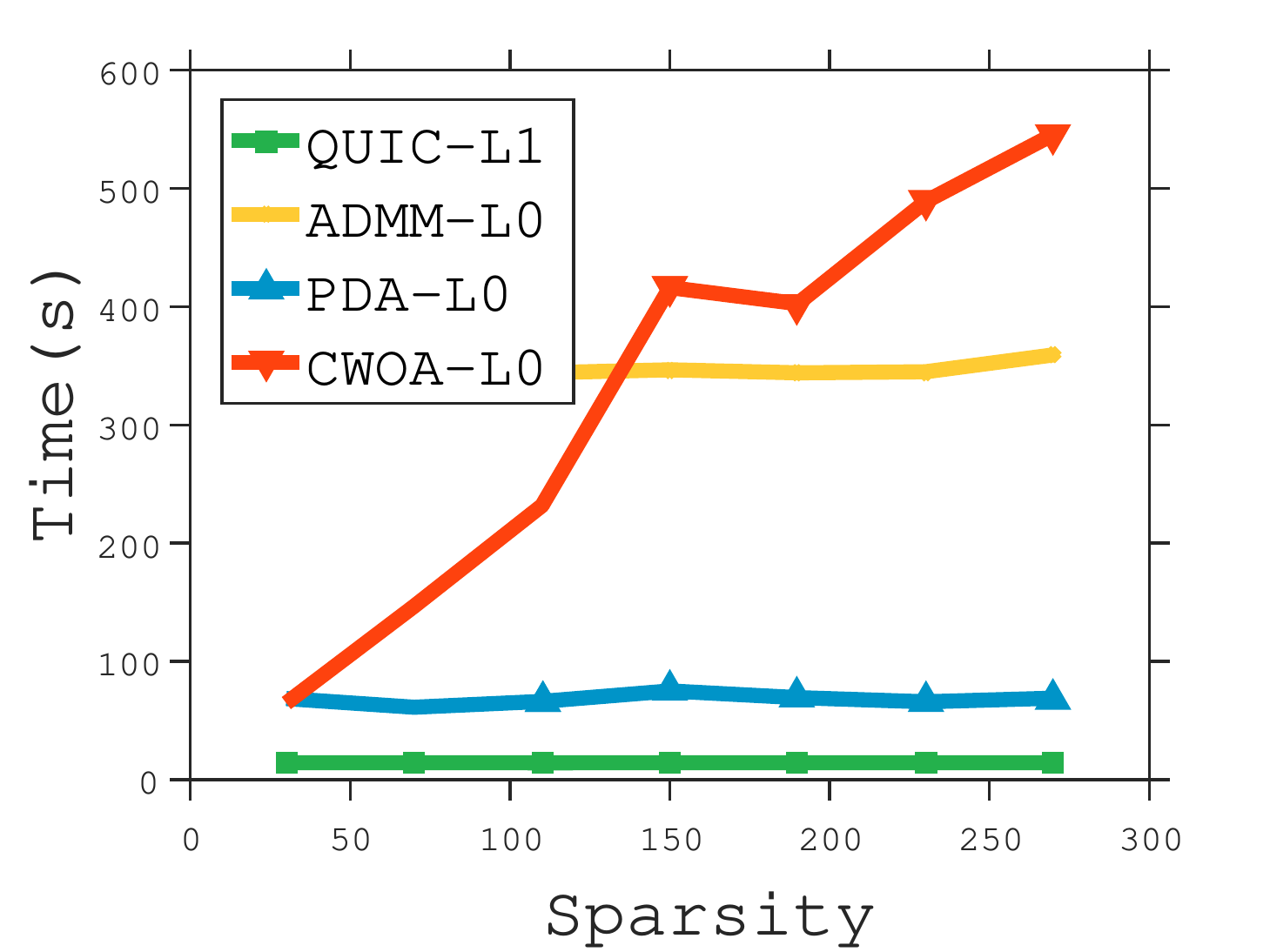}\vspace{-6pt}\caption{\scriptsize Real-World-mnist}
      \end{subfigure}
      \begin{subfigure}{0.24\textwidth}\includegraphics[width=1\textwidth,height=\htwo]{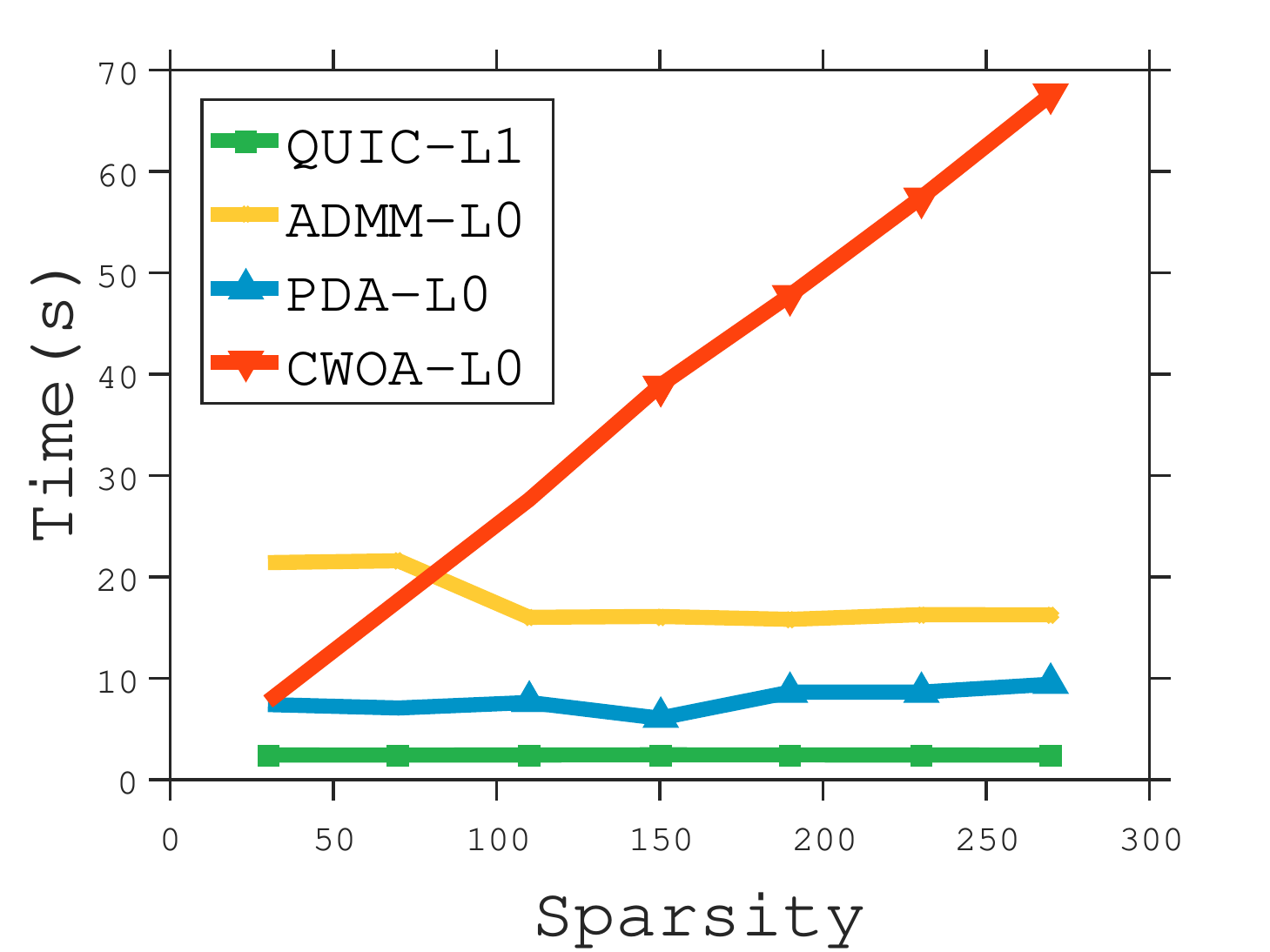}\vspace{-6pt}\caption{\scriptsize Real-World-usps}
      \end{subfigure}
      \begin{subfigure}{0.24\textwidth}\includegraphics[width=1\textwidth,height=\htwo]{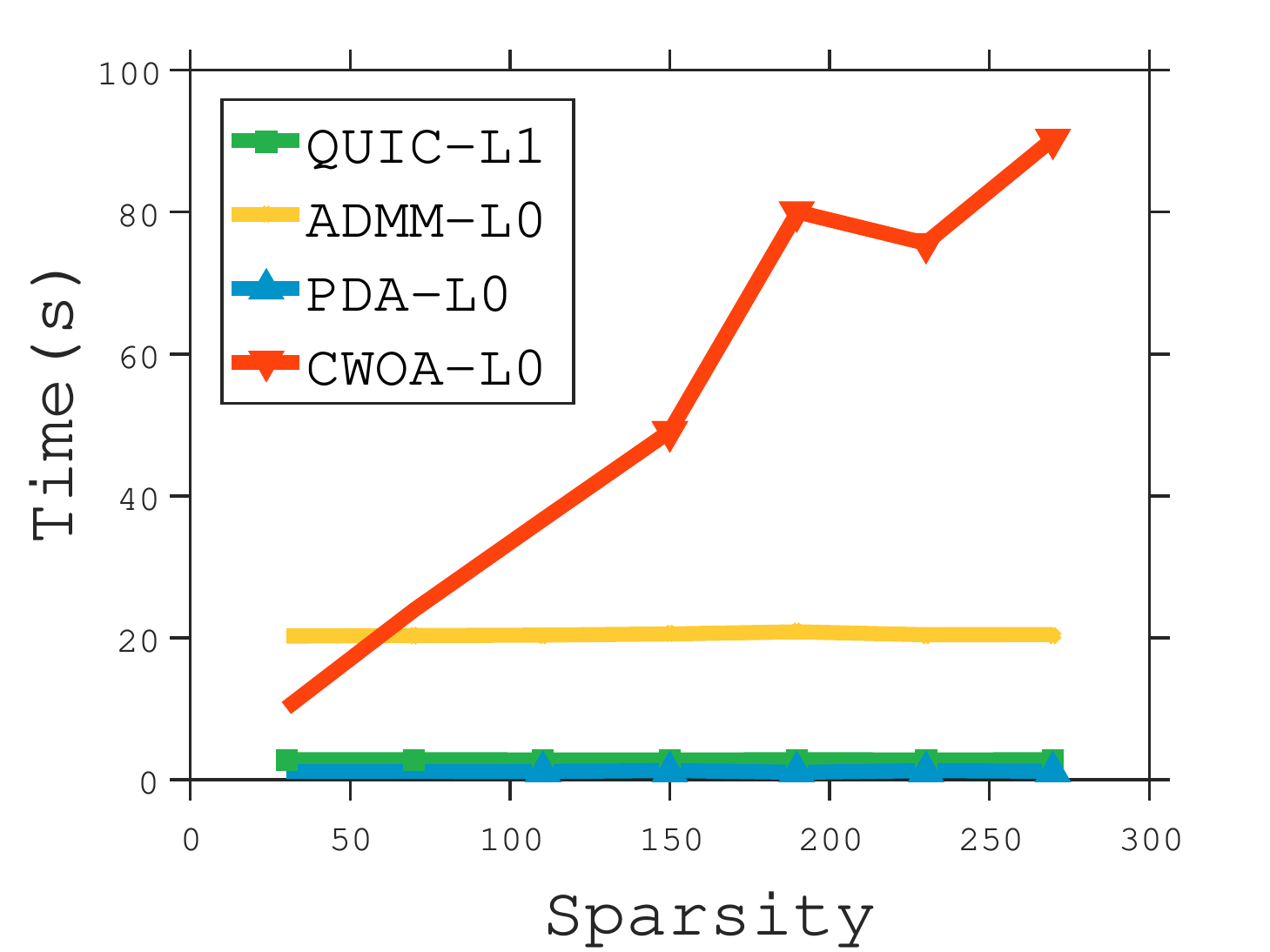}\vspace{-6pt}\caption{\scriptsize Real-World-w1a}
      \end{subfigure}
\vspace{-6pt}
\caption{Comparison of computational time for different methods.}
\label{fig:exp:time}
\end{figure*}

\begin{figure*} [!t]
\captionsetup{singlelinecheck = on, format= hang, justification=justified, font=footnotesize, labelsep=space}
\centering
      \begin{subfigure}{0.24\textwidth}\includegraphics[width=1\textwidth,height=\htwo]{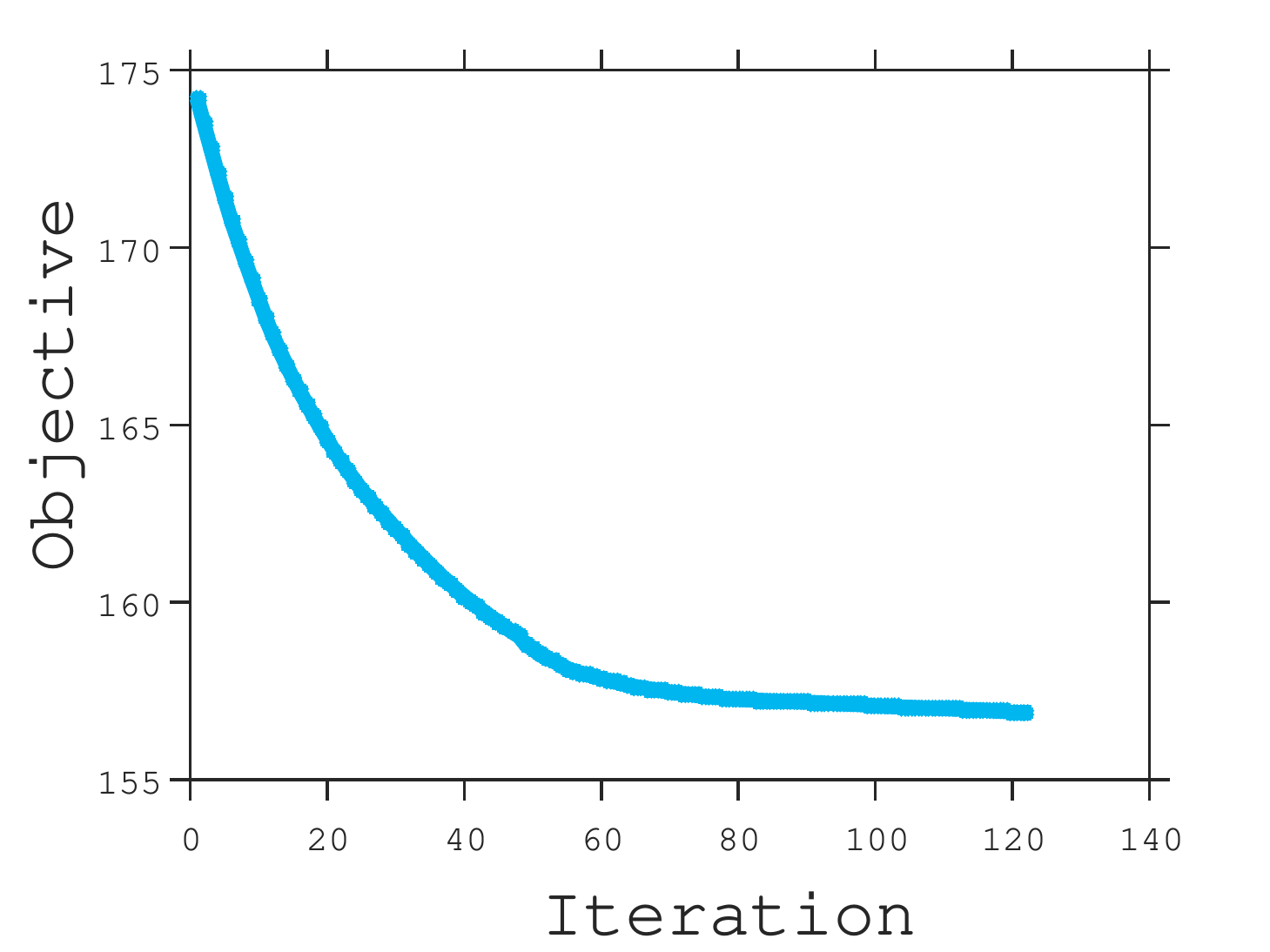}\vspace{-6pt}\caption{\scriptsize Gaussian-Random-100}
      \label{unknown-500}
      \end{subfigure}5
      \begin{subfigure}{0.24\textwidth}\includegraphics[width=1\textwidth,height=\htwo]{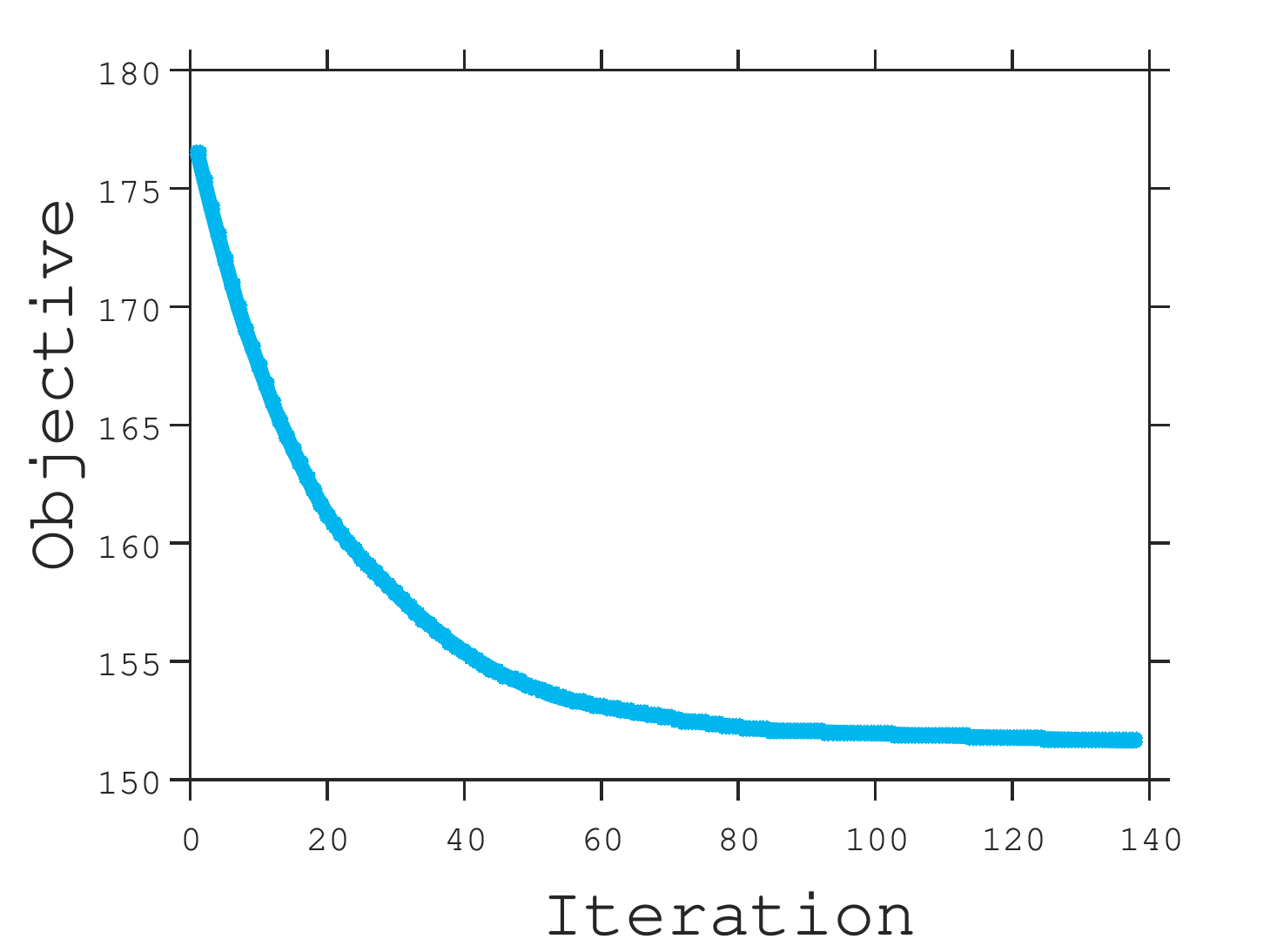}\vspace{-6pt}\caption{\scriptsize Gaussian-Random-100}
      \end{subfigure}
      \begin{subfigure}{0.24\textwidth}\includegraphics[width=1\textwidth,height=\htwo]{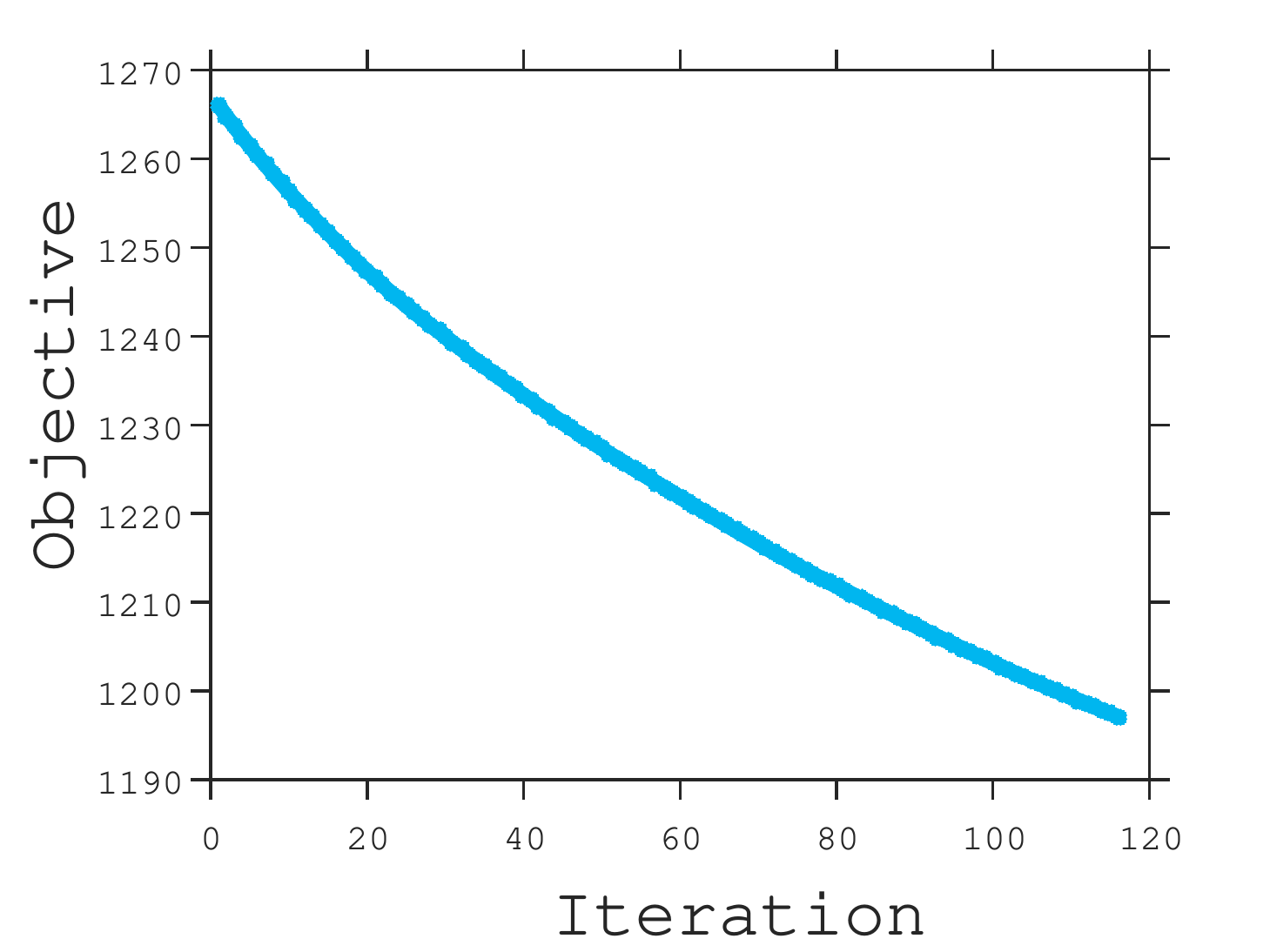}\vspace{-6pt}\caption{\scriptsize Gaussian-Random-500}
      \end{subfigure}
      \begin{subfigure}{0.24\textwidth}\includegraphics[width=1\textwidth,height=\htwo]{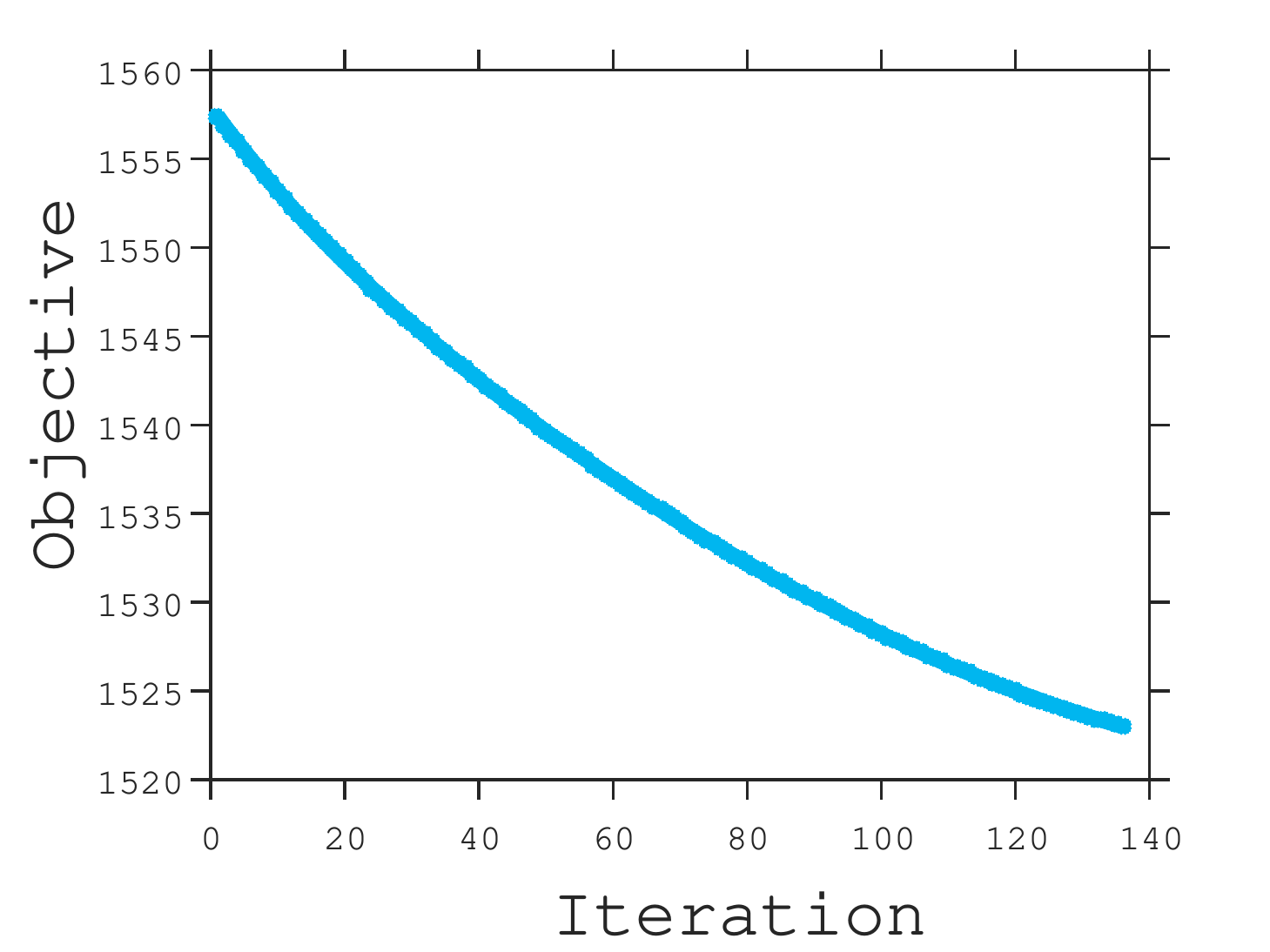}\vspace{-6pt}\caption{\scriptsize Gaussian-Random-500}
      \end{subfigure}
\caption{Convergence curve for Algorithm \ref{alg:main} on different data sets. $s$ is set to 230 for (a) and (c), and $s$ is set to 270 for (b) and (d).}
\label{fig:exp:converge:1}

\captionsetup{singlelinecheck = on, format= hang, justification=justified, font=footnotesize, labelsep=space}
\centering
      \begin{subfigure}{0.24\textwidth}\includegraphics[width=1\textwidth,height=\htwo]{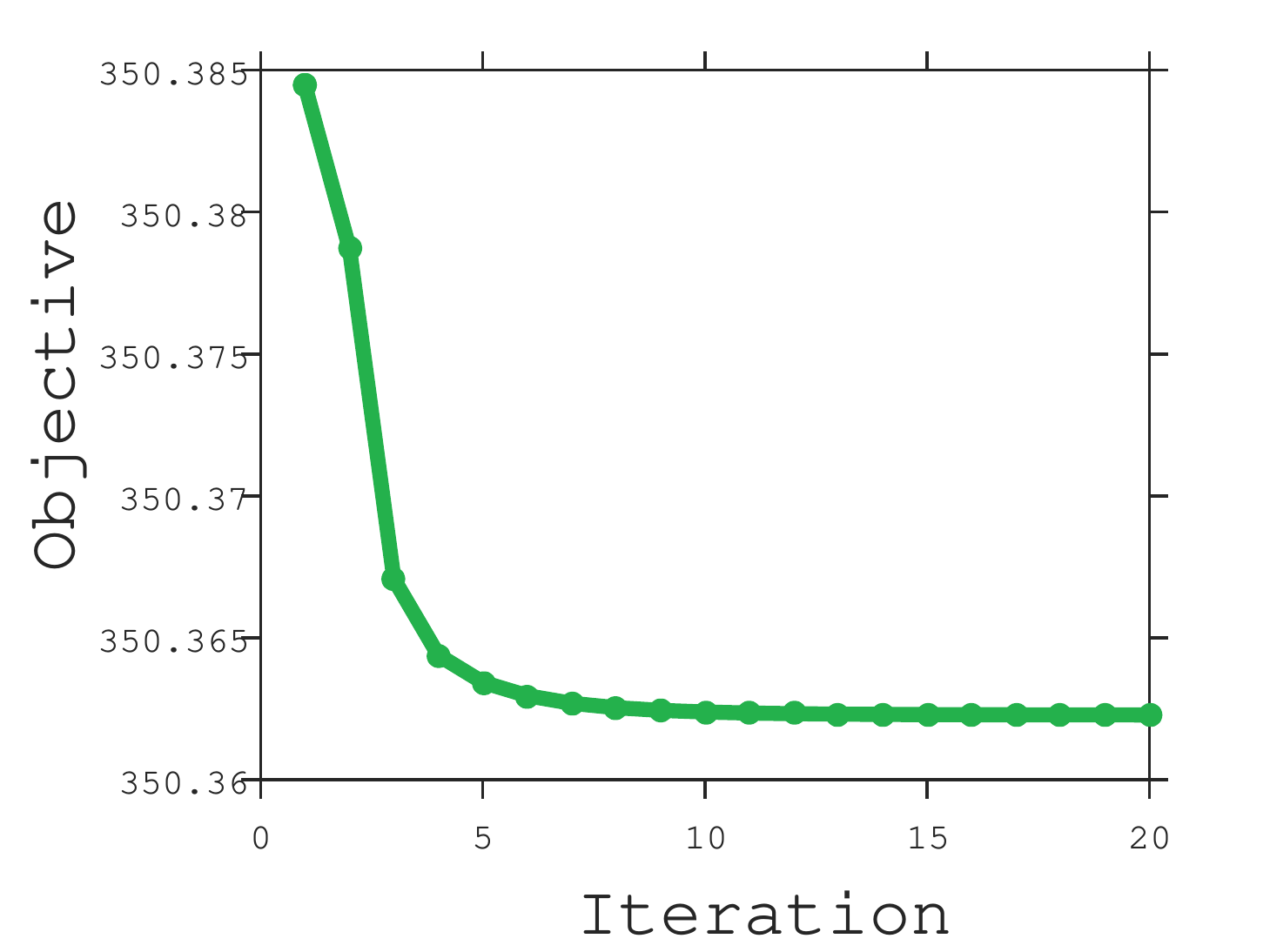}\vspace{-6pt}\caption{\scriptsize $k=20$}
      \label{unknown-500}
      \end{subfigure}
      \begin{subfigure}{0.24\textwidth}\includegraphics[width=1\textwidth,height=\htwo]{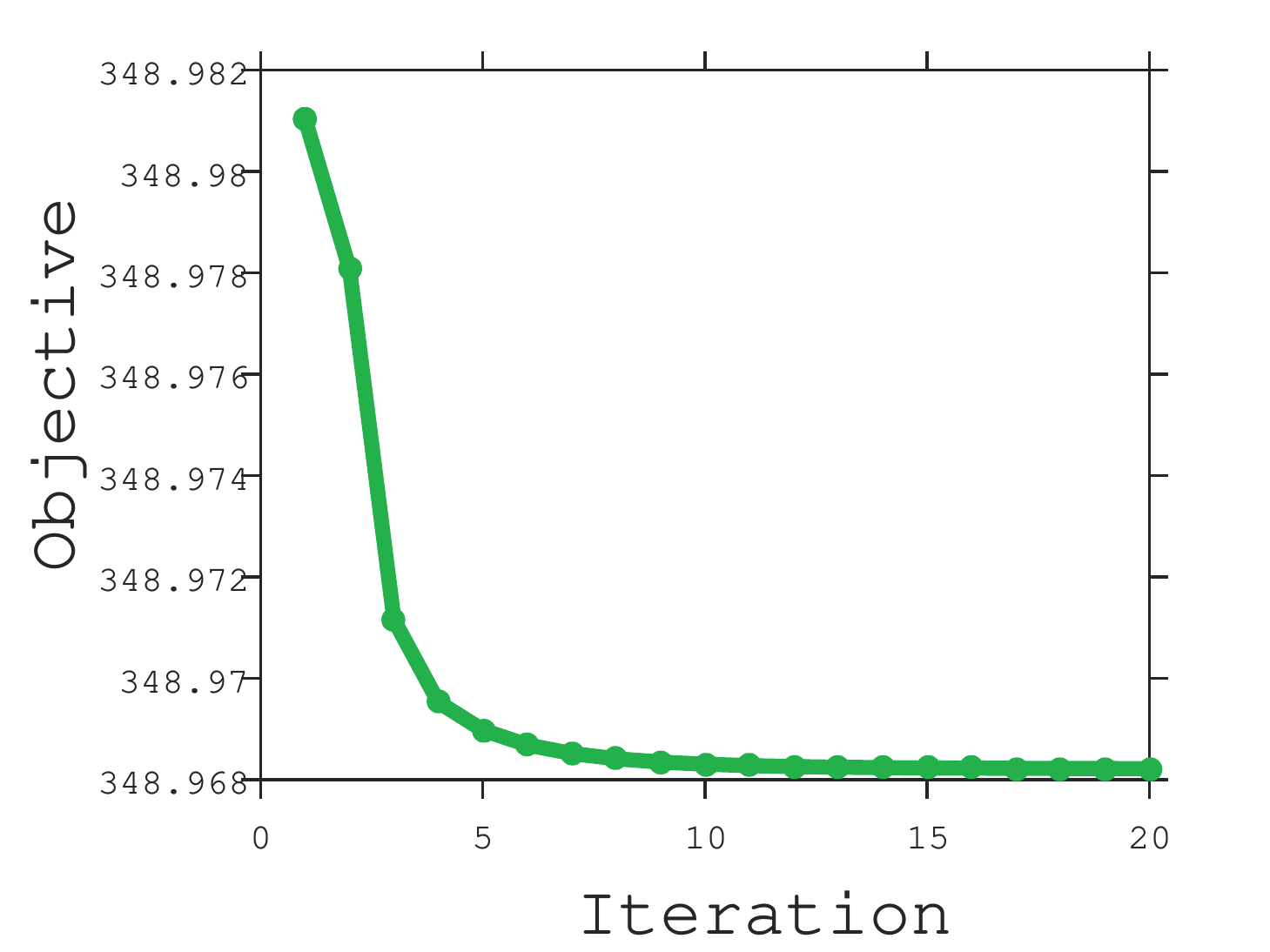}\vspace{-6pt}\caption{\scriptsize $k=30$}
      \end{subfigure}
      \begin{subfigure}{0.24\textwidth}\includegraphics[width=1\textwidth,height=\htwo]{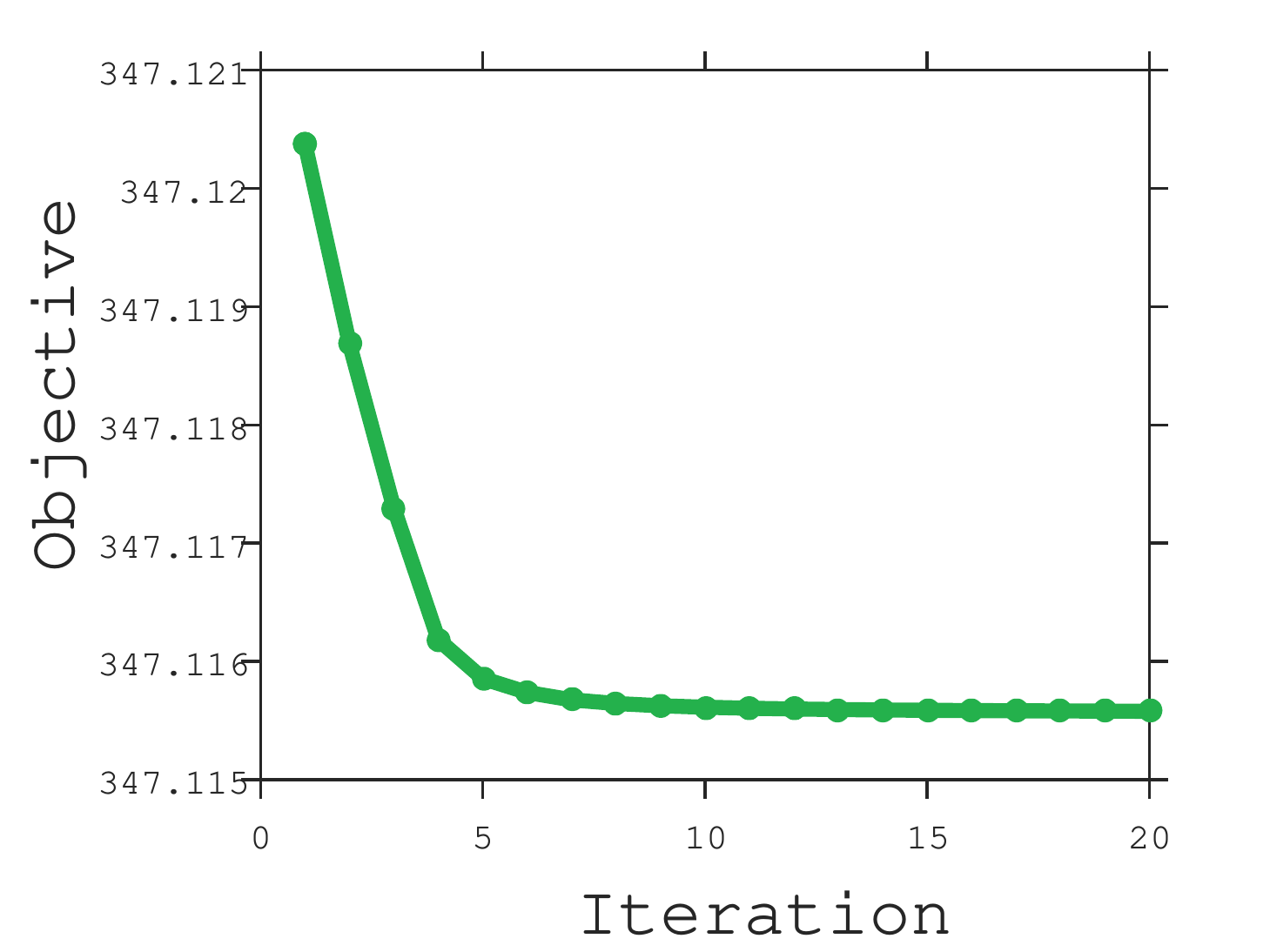}\vspace{-6pt}\caption{\scriptsize $k=50$}
      \end{subfigure}
      \begin{subfigure}{0.24\textwidth}\includegraphics[width=1\textwidth,height=\htwo]{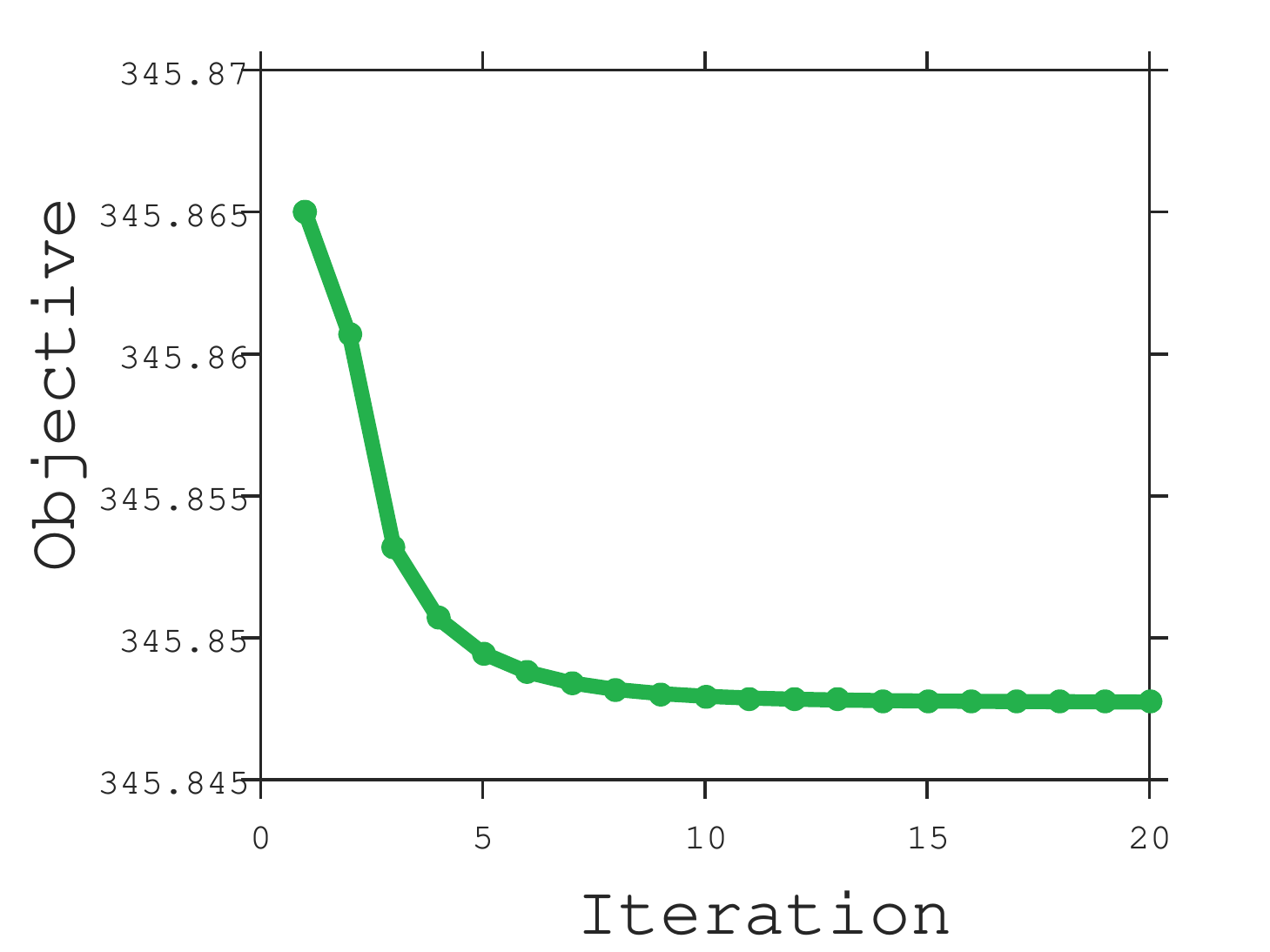}\vspace{-6pt}\caption{\scriptsize $k=70$}
      \end{subfigure}
\caption{Convergence curve for Algorithm \ref{alg:newtoncg} on Real-World-mnist data set with $s=150$ for different iterations $t=\{20,30,50,70\}$ of Algorithm \ref{alg:main}.}
\label{fig:exp:converge:2}
\end{figure*}

\begin{figure*} [!th]
%\captionsetup{singlelinecheck = true, justification=justified}
\captionsetup{singlelinecheck = on, format= hang, justification=justified, font=footnotesize, labelsep=space}
\centering
      \begin{subfigure}{0.24\textwidth}\includegraphics[width=1\textwidth]{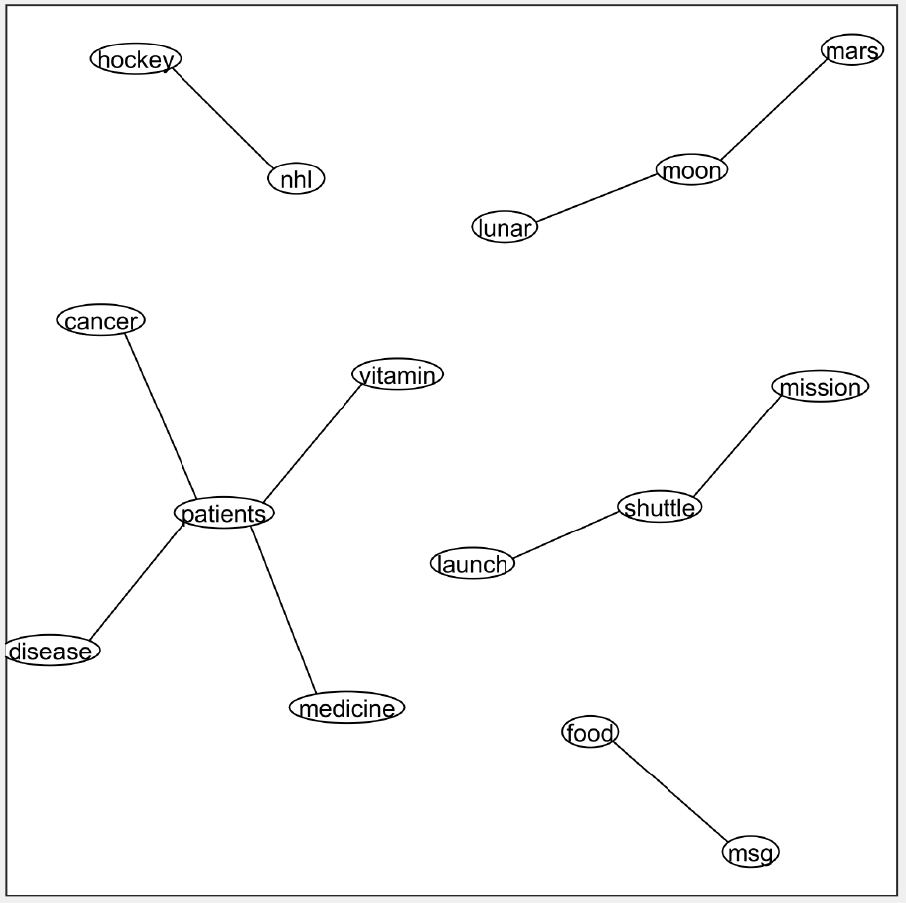}\vspace{-6pt}\caption{\scriptsize QUIC-L1}
      \label{unknown-500}
      \end{subfigure}
      \begin{subfigure}{0.24\textwidth}\includegraphics[width=1\textwidth]{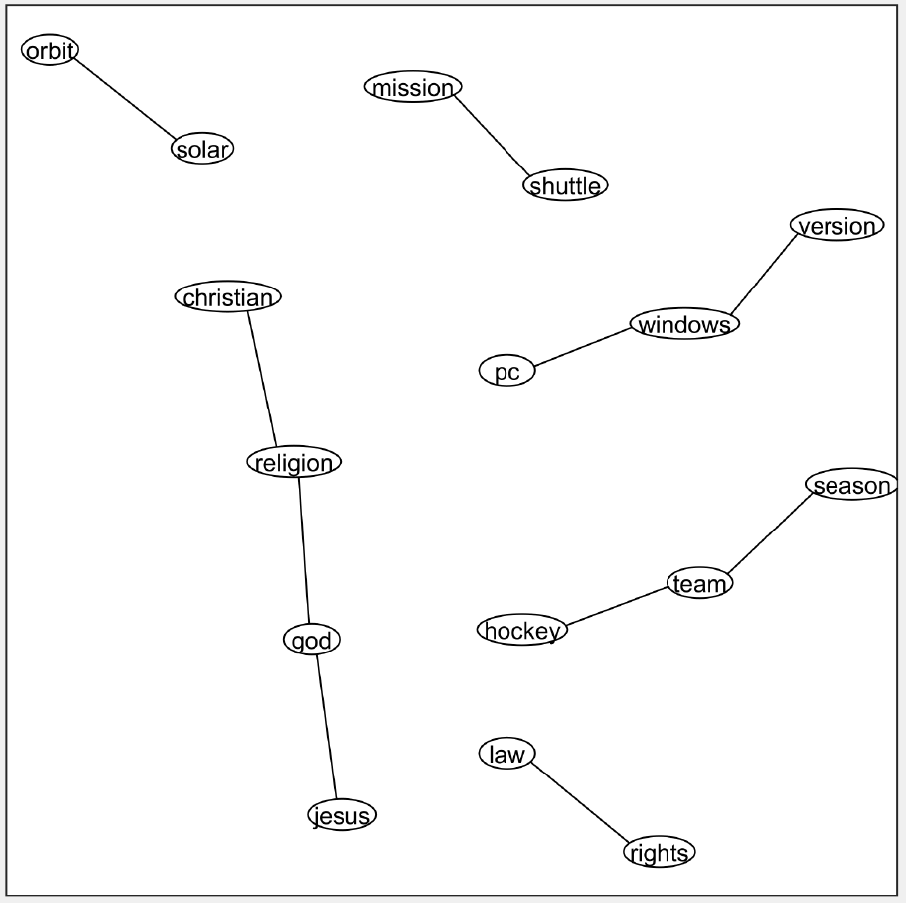}\vspace{-6pt}\caption{\scriptsize ADMM-L0}
      \end{subfigure}
      \begin{subfigure}{0.24\textwidth}\includegraphics[width=1\textwidth]{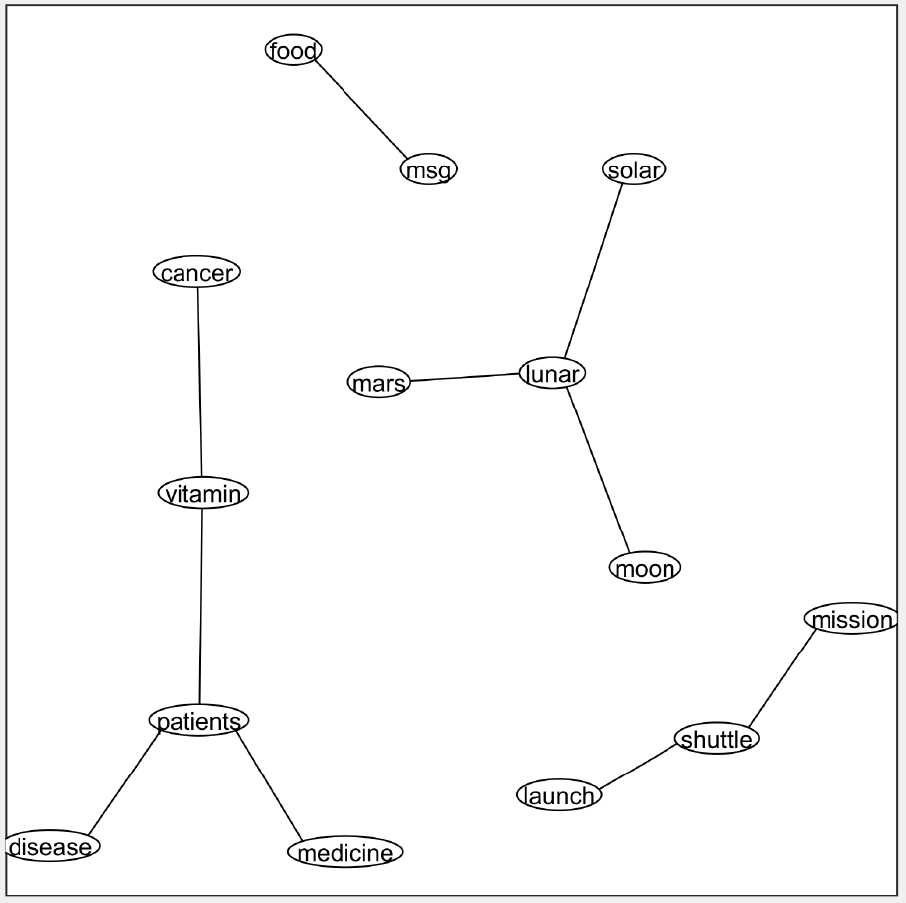}\vspace{-6pt}\caption{\scriptsize PDA-L0}
      \end{subfigure}
      \begin{subfigure}{0.24\textwidth}\includegraphics[width=1\textwidth]{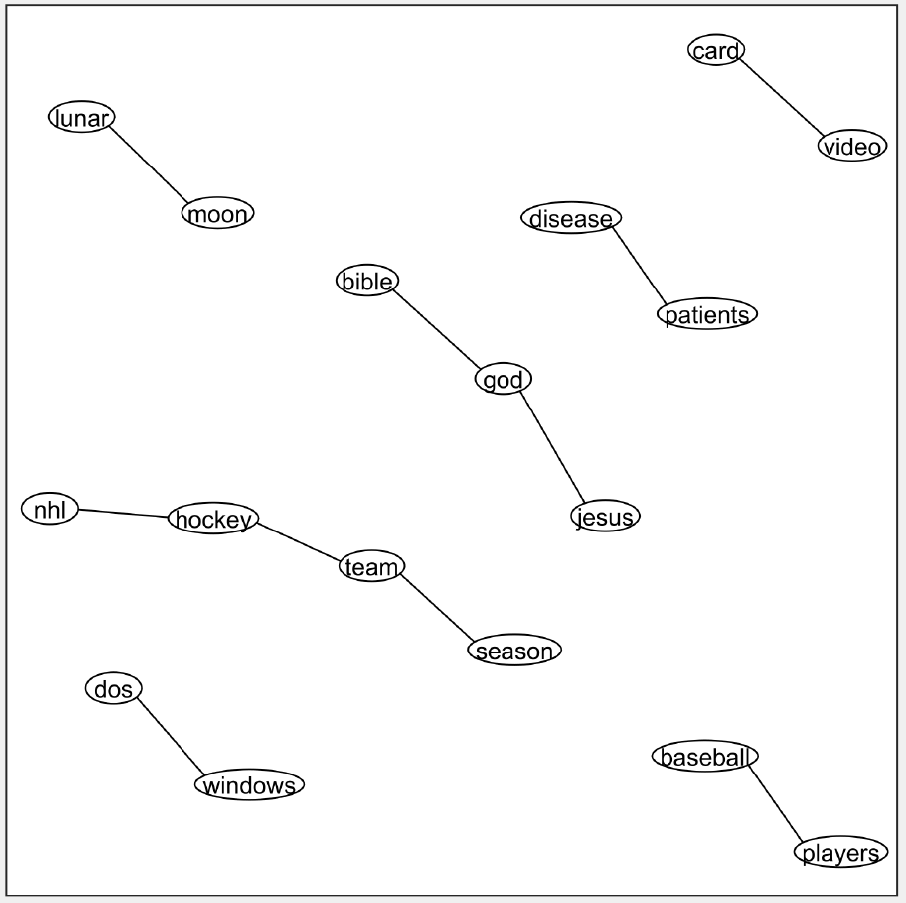}\vspace{-6pt}\caption{\scriptsize CWOA-L0}
      \end{subfigure}
\caption{A practical example on `20newsgroups' data set. The objective values generated by the four methods are $717.04$, $715.88$, $717.45$, and $714.76$, respectively.}
\label{fig:exp:example}
%\vspace{-10pt}
\end{figure*}

%In additon, we also illustrate the convergence of our algorithm by showing the convergence curve.
\section{Experiments}\label{expr}

This section demonstrates the performance of the proposed Coordinated-Wise Optimization Algorithm (CWOA) on synthetic and real-world data sets. All codes are implemented in MATLAB on an Intel 3.20GHz CPU with 8 GB RAM. Some Matlab code can be found in the authors' research webpages.

%We will publish our codes in the authors' research webpages after the acceptance of this paper. % For further evaluation, we provide our code in the \bbb{supplementary material}.

%For further evaluation, we provide our code in the \bbb{supplementary material}.

$\bullet$ \bbb{Data sets}. Three types of data sets are considered in our experiments. (i) Gaussian random data sets. We generate a data matrix $\bbb{Z}\in \mathbb{R}^{m\times n}$ sampled from a standard normal distribution. The parameter $m$ is fixed to $500$. The covariance matrix $\bbb{\Sigma}$ is computed by $\bbb{\Sigma} = 1/(m-1) \sum_{i=1}^m (\bbb{z_i} - \bbb{\mu})(\bbb{z_i} - \bbb{\mu})^T$, where $\bbb{z_i}$ denotes $i^{th}$ column of $\bbb{Z}$ and $\bbb{\mu} = \frac{1}{m} \sum_{i=1}^m\bbb{z_i}\in \mathbb{R}^{n\times 1}$. We consider different values for $n \in\{500,1000,1500,2000\}$ and denote the data sets as `Gaussian-Random-n'. (ii) Sparse-structured data sets. We generate the synthetic data in a similar manner as described in \cite{lu2013}. Roughly speaking, we first generate a true inverse covariance matrix $\bbb{X}^*$ which only contains $p$ non-zero entries with fixing $p=500$. Then we inject Gaussian noise to $\bbb{X}^*$ to obtain a noisy covariance matrix $\bbb{\Sigma}$. We consider different values for $n \in\{500,1000,1500,2000\}$ and denote the data sets as 'Sparse-Structure-n'. (iii) Real-world data sets. We use four well-known real-world data sets \{`isolet', `mnist', `usps', `w1a'\} in our experiments, all of which can be download in the LIBSVM website \footnote{\url{https://www.csie.ntu.edu.tw/~cjlin/libsvmtools/datasets}}. The size of the data sets are $7797\times617$, $10000\times 780$, $9298\times256$ and $49749\times 300$, respectively. We construct $\bbb{\Sigma}$ from the data sets using the same strategy as in Gaussian random data sets.

$\bullet$ \bbb{Compared methods}. We compared the following methods. \bbb{(a)} QUIC applies a Newton-like method \cite{HsiehSDR14} to solve the convex $\ell_1$ regularized problem \footnote{Code: \url{http://www.cs.utexas.edu/~sustik/QUIC}}. Since this method cannot control the sparsity of the solution, we solve the convex problem where the regulation parameter is swept over $2^{\{-10,-9, ..., 10\}}$. Finally, the solution that leads to smallest objective value after a hard thresholding projection (which reduces to setting the small values of the solution in magnitude to 0) is selected. We use the default stopping criterion for QUIC-L1. \bbb{(b)} ADMM directly applies alternating direction method of multipliers to solve the non-convex $\ell_0$ norm problem in(\ref{alg:main}). \bbb{(c)} Penalty Decomposition Algorithm (PDA) \cite{lu2013} decomposes the $\ell_0$ norm problem into a sequence of penalty subproblems which are solved by a block coordinate descent method \footnote{Code: \url{http://people.math.sfu.ca/~zhaosong}}. \bbb{(d)} CWOA is proposed in this paper to solve the original $\ell_0$ norm problem. The four methods above are denoted as QUIC-L1, ADMM-L0, PDA-L0, and CWOA-L0, respectively. We vary the parameter $s$ with the range $\{30,70,110,150,190,230,270\}$.

We remark that this paper pays more attention to the solution quality of the non-convex optimization problem in (\ref{eq:main}). When $\ell_1$ convex relaxation is considered, the resulting problem is strongly convex and existing convex methods will exactly lead to the same unique solution. Therefore, we only select QUIC as the representative of convex methods for comparision.% Hsieh et al¡¯s Newton-like methods (QUIC and BigQUIC) are expected to achieve similar accuracy with ADMM, although they can be faster.

$\bullet$ \bbb{Quantitative Comparisons}. We demonstrate the accuracy of all methods by comparing their objective values in Figure \ref{fig:exp:fobj}. We also report their responding computational time in Figure \ref{fig:exp:time}.

Several conclusions can be drawn. \bbb{(i)} CWOA-L0 consistently outperforms existing state-of-the-art approaches in all data sets in term of accuracy. In addition, with increasing sparsity level $s$, the gap between our method and others becomes larger in some data sets. \bbb{(ii)} The solutions generated by QUIC-L1 and PDA-L0 may not always satisfy the positive definite constraint and incur much larger objective values. In addition, ADM-L0 seems to present the second best results in terms of accuracy. \bbb{(iii)} While the computational time of the other methods are insensitive to the change in sparsity level $s$, the computational time of our method scales linearly with the sparsity level. This is expected since CWOA-L0 needs to solve the reduced convex problem for at least $s$ times. \bbb{(iv)} The proposed method is slower than the compared method when the $s$ is large. However, the computational time pays off since our method consistently achieves lower objective values.

$\bullet$ \bbb{Convergence Behavior.} First of all, we demonstrate the convergence behavior of Algorithm \ref{alg:main} with different sparsity level $s$ on different data sets in Figure \ref{fig:exp:converge:1}. Since the methods \{QUIC-L1, ADMM-L0, PDA-L0\} may violate positive definite constraint or the sparsity constraint and our method always generates feasible solutions ($\bbb{X} \succ0$ and $\| \bbb{X} \|_{0,\text{off}} \leq s$) in \emph{all} iterations, we do not compare the objective values for different iterations of the algorithms.

%We do not include the comparisons with the other methods here since their intermediate solutions fail to fully satisfy the positive-definiteness constraint and sparsity constraint simultaneously while our method does since it is a feasible direction method.

We make two important observations from these results. (i) The objective value decrease monotonically. This is because Algorithm \ref{alg:main} is a greedy descent algorithm. (ii)We observe from Figure \ref{fig:exp:converge:1} that Algorithm \ref{alg:main} terminates at iteration $\{122,138,116,136\}$. Therefore, Algorithm \ref{alg:main} spent $\{230,270,230,270\}/2$ iterations and $\{122,138,116,136\} - \{230,270,230,270\}/2 = \{7,3,1,1\}$ iterations to perform the greedy pursuit stage and swap coordinates stage, respectively.

Secondly, we demonstrate the convergence behavior of Algorithm \ref{alg:newtoncg} on Real-World-mnist data set with different $s$ for different iterations $t=\{20,30,50,70\}$ of Algorithm \ref{alg:main} in Figure \ref{fig:exp:converge:2}. We make two important observations from these results. (i) The objective value decreases monotonically. (ii) The objective values stabilize after the 10th iteration, which means that our algorithm has converged, and the decrease of the objective value is negligible after the 10th iteration. This implies that one may use a looser stopping criterion without sacrificing accuracy.

%
%It should be noted that in Unknown-500, the objective value drops down rapidly
%as a new index is added to the support set and the algortithm terminates as soon as the size of support set meet the
%constraint. We can see from the figure \ref{unknown-500} that the curve is nearly a straight line.

$\bullet$ \bbb{A Practical Example.} We consider different methods to solve the sparse inverse covariance selection problem on a processed version of the 20 newsgroups data set \footnote{\url{http://www.cs.toronto.edu/~roweis/data/20news_w100.mat}}. We expect to obtain a small relation graph with 10 edges, thus, $k$ is set to 20 in our experiments. Two conclusions can be drawn from Figure \ref{fig:exp:example}. (i) Our method achieves the lowest objective value for solving the sparse inverse covariance selection problem. (ii) The proposed method CWOA-L0 is observed to output strong relation patterns in this example.

%While the other methods QUIC-L1, ADMM-L0 and PDA-L0 generate weak relation patterns `food-msg', `team-season' and `food-msg' respectively, our method

 %While the other methods QUIC-L1, ADMM-L0 and PDA-L0 generate weak relation patterns `food-msg', `team-season' and `food-msg' respectively, our method CWOA-L0 is observed to output stronger relation patterns in this example.

\section{Conclusions}\label{conclude} % and Future works

In this paper, we have developed an effective and efficient coordinate-wise optimization algorithm for solving the non-convex $\ell_0$ norm sparse inverse covariance selection problem. The algorithm is guaranteed to converge to a desirable coordinate-wise minimum point. Extensive experiments have shown that the proposed method \emph{consistently} outperforms existing methods.

%achieves state-of-the-art performance.

% or


%
%\appendices
%% you can choose not to have a title for an appendix
%% if you want by leaving the argument blank
%\section{}
%% use section* for acknowledgment
%\ifCLASSOPTIONcompsoc
%  % The Computer Society usually uses the plural form
%\else
%  % regular IEEE prefers the singular form
%\fi

\ifCLASSOPTIONcaptionsoff
  \newpage
\fi



%\bibliographystyle{IEEEtran}
%\begin{thebibliography}{1}
%H.~Kopka and P.~W. Daly, \emph{A Guide to {\LaTeX}}, 3rd~ed.\hskip 1em plus
%\end{thebibliography}

\bibliographystyle{plain}
\bibliography{myref}


%\end{IEEEbiography}

%% insert where needed to balance the two columns on the last page with
%%\newpage


%\vfill


\end{document}